\documentclass[11pt,twoside,a4paper]{article}
\usepackage{amsfonts,amssymb,amsmath,amsthm,mathrsfs,float,multicol,epsfig,setspace,indentfirst,breqn,cite,graphicx,subfigure,verbatim}
\newtheorem{Theorem}{Theorem}[section]
\newtheorem{Lemma}{Lemma}[section]
\newtheorem{Remark}{Remark}[section]
\newtheorem{Definition}{Definition}[section]
\numberwithin{equation}{section}

\begin{document}

\title{Effect of leakage delay on Hopf bifurcation in a fractional BAM neural network}
\author{Jiazhe Lin$^a$,\quad Rui Xu$^{b,c,}$\thanks{Author for Correspondence, e-mail: xur2020meu@163.com, rxu88@163.com.},\quad Liangchen Li$^a$,\quad Xiaohong Tian$^{b,c}$\\
{\small a Institute of Applied Mathematics, Army Engineering University}\\
{\small Shijiazhuang 050003, Hebei, P.R. China}\\
{\small b Complex Systems Research Center, Shanxi University}\\
{\small Taiyuan 030006, Shanxi, P.R. China}\\
{\small c Shanxi Key Laboratory of Mathematical Techniques and Big Data Analysis on}\\
{\small Disease Control and Prevention, Shanxi University}\\
{\small Taiyuan 030006, Shanxi, P.R. China}\\}
\date{}
\maketitle
\noindent
\begin{minipage}{140mm}
{\bf Abstract:} Recently, the influence of leakage delay on the dynamics of integer-order neural networks has been investigated extensively. It has been confirmed that fractional calculus can depict the memory and hereditary attributes of neural networks more accurately. In this paper, we study the existence of Hopf bifurcation in a six-neuron fractional bidirectional associative memory (BAM) neural network with leakage delay. By selecting two appropriate bifurcation parameters and analyzing corresponding characteristic equations, it is verified that the delayed fractional neural network generates a Hopf bifurcation when the bifurcation parameters pass through some critical values. In order to measure how much impact of leakage delay on Hopf bifurcation, sensitivity analysis methods, such as scatter plots and partial rank correlation coefficients (PRCCs), are introduced to assess the sensitivity of bifurcation amplitudes to leakage delay. Numerical examples are carried out to illustrate the theoretical results and help us gain an insight into the effect of leakage delay vividly.  \\
{\bf Keywords:} Leakage delay; Hopf bifurcation; fractional order; BAM neural network; sensitivity analysis.
\end{minipage}

\section{Introduction}

In the past few decades, BAM neural networks as well as their various generalizations have attracted the attention of many researchers due to their potential applications in parallel computation, associative memory and nonlinear optimization problems (see, e.g. \cite{ZPZ,XTL,BZ,L,LXZ,XZ,WYX,YYC}, and the references cited therein). The BAM neural networks are originally proposed by Kosko \cite{K}, which are composed of neurons arranged in two layers and described by the following system
\begin{equation}\label{101}
\left\{
\begin{aligned}
&{{\dot x}_i}(t) =  - {a_i}{x_i}(t) + \sum\nolimits_{j = 1}^p {{u_{ji}}} {f_i}({y_j}(t)) + {I_i},\quad i = 1,2,\cdots,q, \\
&{{\dot y}_j}(t) =  - {b_j}{y_j}(t) + \sum\nolimits_{i = 1}^q {{v_{ij}}} {f_j}({x_i}(t)) + {I_j},\quad j = 1,2,\cdots,p,
\end{aligned}
\right.
\end{equation}
where $u_{ji}$, $v_{ij}$ are the connection weights at the time $t$ through the neurons in two layers: the I-layer and the J-layer, $a_i$ and $b_j$ describe the stability of internal neuron processes on the I-layer and the J-layer, respectively. On the I-layer, the neurons whose states are denoted by $x_i(t)$ receive the inputs $I_i$ and the inputs outputted by
those neurons in the J-layer via activation functions $f_i$, while on the J-layer, the neurons whose associated states denoted by $y_j(t)$ receive the inputs $I_j$ and the inputs outputted by those neurons in the I-layer via activation functions $f_j$.

Obviously, the influence of time delays on the dynamics of BAM neural networks has been not taken into account in system \eqref{101}. Due to the finite speed of signal transmission and amplifiers switching, time delay inevitably exists in neural networks. On one hand, time delays are harmful to the dynamical behavior of considered neural networks, causing oscillation, even chaos. For instance, Xu et al. \cite{XTL} investigated a six-neuron BAM neural network model with communication delays and illustrated that the model undergoes a Hopf bifurcation in some cases. On the other hand, a proper time delay is advantageous for the dynamical behavior of neural networks, such as, it can improve the stability performance of networks (see, e.g. \cite{XHC,XZC,ZW}). Meanwhile, the leakage term in each of the right sides of \eqref{101} corresponds to a stabilizing negative feedback of the system which acts instantaneously without time delay. In practice, time is required to isolate the static state. Gopalsamy \cite{C} found that time delays in leakage terms have an important influence on the dynamic behavior, and the stability of the neural networks was determined by leakage delays. The mathematical model is described by
\begin{equation}\label{102}
\left\{
\begin{aligned}
&{{\dot x}_i}(t) =  - {a_i}{x_i}(t - {\tau _i}) + \sum\nolimits_{j = 1}^p {{u_{ji}}} {f_i}({y_j}(t - {\tau _{ji}})) + {I_i},\quad i = 1,2, \cdots ,q, \\
&{{\dot y}_j}(t) =  - {b_j}{y_j}(t - {\tau _j}) + \sum\nolimits_{i = 1}^q {{v_{ij}}} {f_j}({x_i}(t - {\sigma _{ij}})) + {I_j},\quad j = 1,2, \cdots ,p,
\end{aligned}
\right.
\end{equation}
where $\tau_i$ and $\tau_j$ denote the leakage delay, respectively; $\tau_{ji}$ and $\sigma_{ij}$ represent the communication delay, respectively. Similarly, Tian and Xu \cite{TX} illustrated that time delays in the stabilizing negative feedback terms will have a tendency to destabilize a system.

In recent years, fraction-order derivatives provide an excellent tool for the description of memory and hereditary properties of various materials and processes. In general, plenty of practical objects can be described clearly by the fractional differential equations, due to their more degrees of freedom and infinite memory. Hence, the research of dynamical analysis of fractional neural networks has gained a lot of attention and some valuable results have been referred to \cite{HMC,HCX,HC,TXS,HCX2}. As we know, the integer-order calculus can only determine the local features of neural networks, while the fractional calculus can depict the memory and hereditary attributes more accurately \cite{A,LHS}.

In \cite{HMC}, Huang et al. studied the stability and bifurcation of four-neuron fractional BAM neural networks with time delay in leakage terms and illustrated that the leakage delay has a destabilizing influence on the stability performance, which cannot be ignored. Furthermore, Huang and Cao \cite{HC} concentrated on the issue of bifurcation analysis for high-order fractional BAM neural networks involving leakage delay and obtain similar bifurcation analysis results. In \cite{TXS}, Tao et al. proposed a fractional two-gene regulatory network model with delays and found that the delayed fractional genetic network can generate a Hopf bifurcation when the total delay passes through some critical values.

However, it is noteworthy that some of above results about stability and bifurcation for the fractional systems are defective on the analysis methods or relatively few neurons in both two layers in BAM neural networks. When the number of neurons in BAM neural network \eqref{102} is larger, the simplified system \eqref{201} can reflect the really large neural networks more closely. Motivated by Xu et al. \cite{XTL} and Huang et al. \cite{HMC,HC}, the primary objective of this paper is to study the existence of Hopf bifurcation for the following six-neuron fractional BAM neural network involving leakage delay
\begin{small}
\begin{equation}\label{201}
\left\{
\begin{aligned}
&{D^\theta }{x_1}(t) =  - {k_1}{x_1}(t - {\tau _1}) + {m_{11}}{f_{11}}({y_1}(t - {\tau _2})) + {m_{12}}{f_{12}}({y_2}(t - {\tau _2})) + {m_{13}}{f_{13}}({y_3}(t - {\tau _2})), \\
&{D^\theta }{x_2}(t) =  - {k_2}{x_2}(t - {\tau _1}) + {m_{21}}{f_{21}}({y_1}(t - {\tau _2})) + {m_{22}}{f_{22}}({y_2}(t - {\tau _2})) + {m_{23}}{f_{23}}({y_3}(t - {\tau _2})), \\
&{D^\theta }{x_3}(t) =  - {k_3}{x_3}(t - {\tau _1}) + {m_{31}}{f_{31}}({y_1}(t - {\tau _2})) + {m_{32}}{f_{32}}({y_2}(t - {\tau _2})) + {m_{33}}{f_{33}}({y_3}(t - {\tau _2})), \\
&{D^\theta }{y_1}(t) =  - {k_4}{y_1}(t - {\tau _1}) + {n_{11}}{g_{11}}({x_1}(t - {\tau _2})) + {n_{12}}{g_{12}}({x_2}(t - {\tau _2})) + {n_{13}}{g_{13}}({x_3}(t - {\tau _2})), \\
&{D^\theta }{y_2}(t) =  - {k_5}{y_2}(t - {\tau _1}) + {n_{21}}{g_{21}}({x_1}(t - {\tau _2})) + {n_{22}}{g_{22}}({x_2}(t - {\tau _2})) + {n_{23}}{g_{23}}({x_3}(t - {\tau _2})), \\
&{D^\theta }{y_3}(t) =  - {k_6}{y_3}(t - {\tau _1}) + {n_{31}}{g_{31}}({x_1}(t - {\tau _2})) + {n_{32}}{g_{32}}({x_2}(t - {\tau _2})) + {n_{33}}{g_{33}}({x_3}(t - {\tau _2})),
\end{aligned}
\right.
\end{equation}
\end{small}
where $\theta \in (0,1]$ is the fractional order, $x_i(t)$, $y_i(t)$ $(i=1,2,3)$ stand for state variables of the $i$th neuron in I-layer and J-layer, respectively; $k_l>0$ $(l=1,2,3,4,5,6)$ describes the stability of internal neuron processes on the I-layer and the J-layer, respectively; $m_{ij}$, $n_{ij}$ $(i,j=1,2,3)$ represent connection weights; $f_{ij}(\cdot)$, $g_{ij}(\cdot)$ denote activation functions; $\tau_1$ is leakage delay, $\tau_2$ is communication delay.

Above discussion may raise the question that, for a given fractional BAM neural network, which one contributes more to the influence of dynamical behavior, leakage delay or communication delay? In order to assess how much impact of leakage delay on Hopf bifurcation, we introduce sensitivity analysis into this paper. Sensitivity analysis (SA) assesses how variations in model outputs can be apportioned, qualitatively or quantitatively, to different input sources \cite{S,SCS}. Through the process of recalculating outcomes under alternative assumptions, we can determine the impact of a variable by scatter plots and PRCCs. Sensitivity analysis provides an evaluation of how much each input is contributing to the output uncertainty and performs the role of ordering by importance, namely, by the strength and relevance of the inputs in determining the variations in the output \cite{SR}.

The main contributions of this paper can be summarized as follows:

$\bullet$ The problem of stability and bifurcation in a six-neuron fractional BAM neural network with leakage delay is discussed by introducing two appropriate bifurcation parameters, and the existence of Hopf bifurcation for the proposed network is established.

$\bullet$ The impact of the order on the critical frequencies and bifurcation points is further demonstrated numerically for the fractional BAM neural network.

$\bullet$ Sensitivity analysis is introduced into bifurcation analysis of the fractional BAM neural networks, which helps us to estimate whether leakage delay or communication delay contributes more to the influence of dynamical behavior for a given fractional BAM neural network.

The outline of the paper is listed as follows: In Section 2, the stability of the equilibrium and the existence of Hopf bifurcation are studied. In Section 3, numerical simulations are carried out to illustrate the validity of the main results. Sensitivity analysis of leakage delay and communication delay is performed in Section 4. A brief remark is given in Section 5 to conclude this work.

\section{Main results for bifurcation analysis}

In this section, we introduce two appropriate bifurcation parameters. Through analyzing corresponding characteristic equations, we obtain the local stability of the equilibrium and the existence of Hopf bifurcation to system \eqref{201}. Before bifurcation analysis, we address the following assumption.

{\bf (H1)} $f_{ij},g_{ij} \in C(R,R)$, $f_{ij}(0)=g_{ij}(0)=0$, $xf_{ij}(x)>0$, $xg_{ij}(x)>0$ $(i,j=1,2,3)$ for $x\neq0$.

There are several definitions of fractional derivatives. Riemann-Liouville definition and Caputo definition are commonly used. Since Caputo derivative only requires initial conditions given by means of integer-order derivative, representing well-understood features of physical situation and making it more applicable to real world problems \cite{HMC,HC}. Hence, the definition of Caputo derivative is adopted in this paper.

\begin{Definition}\label{def202}
{\rm \cite{P}} For a function $f(t) \in C^n([t_0,\infty),R)$, Caputo fractional derivative of order $\theta$ is defined by
\[{D^\theta }f(t) = \frac{1}{{\Gamma (n - \theta )}}\int_{{t_0}}^t{\frac{{{f^{(n)}}(s)}}{{{{(t - s)}^{\theta  - n + 1}}}}ds} ,\]
where $t\ge {t_0}$, and $n$ is a positive integer such that $n-1\le\theta<n$.

Moreover, when $0<\theta<1$,
\[{D^\theta }f(t) = \frac{1}{{\Gamma (1 - \theta )}}\int_{{t_0}}^t {\frac{{f'(s)}}{{{{(t - s)}^\theta }}}ds}.\]
\end{Definition}

It is obvious that the origin is an equilibrium point of system \eqref{201} under Assumption {\bf (H1)}. Linearizing system \eqref{201} at the origin, we obtain that
\begin{equation}\label{301}
\left\{
\begin{aligned}
&{D^\theta }{x_1}(t) =  - {k_1}{x_1}(t - {\tau _1}) + {\phi _{11}}{y_1}(t - {\tau _2}) + {\phi _{12}}{y_2}(t - {\tau _2}) + {\phi _{13}}{y_3}(t - {\tau _2}), \\
&{D^\theta }{x_2}(t) =  - {k_2}{x_2}(t - {\tau _1}) + {\phi _{21}}{y_1}(t - {\tau _2}) + {\phi _{22}}{y_2}(t - {\tau _2}) + {\phi _{23}}{y_3}(t - {\tau _2}), \\
&{D^\theta }{x_3}(t) =  - {k_3}{x_3}(t - {\tau _1}) + {\phi _{31}}{y_1}(t - {\tau _2}) + {\phi _{32}}{y_2}(t - {\tau _2}) + {\phi _{33}}{y_3}(t - {\tau _2}), \\
&{D^\theta }{y_1}(t) =  - {k_4}{y_1}(t - {\tau _1}) + {\varphi _{11}}{x_1}(t - {\tau _2}) + {\varphi _{12}}{x_2}(t - {\tau _2}) + {\varphi _{13}}{x_3}(t - {\tau _2}), \\
&{D^\theta }{y_2}(t) =  - {k_5}{y_2}(t - {\tau _1}) + {\varphi _{21}}{x_1}(t - {\tau _2}) + {\varphi _{22}}{x_2}(t - {\tau _2}) + {\varphi _{23}}{x_3}(t - {\tau _2}), \\
&{D^\theta }{y_3}(t) =  - {k_6}{y_3}(t - {\tau _1}) + {\varphi _{31}}{x_1}(t - {\tau _2}) + {\varphi _{32}}{x_2}(t - {\tau _2}) + {\varphi _{33}}{x_3}(t - {\tau _2}),
\end{aligned}
\right.
\end{equation}
where ${\phi _{ij}} = {m_{ij}}{f'_{ij}}(0)$, ${\varphi _{ij}} = {n_{ij}}{g'_{ij}}(0)$ $(i,j=1,2,3)$. By applying Laplace transformation, the characteristic matrix of system \eqref{301} is
\begin{small}
\begin{align*}
\det \left|
\begin{array}{*{20}{c}}
   {{\lambda ^\theta } + {k_1}{e^{ - \lambda{\tau _1}}}} & 0 & 0 & { - {\phi _{11}}{e^{ - \lambda{\tau _2}}}} & { - {\phi _{12}}{e^{ - \lambda{\tau _2}}}} & { - {\phi _{13}}{e^{ - \lambda{\tau _2}}}}  \\
   0 & {{\lambda ^\theta } + {k_2}{e^{ - \lambda{\tau _1}}}} & 0 & { - {\phi _{21}}{e^{ - \lambda{\tau _2}}}} & { - {\phi _{22}}{e^{ - \lambda{\tau _2}}}} & { - {\phi _{23}}{e^{ - \lambda{\tau _2}}}}  \\
   0 & 0 & {{\lambda ^\theta } + {k_3}{e^{ - \lambda{\tau _1}}}} & { - {\phi _{31}}{e^{ - \lambda{\tau _2}}}} & { - {\phi _{32}}{e^{ - \lambda{\tau _2}}}} & { - {\phi _{33}}{e^{ - \lambda{\tau _2}}}}  \\
   { - {\varphi _{11}}{e^{ - \lambda{\tau _2}}}} & { - {\varphi _{12}}{e^{ - \lambda{\tau _2}}}} & { - {\varphi _{13}}{e^{ - \lambda{\tau _2}}}} & {{\lambda ^\theta } + {k_4}{e^{ - \lambda{\tau _1}}}} & 0 & 0  \\
   { - {\varphi _{21}}{e^{ - \lambda{\tau _2}}}} & { - {\varphi _{22}}{e^{ - \lambda{\tau _2}}}} & { - {\varphi _{23}}{e^{ - \lambda{\tau _2}}}} & 0 & {{\lambda ^\theta } + {k_5}{e^{ - \lambda{\tau _1}}}} & 0  \\
   { - {\varphi _{31}}{e^{ - \lambda{\tau _2}}}} & { - {\varphi _{32}}{e^{ - \lambda{\tau _2}}}} & { - {\varphi _{33}}{e^{ - \lambda{\tau _2}}}} & 0 & 0 & {{\lambda ^\theta } + {k_6}{e^{ - \lambda{\tau _1}}}}  \\
\end{array}
\right| = 0.
\end{align*}
\end{small}
Thus, we obtain the characteristic equation of system \eqref{301} as follows
\begin{equation}\label{302}
\begin{aligned}
{\lambda ^{6\theta }} &+ {c_{11}}{e^{ - \lambda{\tau _1}}}{\lambda ^{5\theta }} + \left( {{c_{21}}{e^{ - 2\lambda{\tau _1}}} - {c_{22}}{e^{ - 2\lambda{\tau _2}}}} \right){\lambda ^{4\theta }} + \left( {{c_{31}}{e^{ - 3\lambda{\tau _1}}} - {c_{32}}{e^{ - \lambda{\tau _1}}}{e^{ - 2\lambda{\tau _2}}}} \right){\lambda ^{3\theta }} \\
&+ \left[ {{c_{41}}{e^{ - 4\lambda{\tau _1}}} - {c_{42}}{e^{ - 2\lambda{\tau _1}}}{e^{ - 2\lambda{\tau _2}}} + \left( {{c_{43}} - {c_{44}}} \right){e^{ - 4\lambda{\tau _2}}}} \right]{\lambda ^{2\theta }} \\
&+ \left[ {{c_{51}}{e^{ - 5\lambda{\tau _1}}} - {c_{52}}{e^{ - 3\lambda{\tau _1}}}{e^{ - 2\lambda{\tau _2}}} + \left( {{c_{53}} - {c_{54}}} \right){e^{ - \lambda{\tau _1}}}{e^{ - 4\lambda{\tau _2}}}} \right]{\lambda ^\theta } \\
&+ \left[ {{c_{61}}{e^{ - 6\lambda{\tau _1}}} - {c_{62}}{e^{ - 4\lambda{\tau _1}}}{e^{ - 2\lambda{\tau _2}}} + \left( {{c_{63}} - {c_{64}}} \right){e^{ - 2\lambda{\tau _1}}}{e^{ - 4\lambda{\tau _2}}} - {c_{65}}{e^{ - 6\lambda{\tau _2}}}} \right] = 0,
\end{aligned}
\end{equation}
where  positive constants $c_{11}$, $c_{21}$, $c_{22}$, $c_{31}$, $c_{32}$, $c_{41}$, $c_{42}$, $c_{43}$, $c_{44}$, $c_{51}$, $c_{52}$, $c_{53}$, $c_{54}$, $c_{61}$, $c_{62}$, $c_{63}$, $c_{64}$, $c_{65}$ are defined in Appendix A.

Define ${\tau _3} = \left( {{\tau _1} + {\tau _2}} \right)/2, {\tau _4} = \left( {{\tau _1} - {\tau _2}} \right)/2$, which are selected as bifurcation parameters. If $\tau_3=0$ and $\tau_4=0$, Eq. \eqref{302} is changed into
\[{({{\lambda ^\theta }})^6} + {d_1}{({{\lambda ^\theta }})^5} + {d_2}{({{\lambda ^\theta }})^4} + {d_3}{({{\lambda ^\theta }})^3} + {d_4}{({{\lambda ^\theta }})^2} + {d_5}{\lambda ^\theta } + {d_6} = 0,\]
where
\begin{align*}
&{d_1} = {c_{11}},\quad {d_2} = {c_{21}} - {c_{22}},\quad {d_3} = {c_{31}} - {c_{32}},\quad{d_4} = {c_{41}} - {c_{42}} + {c_{43}} - {c_{44}},  \\
&{d_5} = {c_{51}} - {c_{52}} + {c_{53}} - {c_{54}},\quad {d_6} = {c_{61}} - {c_{62}} + {c_{63}} - {c_{64}} - {c_{65}}.
\end{align*}
The following lemma about the stability of fractional autonomous systems is listed for later analysis.
\begin{Lemma}\label{lem301}
{\rm \cite{M,DLL}} The following autonomous system
\[{D^\theta }x = Jx,\quad x(0) = {x_0},\]
where $0<\theta<1$, $x\in R^n$, $J\in R^{n\times n}$, this system is asymptotically stable if and only if $\left| {\arg ({\lambda _i})} \right| > \theta \pi /2$ $(i = 1,2,\cdots,n)$. In this case, each component of the states decays towards $0$ like $t^{-\theta}$. Also, this system is stable if and only if $\left| {\arg ({\lambda _i})} \right| \ge \theta \pi /2$ and those critical eigenvalues that satisfy $\left| {\arg ({\lambda _i})} \right|=\theta \pi /2$ have geometric multiplicity one.
\end{Lemma}

\begin{Theorem}\label{thm301}
System \eqref{201} is asymptotically stable when $\tau_3=\tau_4=0$ and $D_i>0$ $(i=1,2,3,4,5,6)$ hold, where $D_i$ is defined as follows
\end{Theorem}
\[\begin{array}{l}
 {D_1} = {d_1},\quad {D_2} = \det \left| {\begin{array}{*{20}{c}}
   {{d_1}} & 1  \\
   {{d_3}} & {{d_2}}  \\
\end{array}} \right|,\quad {D_3} = \det \left| {\begin{array}{*{20}{c}}
   {{d_1}} & 1 & 0  \\
   {{d_3}} & {{d_2}} & {{d_1}}  \\
   {{d_5}} & {{d_4}} & {{d_3}}  \\
\end{array}} \right|, \\
 {D_4} = \det \left| {\begin{array}{*{20}{c}}
   {{d_1}} & 1 & 0 & 0  \\
   {{d_3}} & {{d_2}} & {{d_1}} & 1  \\
   {{d_5}} & {{d_4}} & {{d_3}} & {{d_2}}  \\
   0 & {{d_6}} & {{d_5}} & {{d_4}}  \\
\end{array}} \right|,\quad {D_5} = \det \left| {\begin{array}{*{20}{c}}
   {{d_1}} & 1 & 0 & 0 & 0  \\
   {{d_3}} & {{d_2}} & {{d_1}} & 1 & 0  \\
   {{d_5}} & {{d_4}} & {{d_3}} & {{d_2}} & {{d_1}}  \\
   0 & {{d_6}} & {{d_5}} & {{d_4}} & {{d_3}}  \\
   0 & 0 & 0 & {{d_6}} & {{d_5}}  \\
\end{array}} \right|,\quad {D_6} = {d_6}{D_5}. \\
 \end{array}\]
\begin{proof}
If $D_i>0$ $(i=1,2,3,4,5,6)$ holds, it follows that all the roots $\lambda_i$ satisfy $\left| {\arg ({\lambda _i})} \right| > \theta \pi /2$ $(i = 1,2,3,4,5,6)$. According to Lemma \ref{lem301}, we can easily conclude that system \eqref{201} is asymptotically stable when $\tau_3=\tau_4=0$.
\end{proof}

\subsection{Hopf bifurcation with respect to $\tau_3$}

If $\tau_3\neq0$ and $\tau_4=0$, i.e., $\tau_3=\tau_1=\tau_2$, Eq. \eqref{302} becomes
\begin{equation}\label{303}
\begin{aligned}
{\lambda ^{6\theta }} &+ {c_{11}}{e^{ - \lambda{\tau _3}}}{\lambda ^{5\theta }} + \left( {{c_{21}} - {c_{22}}} \right){e^{ - 2\lambda{\tau _3}}}{\lambda ^{4\theta }} + \left( {{c_{31}} - {c_{32}}} \right){e^{ - 3\lambda{\tau _3}}}{\lambda ^{3\theta }} \\
&+ \left( {{c_{41}} + {c_{43}} - {c_{42}} - {c_{44}}} \right){e^{ - 4\lambda{\tau _3}}}{\lambda ^{2\theta }} + \left( {{c_{51}} + {c_{53}} - {c_{52}} - {c_{54}}} \right){e^{ - 5\lambda{\tau _3}}}{\lambda ^\theta } \\
&+ \left( {{c_{61}} + {c_{63}} - {c_{62}} - {c_{64}} - {c_{65}}} \right){e^{ - 6\lambda{\tau _3}}} = 0.
\end{aligned}
\end{equation}
Multiplying $e^{6\lambda\tau_3}$ on both sides of Eq. \eqref{303} and denoting $s=e^{\lambda\tau_3}\lambda^\theta$, it follows that
\begin{equation}\label{304}
\begin{aligned}
{s^6} &+ {c_{11}}{s^5} + \left( {{c_{21}} - {c_{22}}} \right){s^4} + \left( {{c_{31}} - {c_{32}}} \right){s^3} + \left( {{c_{41}} + {c_{43}} - {c_{42}} - {c_{44}}} \right){s^2} \\
&+ \left( {{c_{51}} + {c_{53}} - {c_{52}} - {c_{54}}} \right)s + \left( {{c_{61}} + {c_{63}} - {c_{62}} - {c_{64}} - {c_{65}}} \right) = 0.
\end{aligned}
\end{equation}
Define the six roots of Eq. \eqref{304} as $s_n = {R_n} + i{I_n}$ $(n = 1,2,3,4,5,6)$, where ${R_n}$ and ${I_n}$ are the real and imaginary parts of $s_n$, respectively. Note that
\begin{equation}\label{305}
{e^{\lambda {\tau _3}}}{\lambda ^\theta } = {s_n}.
\end{equation}
To find possible periodic solutions, which may bifurcate from a Hopf bifurcation point, let $\lambda=i\omega$ be a root of \eqref{305}. Substituting $\lambda=i\omega=\omega(cos\frac{\pi}{2}+isin\frac{\pi}{2})$ into \eqref{305} and separating the real and imaginary parts yields
\begin{equation}\label{306}
\left\{
\begin{aligned}
&{\omega ^\theta }cos\frac{{\theta \pi }}{2}\cos \omega {\tau _3} - {\omega ^\theta }sin\frac{{\theta \pi }}{2}\sin \omega {\tau _3} = {R_n}, \\
&{\omega ^\theta }cos\frac{{\theta \pi }}{2}\sin \omega {\tau _3} + {\omega ^\theta }sin\frac{{\theta \pi }}{2}\cos \omega {\tau _3} = {I_n}.
\end{aligned}
\right.
\end{equation}
From \eqref{306}, it follows that
\[\cos \omega {\tau _3} = \frac{{{R_n}\cos \frac{{\theta \pi }}{2} + {I_n}\sin \frac{{\theta \pi }}{2}}}{{{\omega ^\theta }}},\quad \sin \omega {\tau _3} = \frac{{{I_n}\cos \frac{{\theta \pi }}{2} - {R_n}\sin \frac{{\theta \pi }}{2}}}{{{\omega ^\theta }}}.\]
Noting that ${\sin ^2}\omega {\tau _3} + {\cos ^2}\omega {\tau _3} = 1$, we obtain that
\begin{subequations}\label{307}
\begin{align}
&\omega  = \sqrt[2\theta ]{{{R_n}^2 + {I_n}^2}},\\
&{\tau ^{(k)}} = \frac{1}{\omega }\left[ {\arccos \left( {\frac{{{R_n}\cos \frac{{\theta \pi }}{2} + {I_n}\sin \frac{{\theta \pi }}{2}}}{{{\omega ^\theta }}}} \right) + 2k\pi } \right],\quad k = 0,1,2,\cdots.
\end{align}
\end{subequations}
Define the bifurcation point of system \eqref{201} as follows:
\[{\tau _0} = \min \{ {\tau ^{(k)}}\} ,\quad k = 0,1,2,\cdots.\]
To establish the main results of this section, we make the following assumptions.

\noindent{\bf (H2)} Eq. (\ref{307}a) has no positive real root.

\noindent{\bf (H3)} Eq. (\ref{307}a) has at least one positive real root.

\noindent{\bf (H4)} $\left({{\Phi _1}{\Psi _1} + {\Phi _2}{\Psi _2}}\right)/\left({{\Psi _1}^2 + {\Psi _2}^2}\right)\neq0$, where ${\Phi _i}$, ${\Psi _i}$ $(i=1,2)$ are defined in Appendix B.

\begin{Lemma}\label{lem302}
Let $\lambda (\tau_3) = \mu (\tau_3) + i\omega (\tau_3)$, be the root of Eq. \eqref{303} near $\tau_3=\tau_0$ satisfying $\mu(\tau_0)=0$, $\omega(\tau_0)=\omega_0$, then the following transversality condition holds
\[{\mathop{\rm Re}\nolimits} \left[ {\frac{{d\lambda }}{{d\tau_3}}} \right]\Big| {_{\tau_3 = {\tau _0}}}  \ne 0.\]
\end{Lemma}
\begin{proof}
Based on implicit function theorem, we calculate the derivative of Eq. \eqref{303} with respect to $\tau_3$ as follows
\begin{align*}
6\theta {\lambda ^{6\theta  - 1}}\frac{{d\lambda }}{{d{\tau _3}}} &+ {c_{11}}\left[ {{\lambda ^{5\theta }}{e^{ - \lambda {\tau _3}}}\left( { - {\tau _3}\frac{{d\lambda }}{{d{\tau _3}}} - \lambda } \right) + 5\theta {\lambda ^{5\theta  - 1}}{e^{ - \lambda {\tau _3}}}\frac{{d\lambda }}{{d{\tau _3}}}} \right] \\
&+ \left( {{c_{21}} - {c_{22}}} \right)\left[ {{\lambda ^{4\theta }}{e^{ - 2\lambda {\tau _3}}}\left( { - 2{\tau _3}\frac{{d\lambda }}{{d{\tau _3}}} - 2\lambda } \right) + 4\theta {\lambda ^{4\theta  - 1}}{e^{ - 2\lambda {\tau _3}}}\frac{{d\lambda }}{{d{\tau _3}}}} \right] \\
&+ \left( {{c_{31}} - {c_{32}}} \right)\left[ {{\lambda ^{3\theta }}{e^{ - 3\lambda {\tau _3}}}\left( { - 3{\tau _3}\frac{{d\lambda }}{{d{\tau _3}}} - 3\lambda } \right) + 3\theta {\lambda ^{3\theta  - 1}}{e^{ - 3\lambda {\tau _3}}}\frac{{d\lambda }}{{d{\tau _3}}}} \right] \\
&+ \left( {{c_{41}} + {c_{43}} - {c_{42}} - {c_{44}}} \right)\left[ {{\lambda ^{2\theta }}{e^{ - 4\lambda {\tau _3}}}\left( { - 4{\tau _3}\frac{{d\lambda }}{{d{\tau _3}}} - 4\lambda } \right) + 2\theta {\lambda ^{2\theta  - 1}}{e^{ - 4\lambda {\tau _3}}}\frac{{d\lambda }}{{d{\tau _3}}}} \right] \\
&+ \left( {{c_{51}} + {c_{53}} - {c_{52}} - {c_{54}}} \right)\left[ {{\lambda ^\theta }{e^{ - 5\lambda {\tau _3}}}\left( { - 5{\tau _3}\frac{{d\lambda }}{{d{\tau _3}}} - 5\lambda } \right) + \theta {\lambda ^{\theta  - 1}}{e^{ - 5\lambda {\tau _3}}}\frac{{d\lambda }}{{d{\tau _3}}}} \right] \\
&+ {e^{ - 6\lambda {\tau _3}}}\left( {{c_{61}} + {c_{63}} - {c_{62}} - {c_{64}} - {c_{65}}} \right)\left( { - 6{\tau _3}\frac{{d\lambda }}{{d{\tau _3}}} - 6\lambda } \right) = 0.
\end{align*}
Direct calculation shows that
\[\frac{{d\lambda }}{{d{\tau _3}}} = \frac{{\Phi (s)}}{{\Psi (s)}},\]
where
\begin{align*}
\Phi(s) =& \lambda \left[{c_{11}}{\lambda ^{5\theta }}{e^{ - \lambda {\tau _3}}} + 2\left( {{c_{21}} - {c_{22}}} \right){\lambda ^{4\theta }}{e^{ - 2\lambda {\tau _3}}} + 3\left( {{c_{31}} - {c_{32}}} \right){\lambda ^{3\theta }}{e^{ - 3\lambda {\tau _3}}} \right.\\
&\quad + 4\left( {{c_{41}} + {c_{43}} - {c_{42}} - {c_{44}}} \right){\lambda ^{2\theta }}{e^{ - 4\lambda {\tau _3}}} + 5\left( {{c_{51}} + {c_{53}} - {c_{52}} - {c_{54}}} \right){\lambda ^\theta }{e^{ - 5\lambda {\tau _3}}} \\
&\quad \left.+ 6\left( {{c_{61}} + {c_{63}} - {c_{62}} - {c_{64}} - {c_{65}}} \right){e^{ - 6\lambda {\tau _3}}} \right], \\
\Psi(s) =& 6\theta {\lambda ^{6\theta  - 1}} + {c_{11}}\left( {5\theta {\lambda ^{5\theta  - 1}} - {\tau _3}{\lambda ^{5\theta }}} \right){e^{ - \lambda {\tau _3}}} + \left( {{c_{21}} - {c_{22}}} \right)\left( {4\theta {\lambda ^{4\theta  - 1}} - 2{\tau _3}{\lambda ^{4\theta }}} \right){e^{ - 2\lambda {\tau _3}}} \\
&+ \left( {{c_{31}} - {c_{32}}} \right)\left( {3\theta {\lambda ^{3\theta  - 1}} - 3{\tau _3}{\lambda ^{3\theta }}} \right){e^{ - 3\lambda {\tau _3}}} \\
&+ \left( {{c_{41}} + {c_{43}} - {c_{42}} - {c_{44}}} \right)\left( {2\theta {\lambda ^{2\theta  - 1}} - 4{\tau _3}{\lambda ^{2\theta }}} \right){e^{ - 4\lambda {\tau _3}}} \\
&+ \left( {{c_{51}} + {c_{53}} - {c_{52}} - {c_{54}}} \right)\left( {\theta {\lambda ^{\theta  - 1}} - 5{\tau _3}{\lambda ^\theta }} \right){e^{ - 5\lambda {\tau _3}}} \\
&- 6\left( {{c_{61}} + {c_{63}} - {c_{62}} - {c_{64}} - {c_{65}}} \right){\tau _3}{e^{ - 6\lambda {\tau _3}}}.
\end{align*}
The real part of $d\lambda /d{\tau _3}$ at $\tau_3 = {\tau _0}$ is
\[{\mathop{\rm Re}\nolimits} \left[ {\frac{{d\lambda }}{{d{\tau _3}}}} \right]\Big| {_{\tau_3 = {\tau _0}}}  = \frac{{{\Phi _1}{\Psi _1} + {\Phi _2}{\Psi _2}}}{{{\Psi _1}^2 + {\Psi _2}^2}},\]
where $\Phi_1$ and $\Phi_2$ are the real and imaginary parts of $\Phi(s)$, respectively; $\Psi_1$ and $\Psi_2$ are the real and imaginary parts of $\Psi(s)$, respectively. Based on Assumption {\bf (H4)}, the transversality condition meets. This completes the proof.
\end{proof}
According to Assumptions {\bf (H1)-(H4)}, the following theorem can be derived.
\begin{Theorem}\label{thm302}
For system \eqref{201}, the following results hold

\noindent{\bf (\romannumeral1)} If Assumptions {\bf (H1), (H2)} are satisfied, the zero equilibrium point is global asymptotically stable for $\tau_3 \in [0,\infty)$.

\noindent{\bf (\romannumeral2)} If Assumptions {\bf (H1), (H3), (H4)} and Lemma \ref{lem302} hold,

{\bf a)} The zero equilibrium point is locally asymptotically stable for $\tau_3 \in [0,\tau_0)$;

{\bf b)} system \eqref{201} undergoes a Hopf bifurcation at the origin when $\tau_3=\tau_0$, i.e., it has a branch of periodic solutions bifurcating from the zero equilibrium point near $\tau_3=\tau_0$.
\end{Theorem}

\subsection{Hopf bifurcation with respect to $\tau_4$}

If $\tau_4\neq0$, fix $\tau_1$ and the characteristic equation \eqref{302} becomes
\begin{equation}\label{311}
{e^{ - 4\lambda {\tau _4}}}{p_1}(\lambda ) + {p_2}(\lambda ) + {e^{4\lambda {\tau _4}}}{p_3}(\lambda ) + {e^{8\lambda {\tau _4}}}{p_4} = 0,
\end{equation}
where
\begin{align*}
{p_1}(\lambda ) =& {e^{6\lambda {\tau _1}}}{\lambda ^{6\theta }} + {c_{11}}{e^{5\lambda {\tau _1}}}{\lambda ^{5\theta }} + {c_{21}}{e^{4\lambda {\tau _1}}}{\lambda ^{4\theta }} + {c_{31}}{e^{3\lambda {\tau _1}}}{\lambda ^{3\theta }} \\
&+ {c_{41}}{e^{2\lambda {\tau _1}}}{\lambda ^{2\theta }} + {c_{51}}{e^{\lambda {\tau _1}}}{\lambda ^\theta } + {c_{61}}, \\
{p_2}(\lambda ) =& - \left( {{c_{22}}{e^{4\lambda {\tau _1}}}{\lambda ^{4\theta }} + {c_{32}}{e^{3\lambda {\tau _1}}}{\lambda ^{3\theta }} + {c_{42}}{e^{2\lambda {\tau _1}}}{\lambda ^{2\theta }} + {c_{52}}{e^{\lambda {\tau _1}}}{\lambda ^\theta } + {c_{62}}} \right), \\
{p_3}(\lambda ) =& \left( {{c_{43}} - {c_{44}}} \right){e^{2\lambda {\tau _1}}}{\lambda ^{2\theta }} + \left( {{c_{53}} - {c_{54}}} \right){e^{\lambda {\tau _1}}}{\lambda ^\theta } + \left( {{c_{63}} - {c_{64}}} \right), \\
{p_4} =&  - {c_{65}}.
\end{align*}
Let $\lambda=i\omega$ and substitute it into Eq. \eqref{311}, it follows that
\begin{equation}\label{312}
\begin{aligned}
&\left( {{a_1} + i{b_1}} \right)\left( {\cos (4\omega {\tau _4}) - i\sin (4\omega {\tau _4}}) \right)+ \left( {{a_2} + i{b_2}} \right) + \left( {{a_3} + i{b_3}} \right)\left( {\cos (4\omega {\tau _4}) + i\sin (4\omega {\tau _4}}) \right) \\
&+ {a_4} \left( {\cos (8\omega {\tau _4}) + i\sin (8\omega {\tau _4}}) \right) = 0,
\end{aligned}
\end{equation}
in which ${a_n} = {\mathop{\rm Re}\nolimits} [{p_n}(i\omega )]$, ${b_n} = {\mathop{\rm Im}\nolimits} [ {p_n}(i\omega )]$, $n = 1,2,3,4$. Separating the real and imaginary parts, we obtain that
\begin{equation}\label{313}
\begin{aligned}
&\left( {{a_1} + {a_3}} \right)\cos (4\omega {\tau _4}) + \left( {{b_1} - {b_3}} \right)\sin (4\omega {\tau _4}) + {a_2} =  - {a_4}\cos (8\omega {\tau _4}), \\
&\left( {{b_1} + {b_3}} \right)\cos (4\omega {\tau _4}) + \left( {{a_3} - {a_1}} \right)\sin (4\omega {\tau _4}) + {b_2} =  - {a_4}\sin (8\omega {\tau _4}).
\end{aligned}
\end{equation}
Squaring both sides of the two equations of \eqref{313}, respectively, and adding them up yields
\begin{equation}\label{314}
\begin{aligned}
&\left[ {\left( {{a_1} + {a_3}} \right)\cos (4\omega {\tau _4}) + \left( {{b_1} - {b_3}} \right)\sin (4\omega {\tau _4}) + {a_2}} \right]^2 \\
&+ \left[ {\left( {{b_1} + {b_3}} \right)\cos (4\omega {\tau _4}) + \left( {{a_3} - {a_1}} \right)\sin (4\omega {\tau _4}) + {b_2}} \right]^2 = {a_4}^2.
\end{aligned}
\end{equation}
Noting that $\sin(4\omega {\tau _4}) =  \pm \sqrt {1 - {{\cos }^2}(4\omega {\tau _4}})$, we consider the two cases:

{\bf (1)} If $\sin(4\omega {\tau _4}) = \sqrt {1 - {{\cos }^2}(4\omega {\tau _4}})$, then Eq. \eqref{314} takes the following form:
\begin{equation}\label{315}
\begin{aligned}
&\left[ {\left( {{a_1} + {a_3}} \right)\cos (4\omega {\tau _4}) + \left( {{b_1} - {b_3}} \right)\sqrt {1 - {{\cos }^2}(4\omega {\tau _4}})  + {a_2}} \right]^2 \\
&+ \left[ {\left( {{b_1} + {b_3}} \right)\cos (4\omega {\tau _4}) + \left( {{a_3} - {a_1}} \right)\sqrt {1 - {{\cos }^2}(4\omega {\tau _4}})  + {b_2}} \right]^2 = {a_4}^2.
\end{aligned}
\end{equation}
It is easy to see that Eq. \eqref{315} is equivalent to
\begin{equation}\label{316}
{q_1}{\cos ^4}(4\omega {\tau _4}) + {q_2}{\cos ^3}(4\omega {\tau _4}) + {q_3}{\cos ^2}(4\omega {\tau _4}) + {q_4}\cos (4\omega {\tau _4}) + {q_5} = 0,
\end{equation}
where
\begin{small}
\begin{align*}
{q_1} =& 16{\left( {{a_1}{a_3} + {b_1}{b_3}} \right)^2} + 16{\left( {{a_3}{b_1} - {a_1}{b_3}} \right)^2}, \\
{q_2} =& 16\left( {{a_1}{a_3} + {b_1}{b_3}} \right)\left[ {{a_2}\left( {{a_1} + {a_3}} \right) + {b_2}\left( {{b_1} + {b_3}} \right)} \right] + 16\left( {{a_3}{b_1} - {a_1}{b_3}} \right)\left[ {{a_2}\left( {{b_1} - {b_3}} \right) + {b_2}\left( {{a_3} - {a_1}} \right)} \right], \\
{q_3} =& 8\left( {{a_1}{a_3} + {b_1}{b_3}} \right)\left[ {{a_2}^2 + {{\left( {{a_3} - {a_1}} \right)}^2} - {a_4}^2 + {{\left( {{b_1} - {b_3}} \right)}^2} + {b_2}^2} \right] + 4{\left[ {{a_2}\left( {{a_1} + {a_3}} \right) + {b_2}\left( {{b_1} + {b_3}} \right)} \right]^2} \\
&+ 4{\left[ {{a_2}\left( {{b_1} - {b_3}} \right) + {b_2}\left( {{a_3} - {a_1}} \right)} \right]^2} - 16{\left( {{a_3}{b_1} - {a_1}{b_3}} \right)^2}, \\
{q_4} =& 4\left[ {{a_2}\left( {{a_1} + {a_3}} \right) + {b_2}\left( {{b_1} + {b_3}} \right)} \right]\left[ { {a_2}^2 +{{\left( {{a_3} - {a_1}} \right)}^2} - {a_4}^2 + {{\left( {{b_1} - {b_3}} \right)}^2} + {b_2}^2} \right] \\
&- 16\left( {{a_3}{b_1} - {a_1}{b_3}} \right)\left[ {{a_2}\left( {{b_1} - {b_3}} \right) + {b_2}\left( {{a_3} - {a_1}} \right)} \right], \\
{q_5} =& {\left[ {{a_2}^2 + {{\left( {{a_3} - {a_1}} \right)}^2} - {a_4}^2 + {{\left( {{b_1} - {b_3}} \right)}^2} + {b_2}^2} \right]^2} - 4{\left[ {{a_2}\left( {{b_1} - {b_3}} \right) + {b_2}\left( {{a_3} - {a_1}} \right)} \right]^2}.
\end{align*}
\end{small}
Denote $r=\cos (4\omega {\tau _4})$, Eq. \eqref{316} is changed into
\begin{equation}\label{317}
{q_1}{r^4} + {q_2}{r^3} + {q_3}{r^2} + {q_4}r + {q_5} = 0.
\end{equation}
From Ferrari's method by back changing the variables and using the formulas for the quadratic and cubic equations, the four roots for Eq. \eqref{317} are listed as follows
\begin{equation}\label{318}
{r_{1,2}} =  - \frac{{{q_2}}}{{4{q_1}}} - S \pm \frac{1}{2}\sqrt { - 4{S^2} - 2\alpha  + \frac{\beta }{S}} ,\quad {r_{3,4}} =  - \frac{{{q_2}}}{{4{q_1}}} + S \pm \frac{1}{2}\sqrt { - 4{S^2} - 2\alpha  - \frac{\beta }{S}},
\end{equation}
where
\[\alpha  = \frac{{8{q_1}{q_3} - 3{q_2^2}}}{{8{q_1^2}}},\quad \beta  = \frac{{{q_2^3} - 4{q_1}{q_2}{q_3} + 8{q_1^2}{q_4}}}{{8{q_1^3}}},\]
in which
\[S = \frac{1}{2}\sqrt { - \frac{2}{3}\alpha + \frac{1}{{3q_1}}\left( {Q + \frac{{{\Delta _0}}}{Q}} \right)} ,\quad Q = \sqrt[3]{{\frac{{{\Delta _1} + \sqrt {{\Delta _1}^2 - 4{\Delta _0^3}} }}{2}}},\]
with ${\Delta _0} = {q_3^2} - 3{q_2}{q_4} + 12{q_1}{q_5}$, ${\Delta _1} = 2{q_3^3} - 9{q_2}{q_3}{q_4} + 27{q_2^2}{q_5} + 27{q_1}{q_4^2} - 72{q_1}{q_3}{q_5}$.
The bifurcation point of system \eqref{201} in this case follows that
\begin{equation}\label{319}
\tau^{*(k)} = \frac{1}{{4\omega }}\left[ {\arccos \left( {{r_j}} \right) + 2k\pi } \right],\quad j = 1,2,3,4,\quad k = 0,1,2, \cdots .
\end{equation}

{\bf (2)} If $\sin(4\omega {\tau _4}) = -\sqrt {1 - {{\cos }^2}(4\omega {\tau _4}})$, then \eqref{314} is changed into
\begin{equation}\label{320}
\begin{aligned}
&{\left[ {\left( {{a_1} + {a_3}} \right)\cos (4\omega {\tau _4}) - \left( {{b_1} - {b_3}} \right)\sqrt {1 - {{\cos }^2}(4\omega {\tau _4}}) + {a_2}} \right]^2} \\
&+ {\left[ {\left( {{b_1} + {b_3}} \right)\cos (4\omega {\tau _4}) - \left( {{a_3} - {a_1}} \right)\sqrt {1 - {{\cos }^2}(4\omega {\tau _4}}) + {b_2}} \right]^2} = {a_4}^2.
\end{aligned}
\end{equation}
Direct calculation shows that Eq. \eqref{320} has similar roots with Eq. \eqref{315}. Define the bifurcation point of system \eqref{201} as ${\tau _0^*} = \min \{\tau^{*(k)}\}$, $k = 0,1,2,\cdots$, where $\tau^{*(k)}$ is defined by \eqref{319}. By the above discussion, we obtain that ${\tau _0^*}$ is a function with respect to $\omega$. Denote $f(\omega)={\tau _0^*}$. From \eqref{314}, we have
\begin{equation}\label{321}
\begin{aligned}
&{\left[ {\left( {{a_1} + {a_3}} \right)\cos \left( {4\omega f(\omega )} \right) + \left( {{b_1} - {b_3}} \right)\sin \left( {4\omega f(\omega )} \right) + {a_2}} \right]^2} \\
&+ {\left[ {\left( {{b_1} + {b_3}} \right)\cos \left( {4\omega f(\omega )} \right) + \left( {{a_3} - {a_1}} \right)\sin \left( {4\omega f(\omega )} \right) + {b_2}} \right]^2} = {a_4}^2.
\end{aligned}
\end{equation}

\noindent Similarly, we make the following assumptions.

\noindent{\bf (H5)} Eq. \eqref{321} has at least one positive real root.

\noindent{\bf (H6)} $\left({{\Theta_1}{\Upsilon_1} + {\Theta_2}{\Upsilon _2}}\right)/\left({{\Upsilon _1}^2 + {\Upsilon _2}^2}\right)\neq0$, where ${\Theta _i}$, ${\Upsilon _i}$ $(i=1,2)$ are defined in Appendix C.

\begin{Lemma}\label{lem303}
Let $\lambda (\tau_4) = \mu (\tau_4) + i\omega (\tau_4)$, be the root of Eq. \eqref{311} near $\tau_4=\tau_0^*$ satisfying $\mu(\tau_0^*)=0$, $\omega(\tau_0^*)=\omega_0^*$, then the following transversality condition holds
\[{\mathop{\rm Re}\nolimits} \left[ {\frac{{d\lambda }}{{d\tau_4}}} \right]\Big| {_{\tau_4 = {\tau _0^*}}}  \ne 0.\]
\end{Lemma}
\begin{proof}
Based on implicit function theorem, we calculate the derivative of Eq. \eqref{311} with respect to $\tau_4$ as follows
\begin{align*}
&{e^{ - 4\lambda {\tau _4}}}{p_1}(\lambda )\left( { - 4\lambda  - 4{\tau _4}\frac{{d\lambda }}{{d{\tau _4}}}} \right) + {e^{ - 4\lambda {\tau _4}}}\frac{{d{p_1}(\lambda )}}{{d{\tau _4}}} + \frac{{d{p_2}(\lambda )}}{{d{\tau _4}}} \\
&+ {e^{4\lambda {\tau _4}}}{p_3}(\lambda )\left( {4\lambda  + 4{\tau _4}\frac{{d\lambda }}{{d{\tau _4}}}} \right) + {e^{4\lambda {\tau _4}}}\frac{{d{p_3}(\lambda )}}{{d{\tau _4}}} + \left( {8\lambda  + 8{\tau _4}\frac{{d\lambda }}{{d{\tau _4}}}} \right){e^{8\lambda {\tau _4}}}{p_4} = 0.
\end{align*}
Direct calculation yields
\[\frac{{d\lambda }}{{d{\tau _4}}} = \frac{{\Theta (s)}}{{\Upsilon (s)}},\]
where
\begin{align*}
\Theta (s) =& 4{e^{ - 4\lambda {\tau _4}}}\lambda {p_1}(\lambda ) - 4{e^{4\lambda {\tau _4}}}\lambda {p_3}(\lambda ) - 8{e^{8\lambda {\tau _4}}}\lambda {p_4}, \\
\Upsilon (s) =& 6\theta {e^{6\lambda {\tau _1}}}{e^{ - 4\lambda {\tau _4}}}{\lambda ^{6\theta  - 1}} + 5{c_{11}}\theta {e^{5\lambda {\tau _1}}}{e^{ - 4\lambda {\tau _4}}}{\lambda ^{5\theta  - 1}} + 4{c_{21}}\theta {e^{4\lambda {\tau _1}}}{e^{ - 4\lambda {\tau _4}}}{\lambda ^{4\theta  - 1}} \\
&+ 3{c_{31}}\theta {e^{3\lambda {\tau _1}}}{e^{ - 4\lambda {\tau _4}}}{\lambda ^{3\theta  - 1}} + 2{c_{41}}\theta {e^{2\lambda {\tau _1}}}{e^{ - 4\lambda {\tau _4}}}{\lambda ^{2\theta  - 1}} + {c_{51}}\theta {e^{\lambda {\tau _1}}}{e^{ - 4\lambda {\tau _4}}}{\lambda ^{\theta  - 1}} \\
&- 4{\tau _4}{e^{ - 4\lambda {\tau _4}}}{p_1}(\lambda ) + 4{\tau _4}{e^{4\lambda {\tau _4}}}{p_3}(\lambda ) + 8{p_4}{\tau _4}{e^{8\lambda {\tau _4}}} + 2\theta \left( {{c_{43}} - {c_{44}}} \right){e^{2\lambda {\tau _1}}}{e^{4\lambda {\tau _4}}}{\lambda ^{2\theta  - 1}} \\
&+ \theta \left( {{c_{53}} - {c_{54}}} \right){e^{\lambda {\tau _1}}}{e^{4\lambda {\tau _4}}}{\lambda ^{\theta  - 1}} - 4{c_{22}}\theta {e^{4\lambda {\tau _1}}}{\lambda ^{4\theta  - 1}} - 3{c_{32}}\theta {e^{3\lambda {\tau _1}}}{\lambda ^{3\theta  - 1}} \\
&- 2{c_{42}}\theta {e^{2\lambda {\tau _1}}}{\lambda ^{2\theta  - 1}} - {c_{52}}\theta {e^{\lambda {\tau _1}}}{\lambda ^{\theta  - 1}}.
\end{align*}
The real part of $d\lambda /d{\tau _4}$ at $\tau_4 = {\tau_0^*}$ is
\[{\rm{Re}}\left[ {\frac{{d\lambda }}{{d{\tau _4}}}} \right]\Big|_{{\tau _4} = {\tau _0^*}} = \frac{{{\Theta _1}{\Upsilon _1} + {\Theta _2}{\Upsilon _2}}}{{{\Upsilon _1}^2 + {\Upsilon _2}^2}},\]
where ${\Theta _1}$ and ${\Theta _2}$ are the real and imaginary parts of $\Theta(s)$, respectively; ${\Upsilon _1}$ and ${\Upsilon _2}$ are the real and imaginary parts of $\Upsilon(s)$, respectively. From Assumption {\bf (H6)}, the transversality condition meets. This completes the proof.
\end{proof}

\begin{Theorem}\label{thm303}
For system \eqref{201}, if Assumptions {\bf (H1), (H5), (H6)} and Lemma \ref{lem303} hold, we obtain the following results

{\bf a)} The zero equilibrium point is locally asymptotically stable for $\tau_4 \in [0,\tau_0^*)$;

{\bf b)} system \eqref{201} undergoes a Hopf bifurcation at the origin when $\tau_4=\tau_0^*$, i.e., it has a branch of periodic solutions bifurcating from the zero equilibrium point near $\tau_4=\tau_0^*$.
\end{Theorem}

\section{Representative examples}

In this section, we present some numerical simulations of system \eqref{201} to verify the analytical predictions obtained in Section 2. The simulation results are based on Adams-Bashforth-Moulton algorithm \cite{BV} and step-length $h=0.01$. Selecting appropriate system parameters, system \eqref{201} is changed into
\begin{small}
\begin{equation}\label{401}
\left\{
\begin{aligned}
&{D^\theta }{x_1}(t) =  - 0.4{x_1}(t - {\tau _1}) - 0.8{f_{11}}({y_1}(t - {\tau _2})) - 1.5{f_{12}}({y_2}(t - {\tau _2})) - 0.7{f_{13}}({y_3}(t - {\tau _2})), \\
&{D^\theta }{x_2}(t) =  - 0.6{x_2}(t - {\tau _1}) - 0.5{f_{21}}({y_1}(t - {\tau _2})) - 0.6{f_{22}}({y_2}(t - {\tau _2})) - 0.8{f_{23}}({y_3}(t - {\tau _2})), \\
&{D^\theta }{x_3}(t) =  - 0.5{x_3}(t - {\tau _1}) - 1.2{f_{31}}({y_1}(t - {\tau _2})) - 1.3{f_{32}}({y_2}(t - {\tau _2})) + 0.8{f_{33}}({y_3}(t - {\tau _2})), \\
&{D^\theta }{y_1}(t) =  - 0.7{y_1}(t - {\tau _1}) - 0.5{g_{11}}({x_1}(t - {\tau _2})) + 1.8{g_{12}}({x_2}(t - {\tau _2})) + 1.5{g_{13}}({x_3}(t - {\tau _2})), \\
&{D^\theta }{y_2}(t) =  - 0.8{y_2}(t - {\tau _1}) - 0.5{g_{21}}({x_1}(t - {\tau _2})) + 1.2{g_{22}}({x_2}(t - {\tau _2})) + 1.5{g_{23}}({x_3}(t - {\tau _2})), \\
&{D^\theta }{y_3}(t) =  - 0.3{y_3}(t - {\tau _1}) + 1.5{g_{31}}({x_1}(t - {\tau _2})) + 1.6{g_{32}}({x_2}(t - {\tau _2})) + 1.2{g_{33}}({x_3}(t - {\tau _2})),
\end{aligned}
\right.
\end{equation}
\end{small}
where activation functions are chosen as $f_{ij}(\cdot)=g_{ij}(\cdot)=tanh(\cdot)$.

In the next two subsections, we will illustrate our theoretical results by selecting different leakage delays and communication delays. Example 1 is relevant for Section 3.1, while Example 2 is related to Section 3.2. By means of plotting temporal solutions of $x_1(t)$, $x_2(t)$, $x_3(t)$, $y_1(t)$, $y_2(t)$ and $y_3(t)$ versus $t$ and corresponding phase diagrams of system \eqref{401} projected on different quadrants, the dynamical behavior of system \eqref{401} can be depicted clearly.

\subsection{Example 1}

In this example, we choose $\tau_3=\tau_1=\tau_2$, i.e., $\tau_4=0$, and illustrate the theoretical results obtained in Section 3.1. The fractional order here is chosen as $\theta=0.91$ and the initial value is selected as $\left(x_1(0),x_2(0),x_3(0),y_1(0),y_2(0),y_3(0)\right)=(0.2,0.4,-0.3,0.3,-0.5,-0.4)$. From \eqref{307}, we obtain that the critical frequency $\omega_0=2.2603$ and the bifurcation point $\tau_0=0.1234$. From Fig.\ref{fig1} and Fig.\ref{fig2}, it can be illustrated that the zero equilibrium point is locally asymptotically stable when $\tau_3=0.1<\tau_0=0.1234$, while, Fig.\ref{fig3} and Fig.\ref{fig4} demonstrate that the zero equilibrium point loses stability and Hopf bifurcation occurs when $\tau_3=0.15>\tau_0=0.1234$. This example indicates that leakage delay and communication delay indeed have a destabilizing influence on the stability performance of system \eqref{401}.

\begin{figure}[H]
\centering
\includegraphics[width=2in]{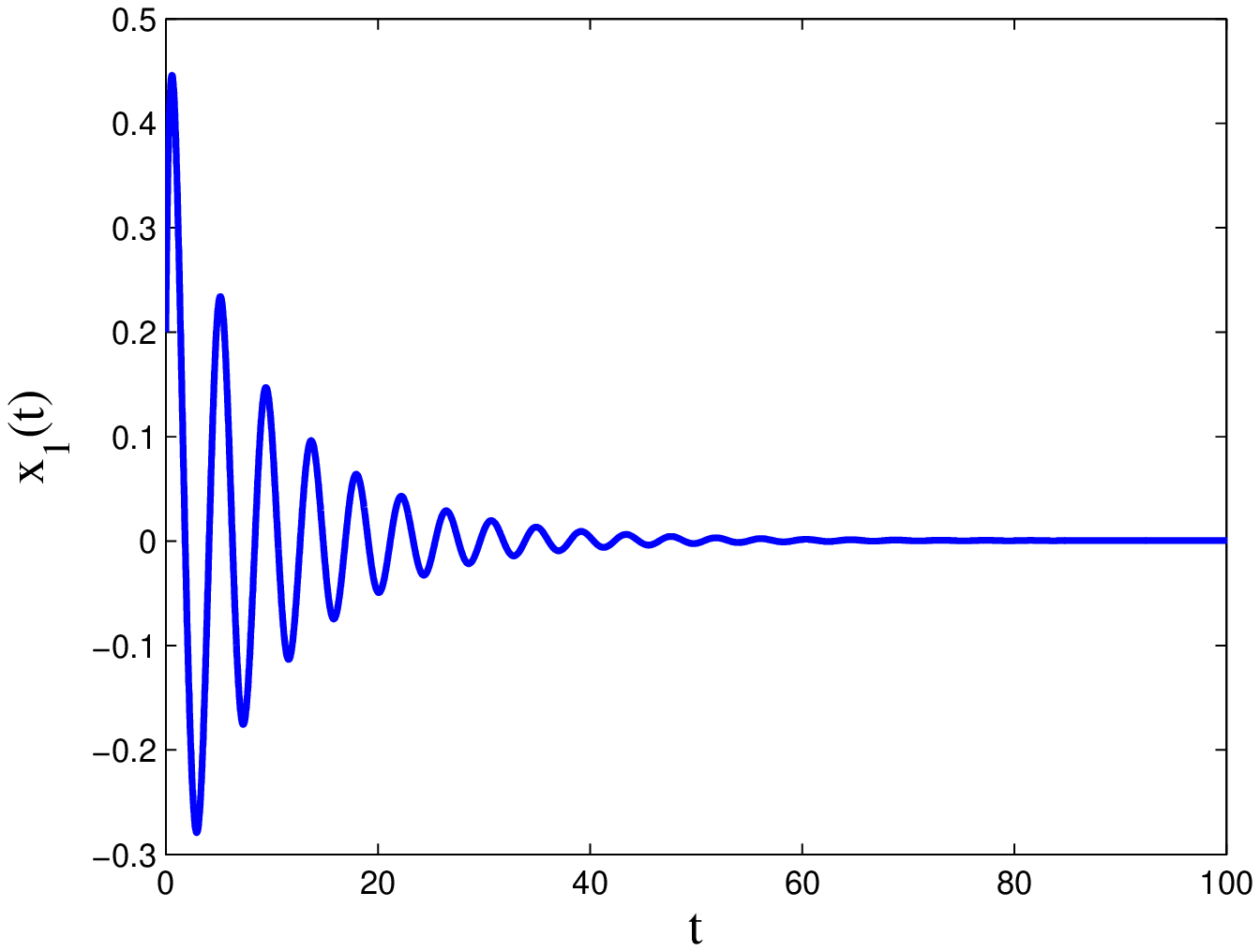}
\includegraphics[width=2in]{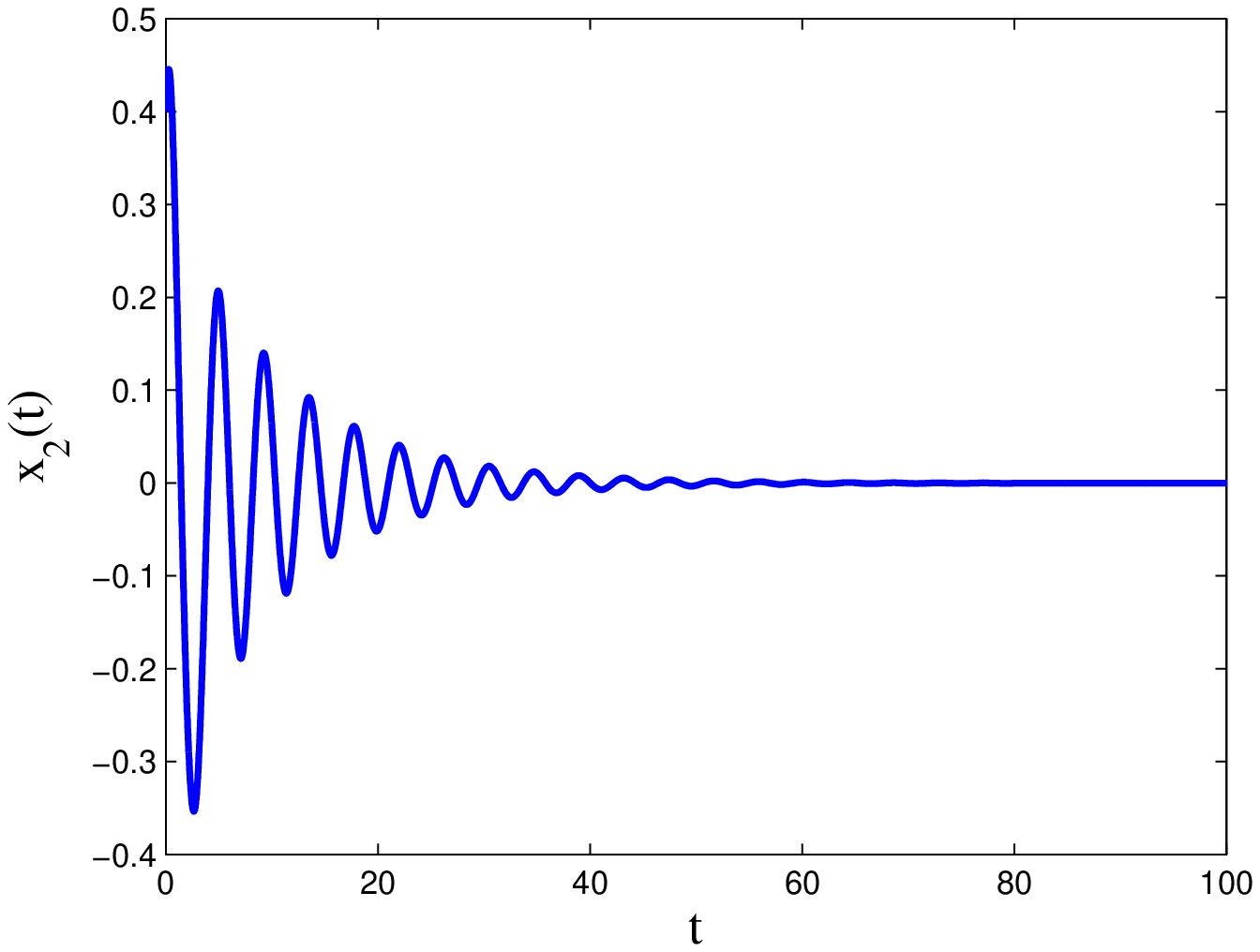}
\includegraphics[width=2in]{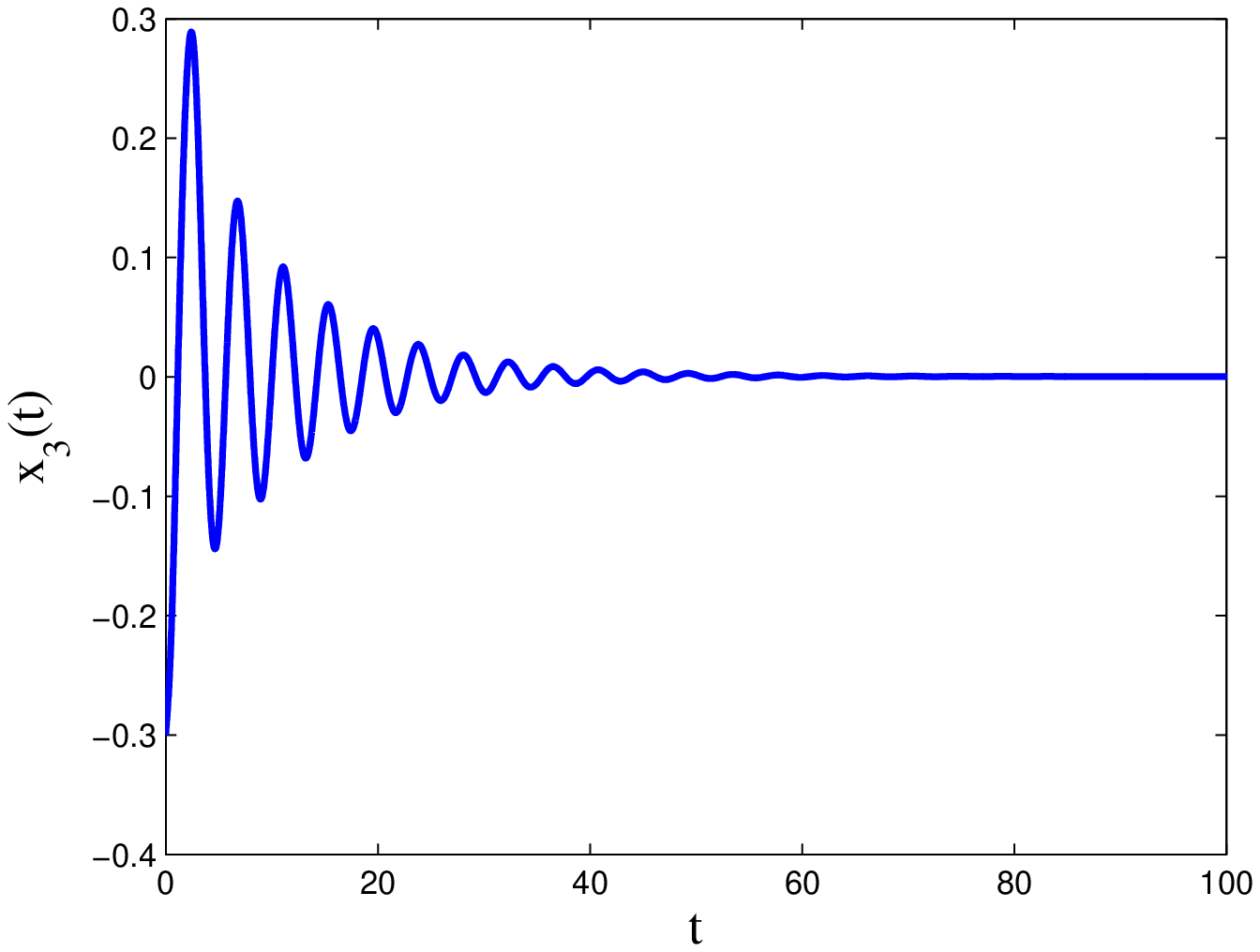}
\includegraphics[width=2in]{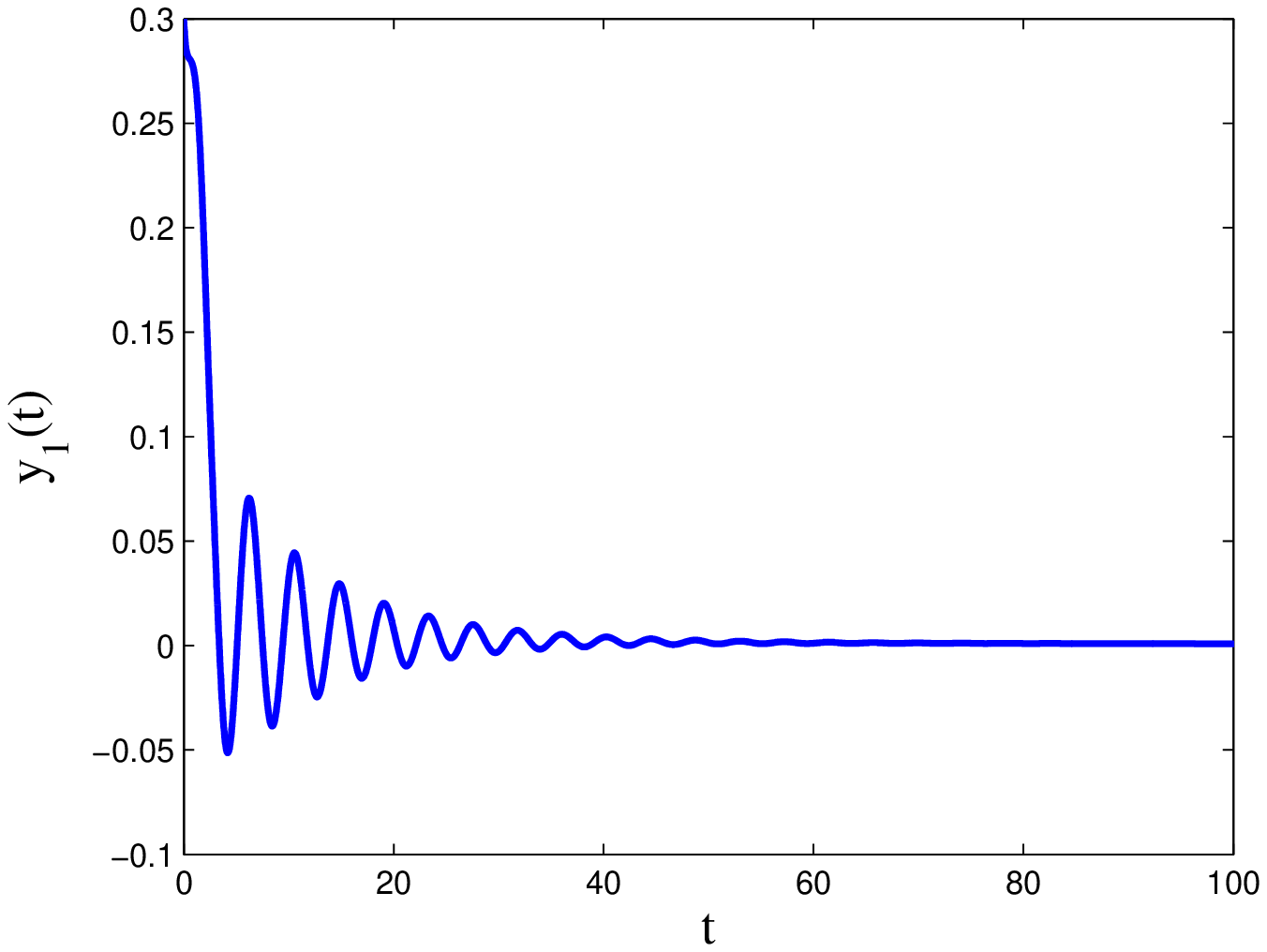}
\includegraphics[width=2in]{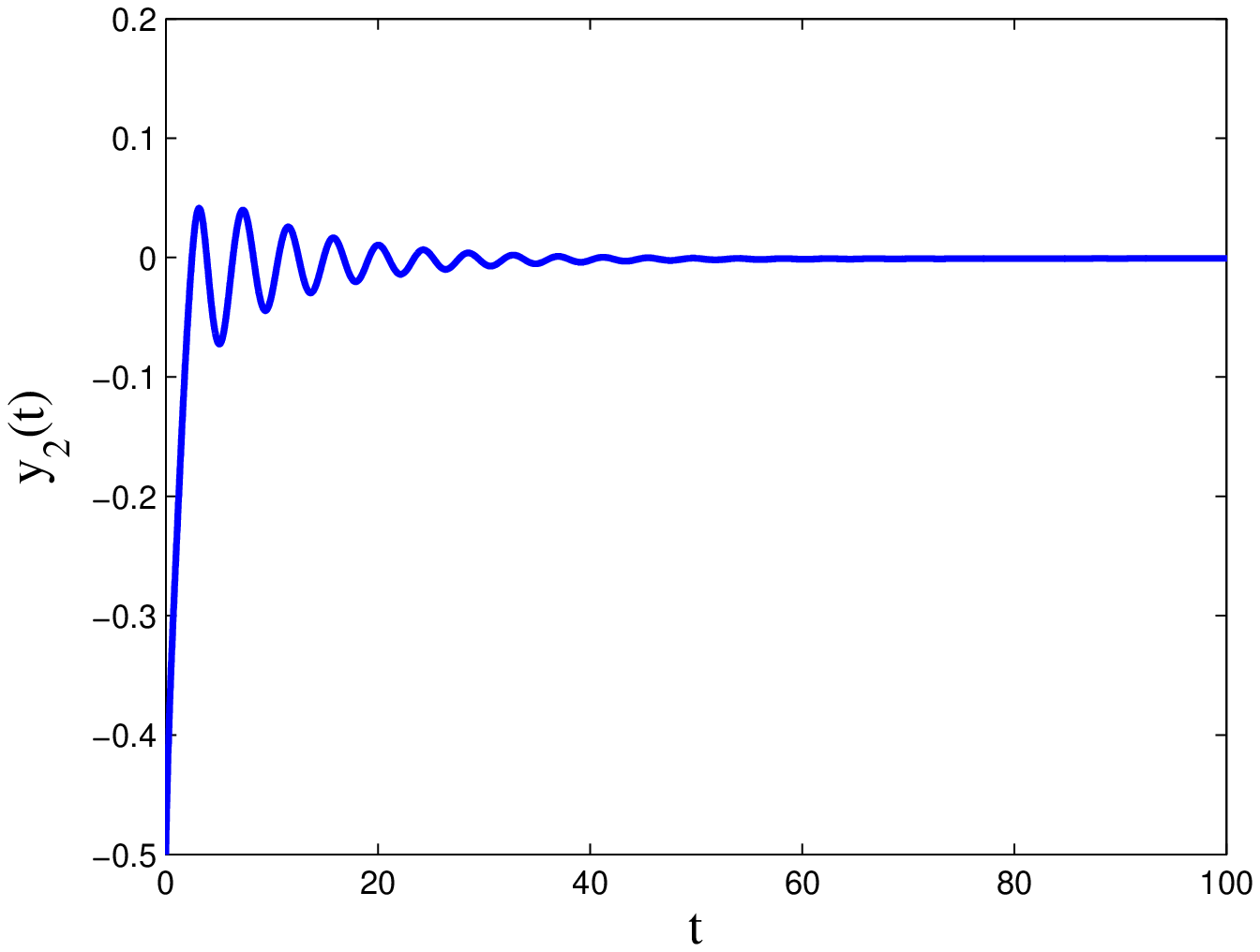}
\includegraphics[width=2in]{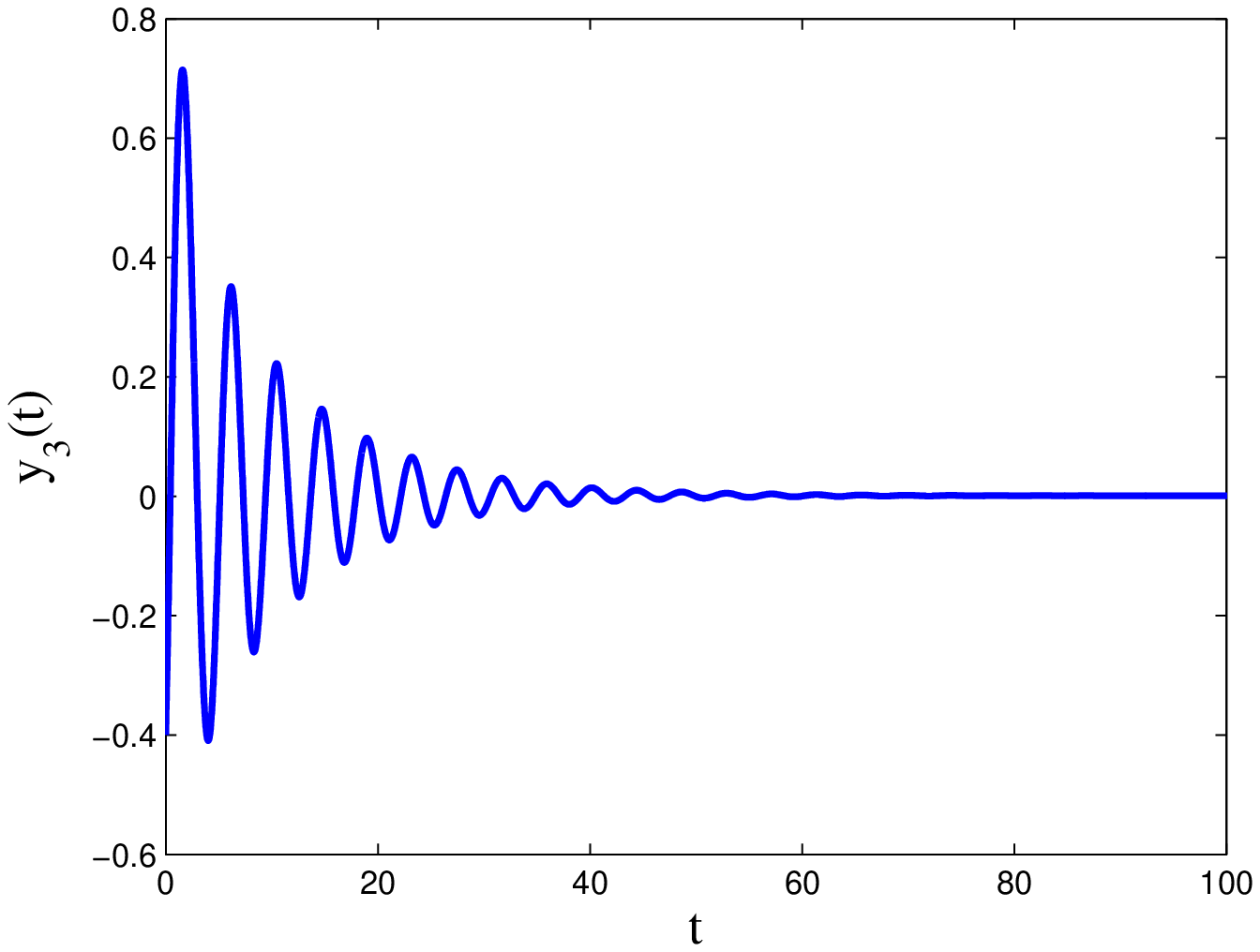}
\caption{\small{The temporal solutions of $x_1(t)$, $x_2(t)$, $x_3(t)$, $y_1(t)$, $y_2(t)$ and $y_3(t)$ versus $t$ of system \eqref{401} with $\theta=0.91$ and $\tau_3=0.1<\tau_0=0.1234$.}}
\label{fig1}
\end{figure}

\begin{figure}[H]
\centering
\includegraphics[width=2in]{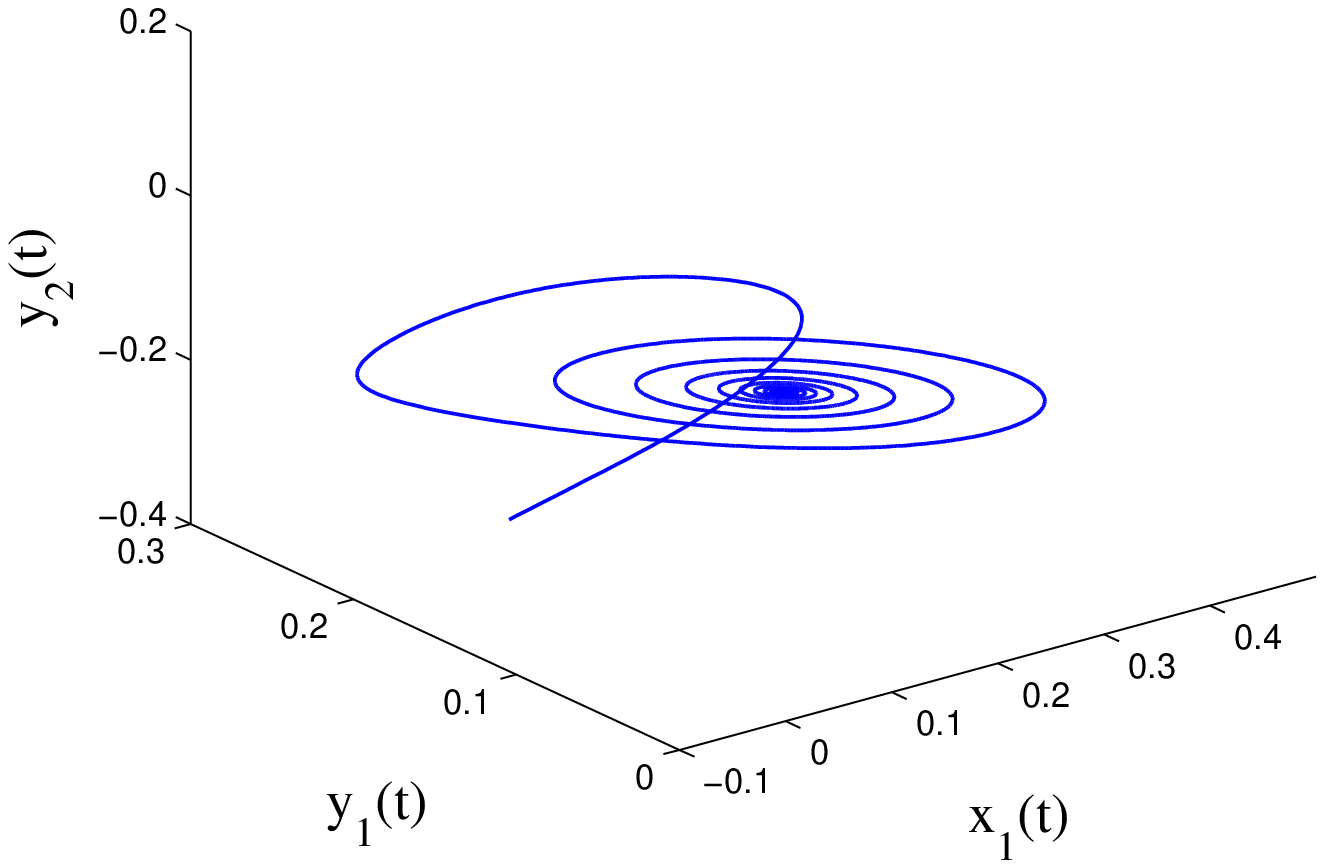}
\includegraphics[width=2in]{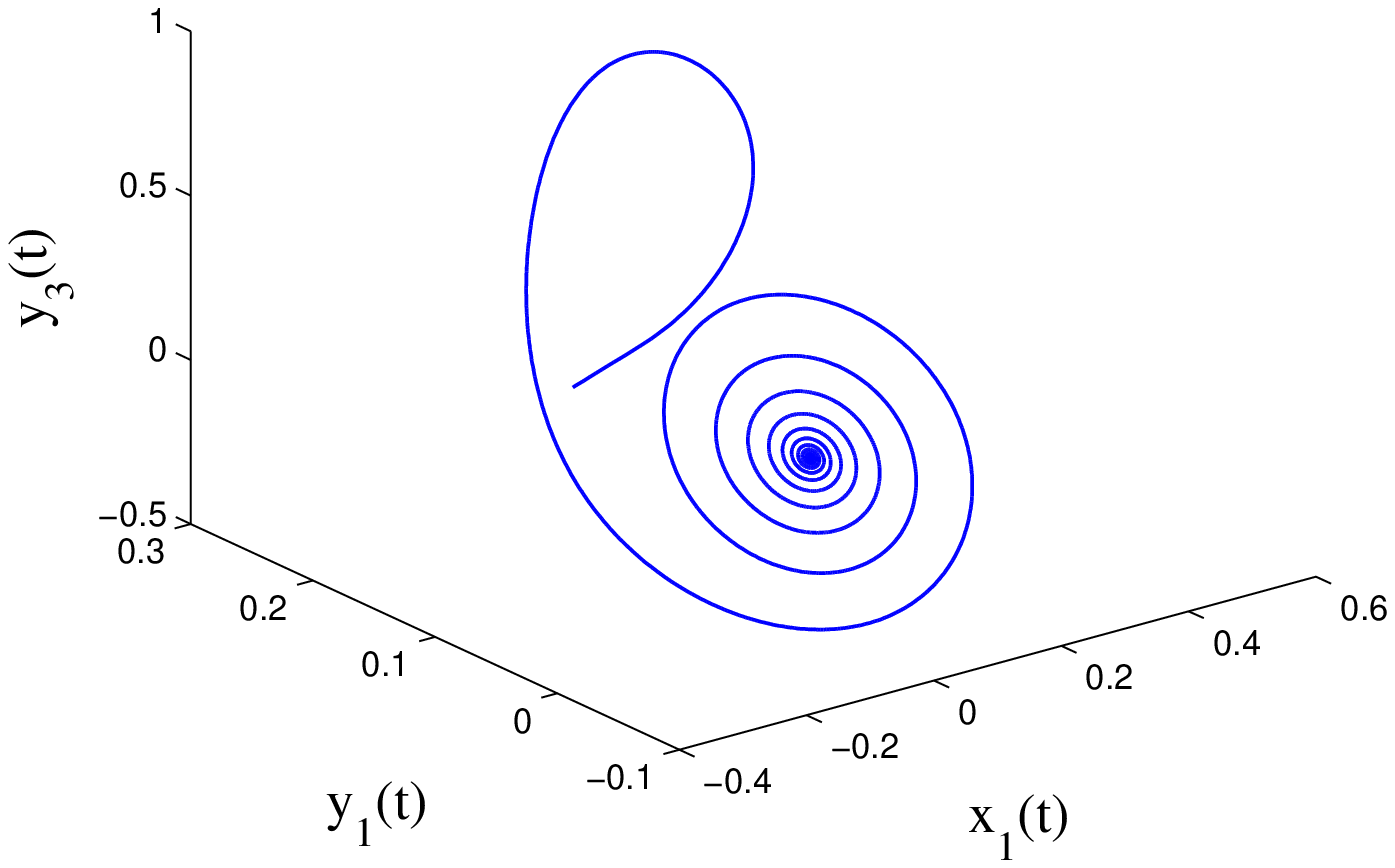}
\includegraphics[width=2in]{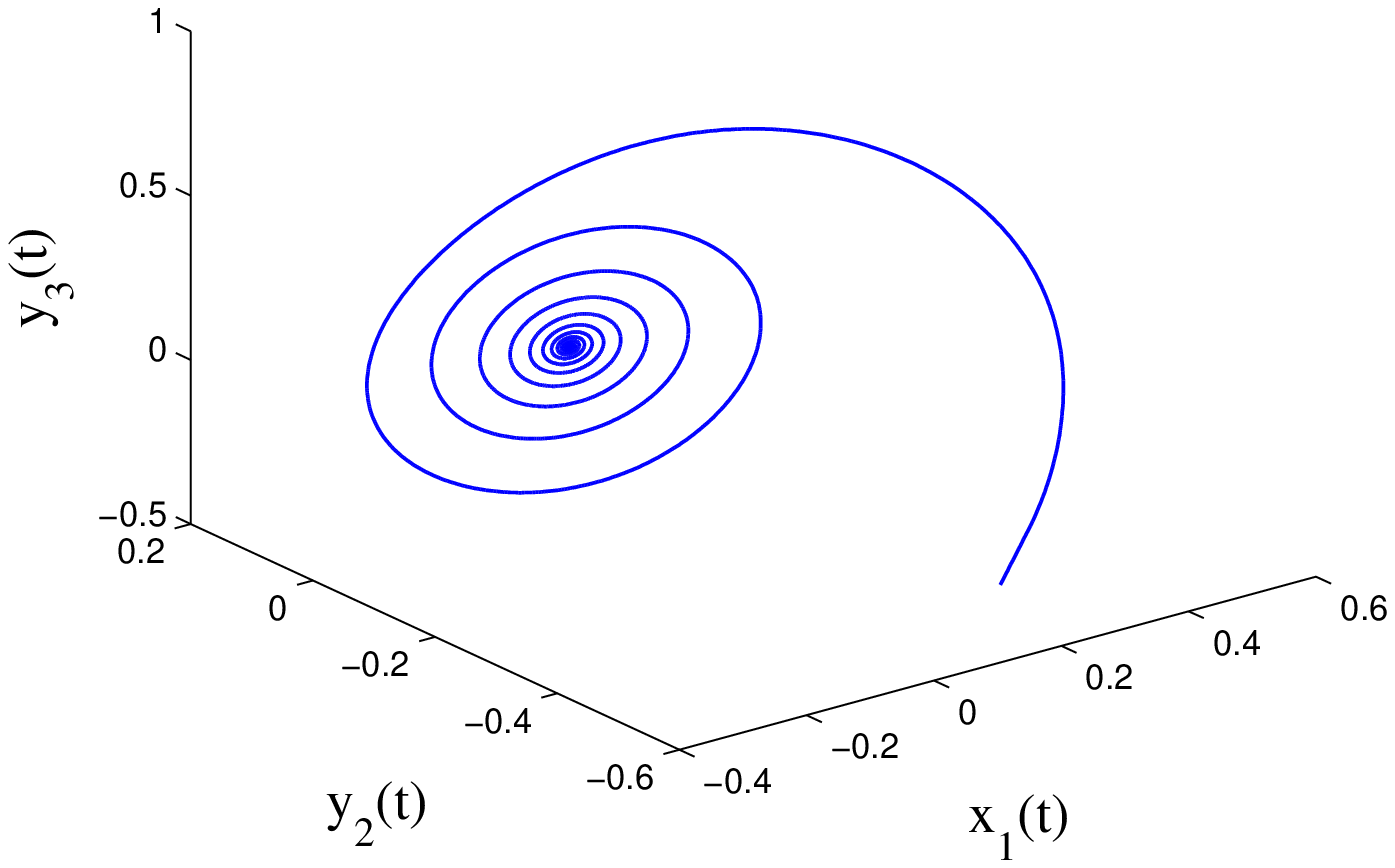}
\includegraphics[width=2in]{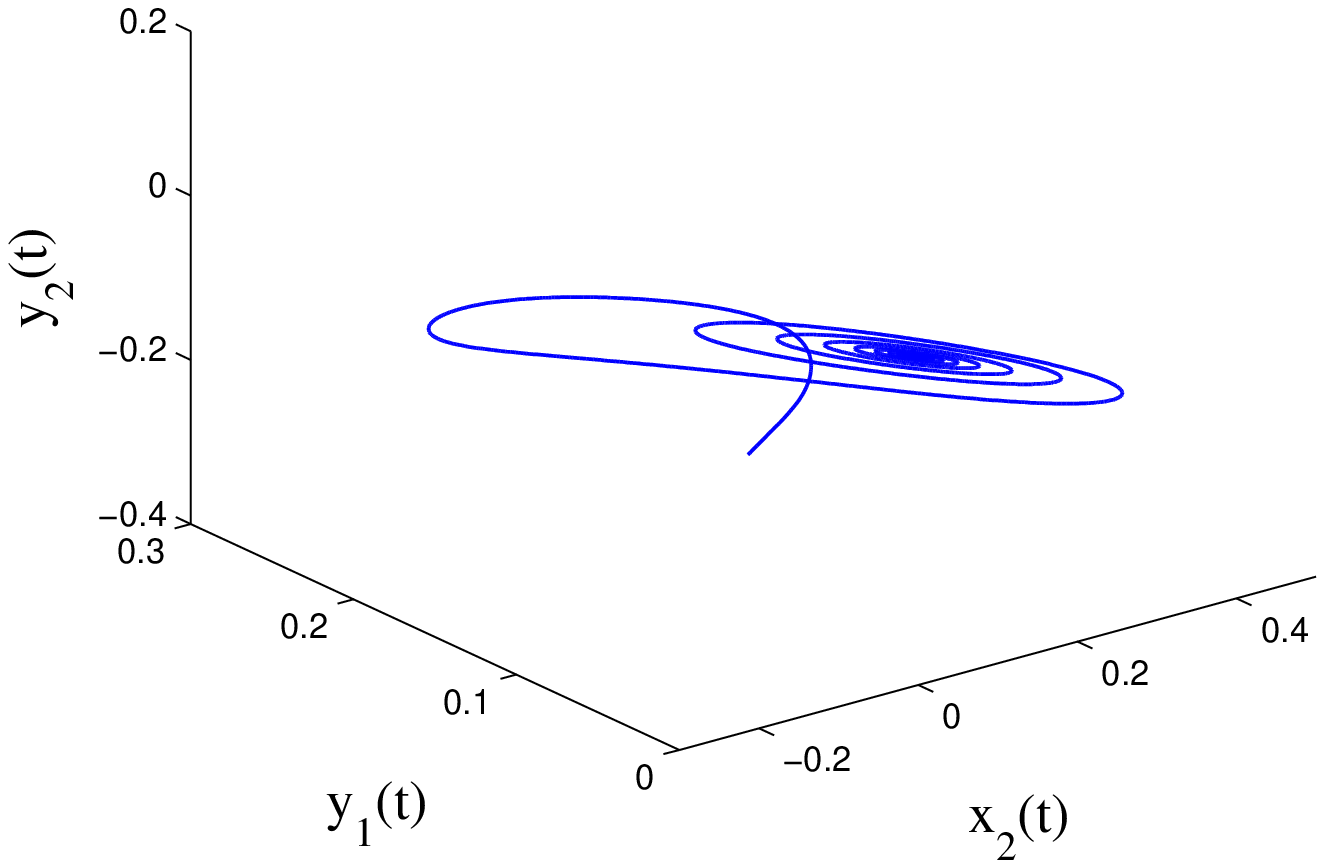}
\includegraphics[width=2in]{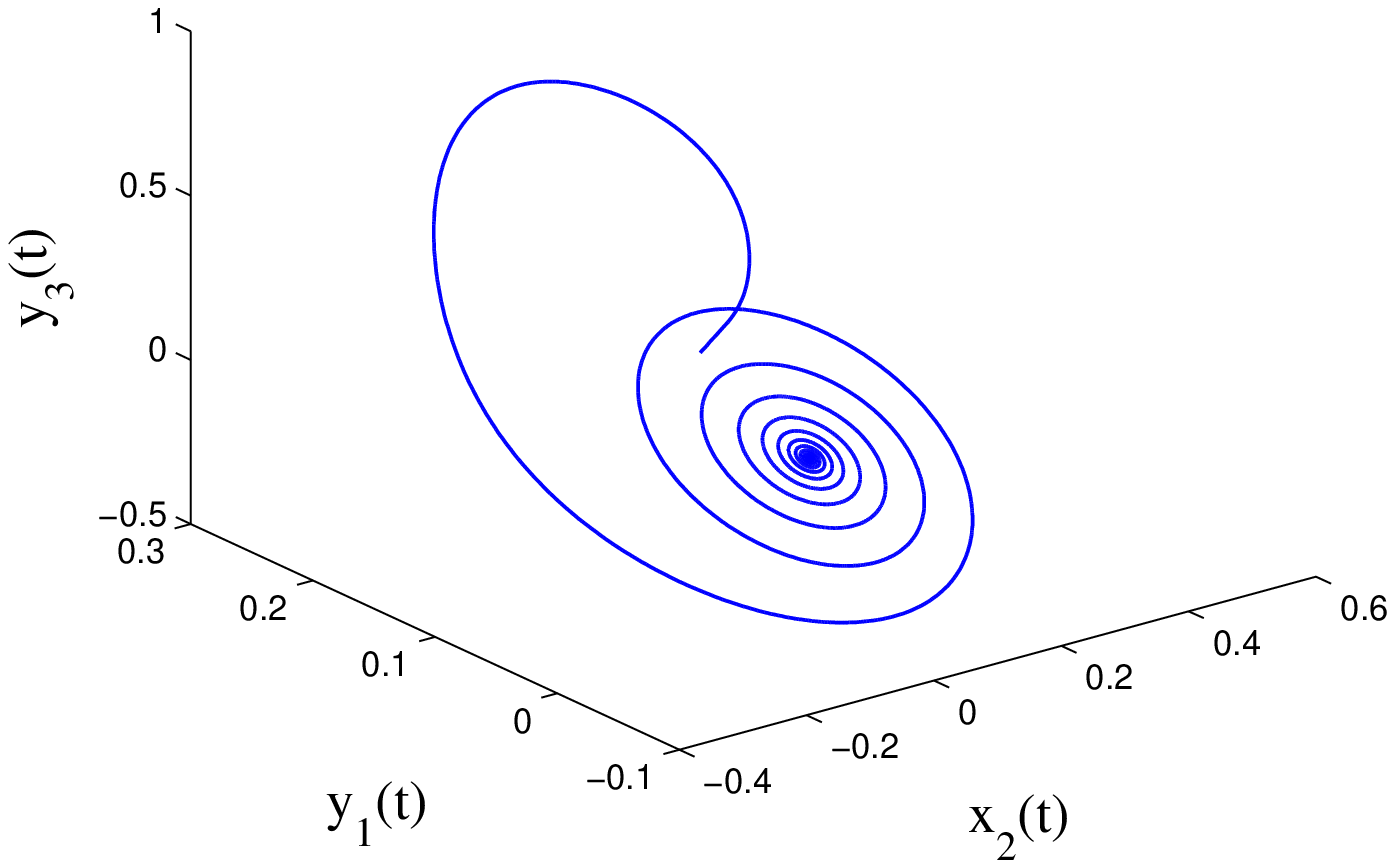}
\includegraphics[width=2in]{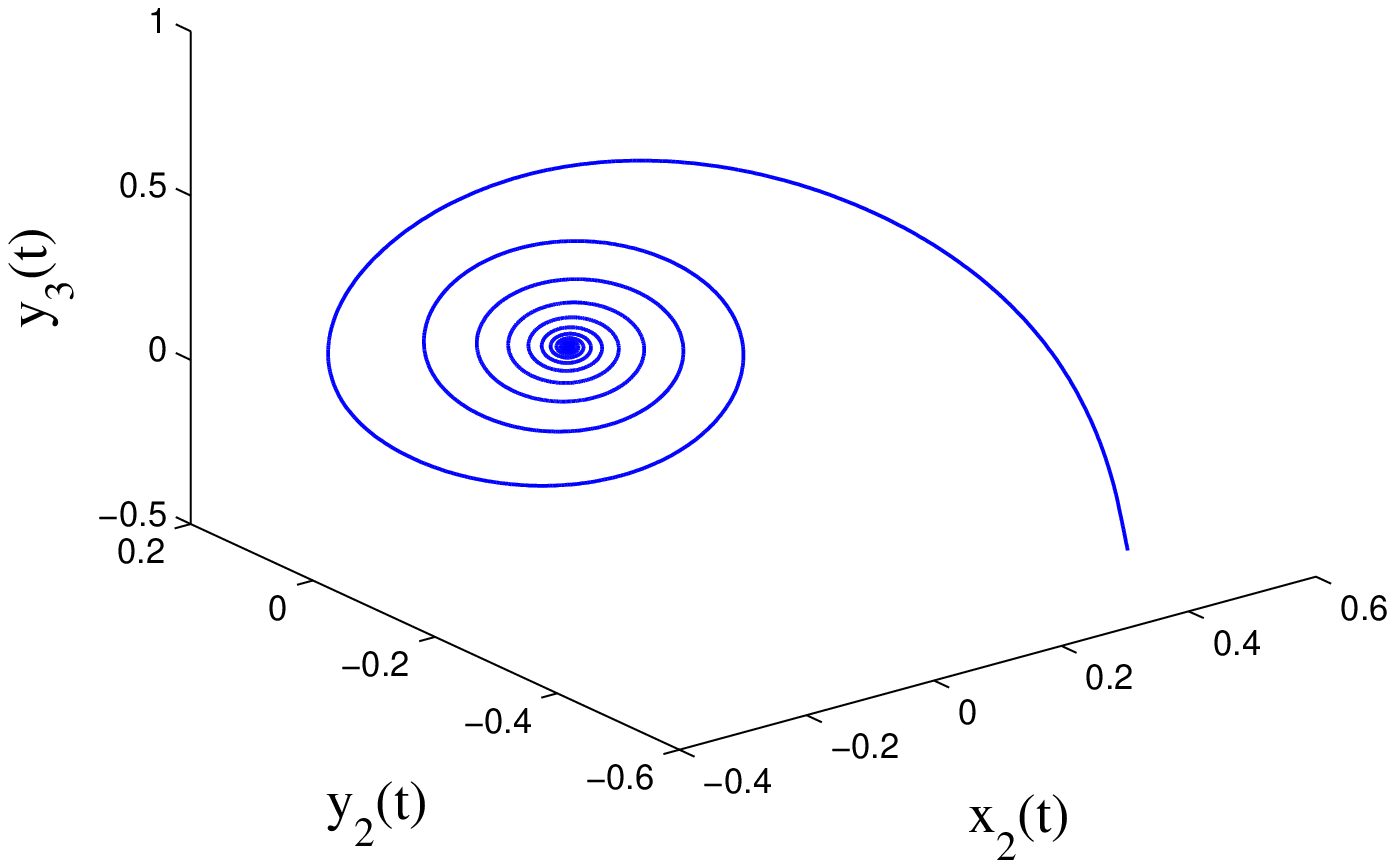}
\caption{\small{Phase diagrams of system \eqref{401} with $\theta=0.91$ and $\tau_3=0.1<\tau_0=0.1234$ projected on $x_1-y_1-y_2$, $x_1-y_1-y_3$, $x_1-y_2-y_3$, $x_2-y_1-y_2$, $x_2-y_1-y_3$ and $x_2-y_2-y_3$, respectively.}}
\label{fig2}
\end{figure}

\begin{figure}[H]
\centering
\includegraphics[width=2in]{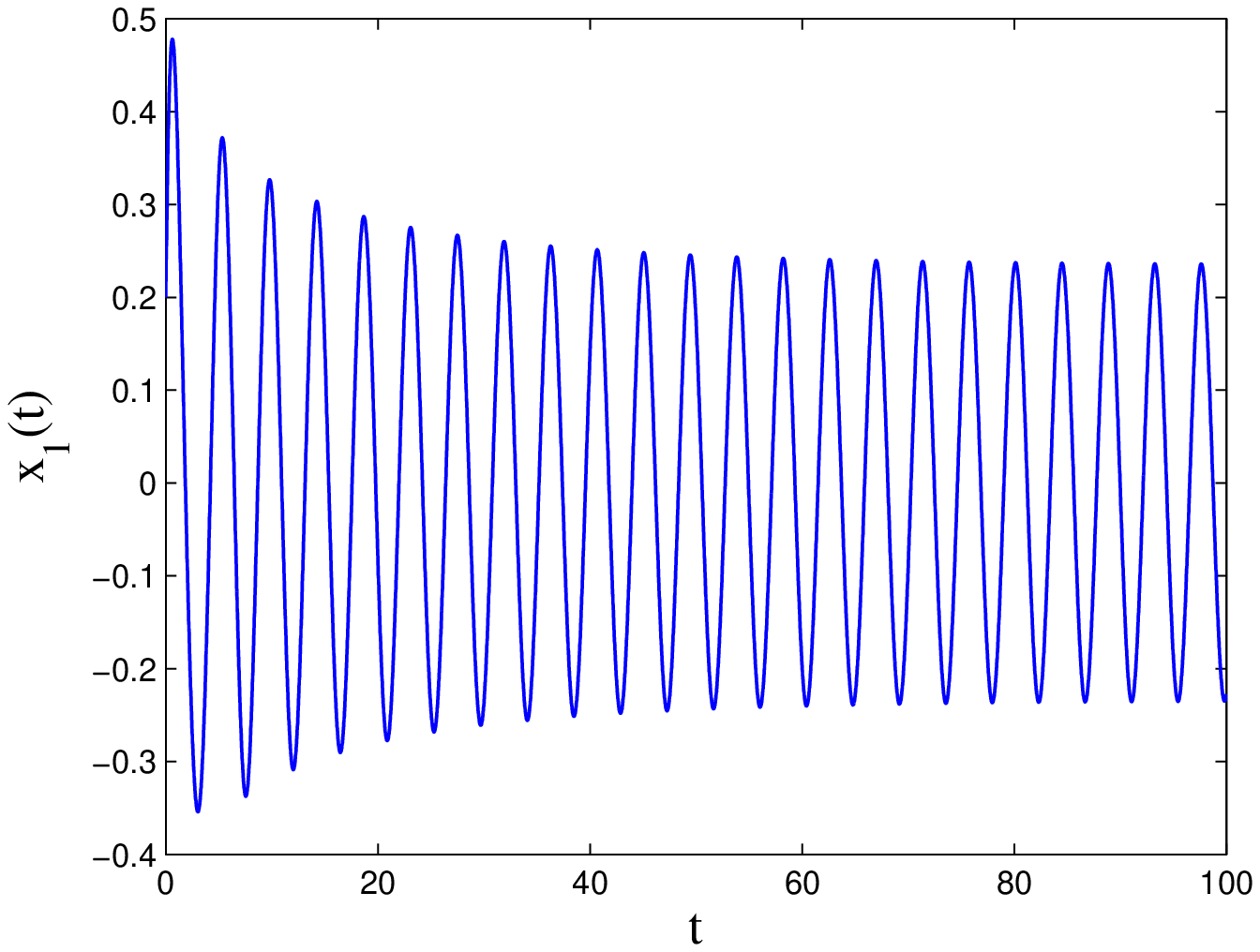}
\includegraphics[width=2in]{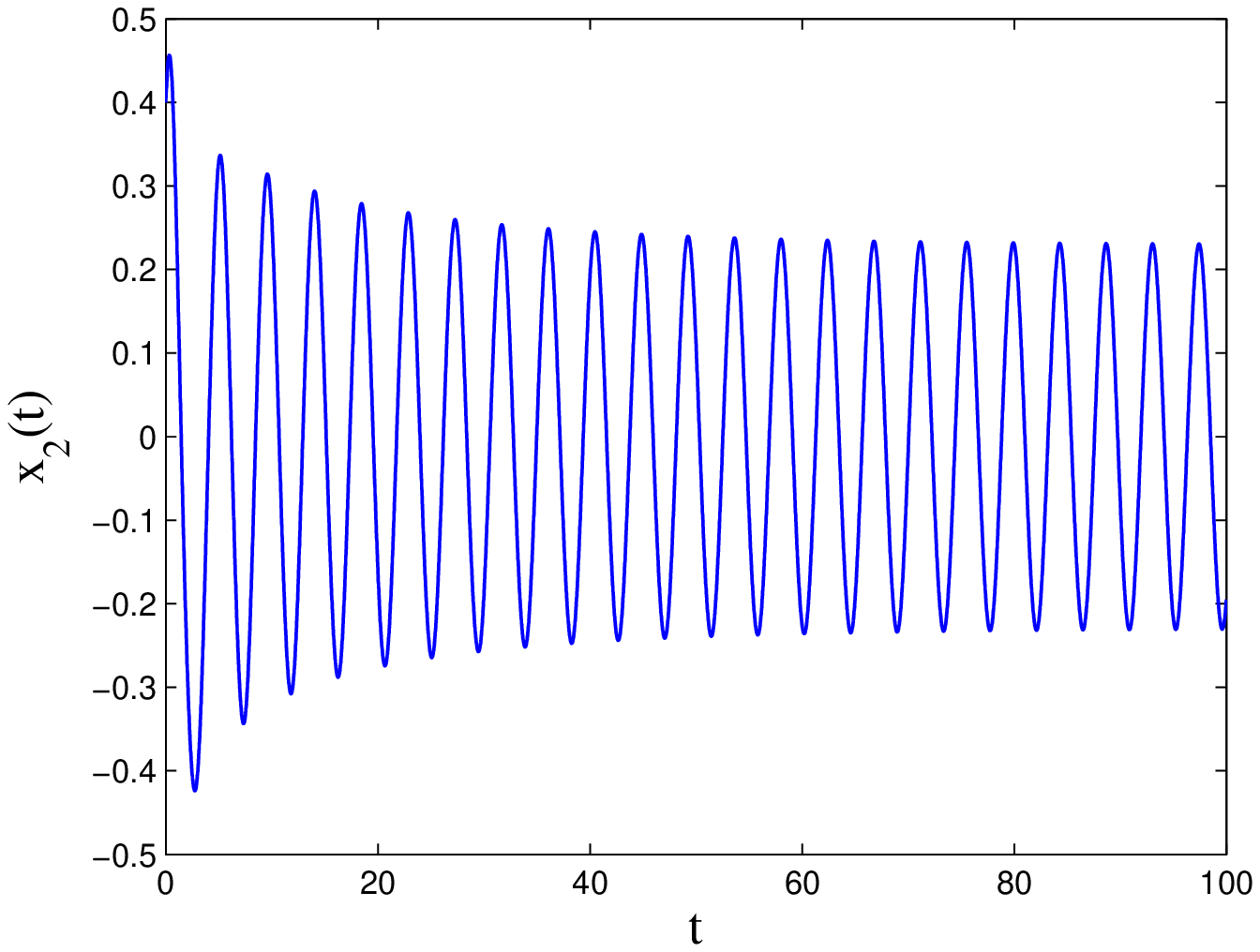}
\includegraphics[width=2in]{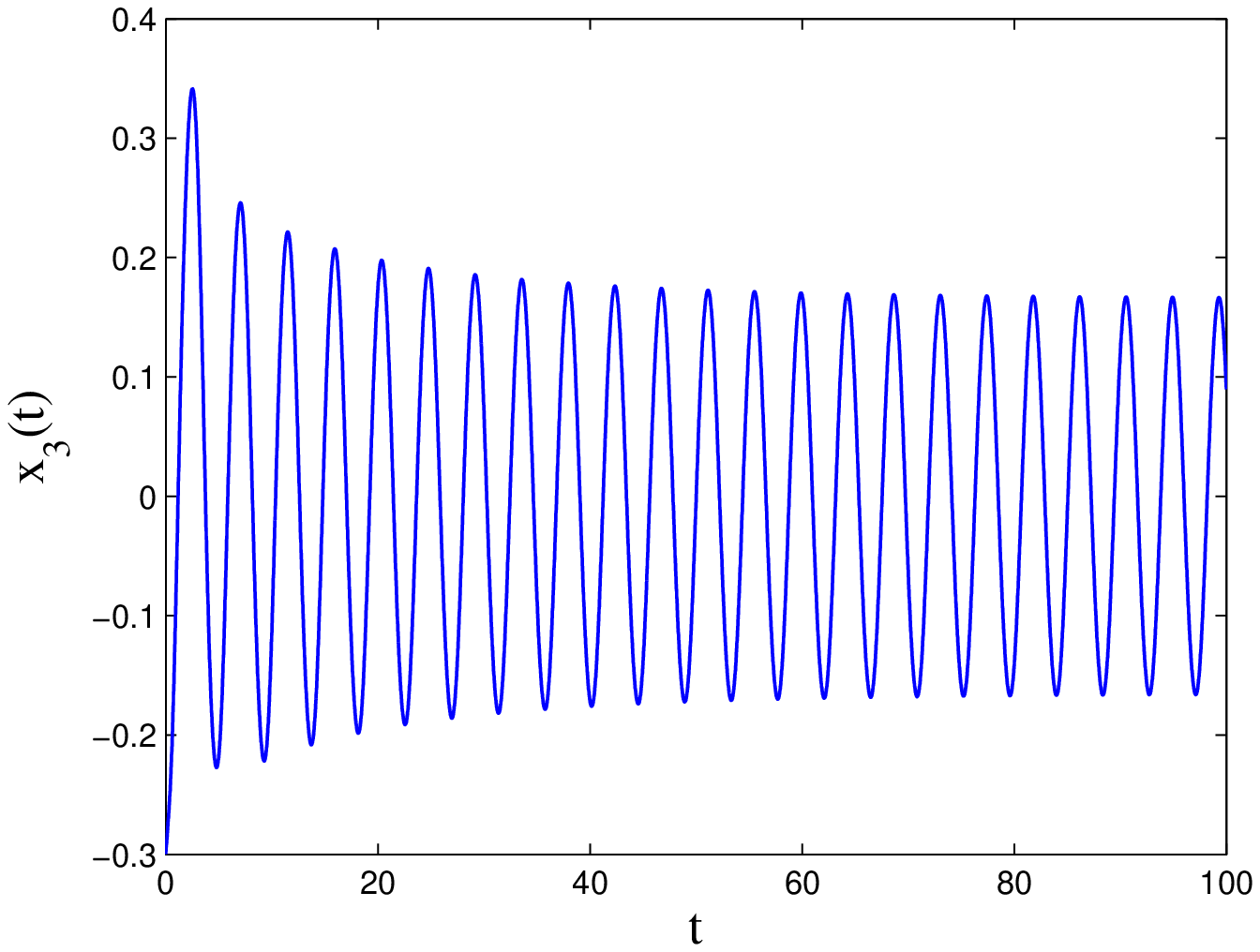}
\includegraphics[width=2in]{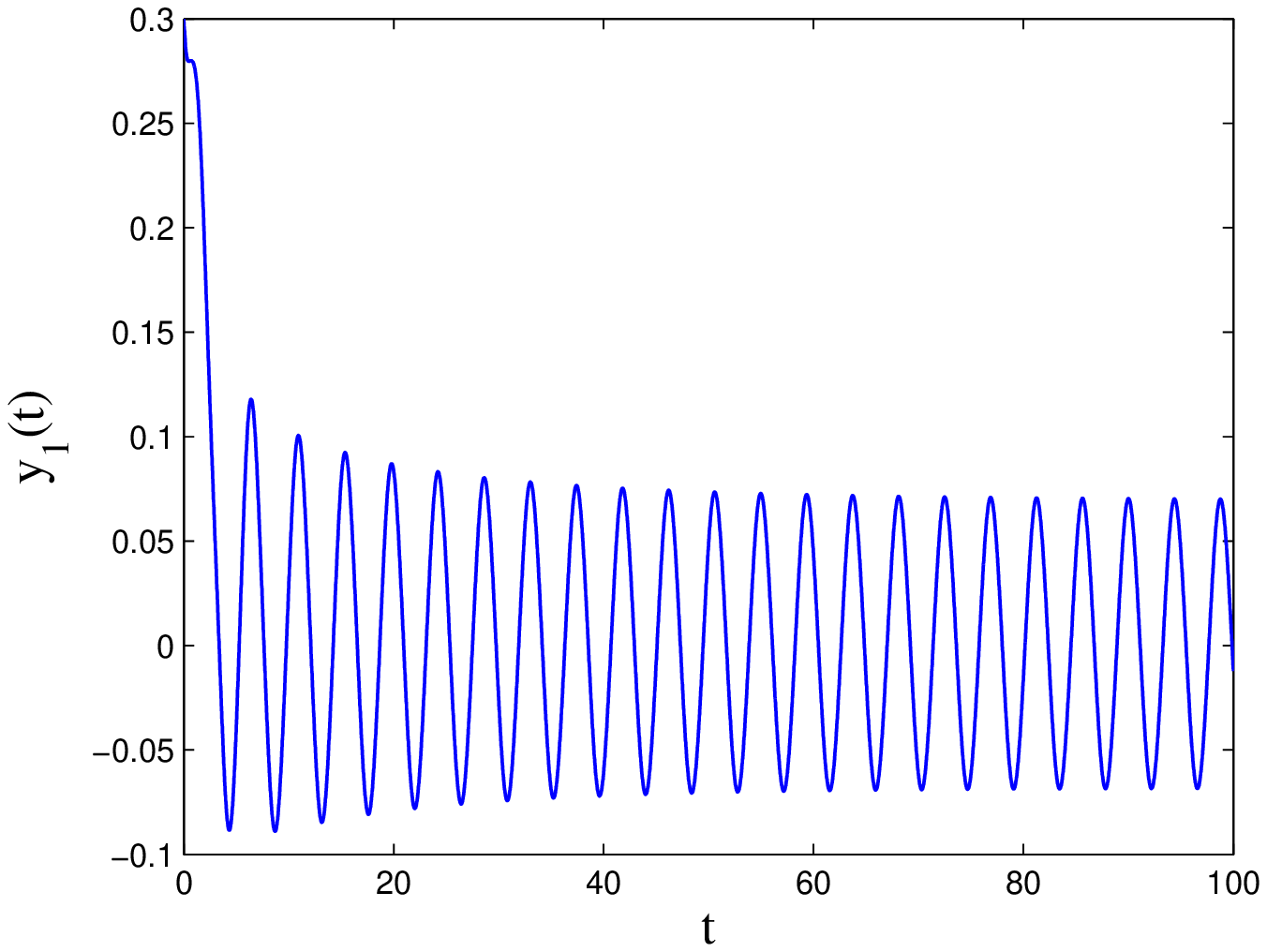}
\includegraphics[width=2in]{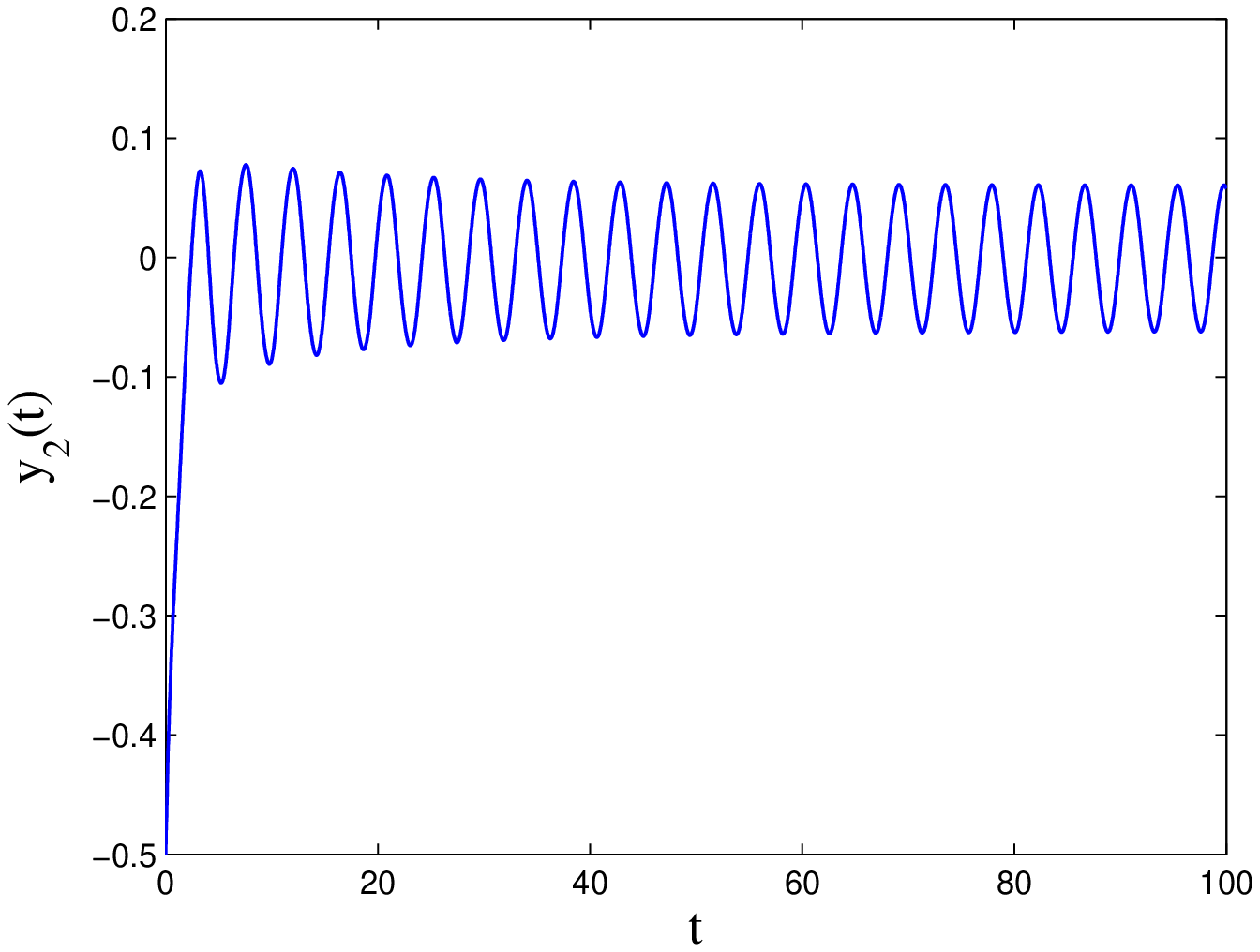}
\includegraphics[width=2in]{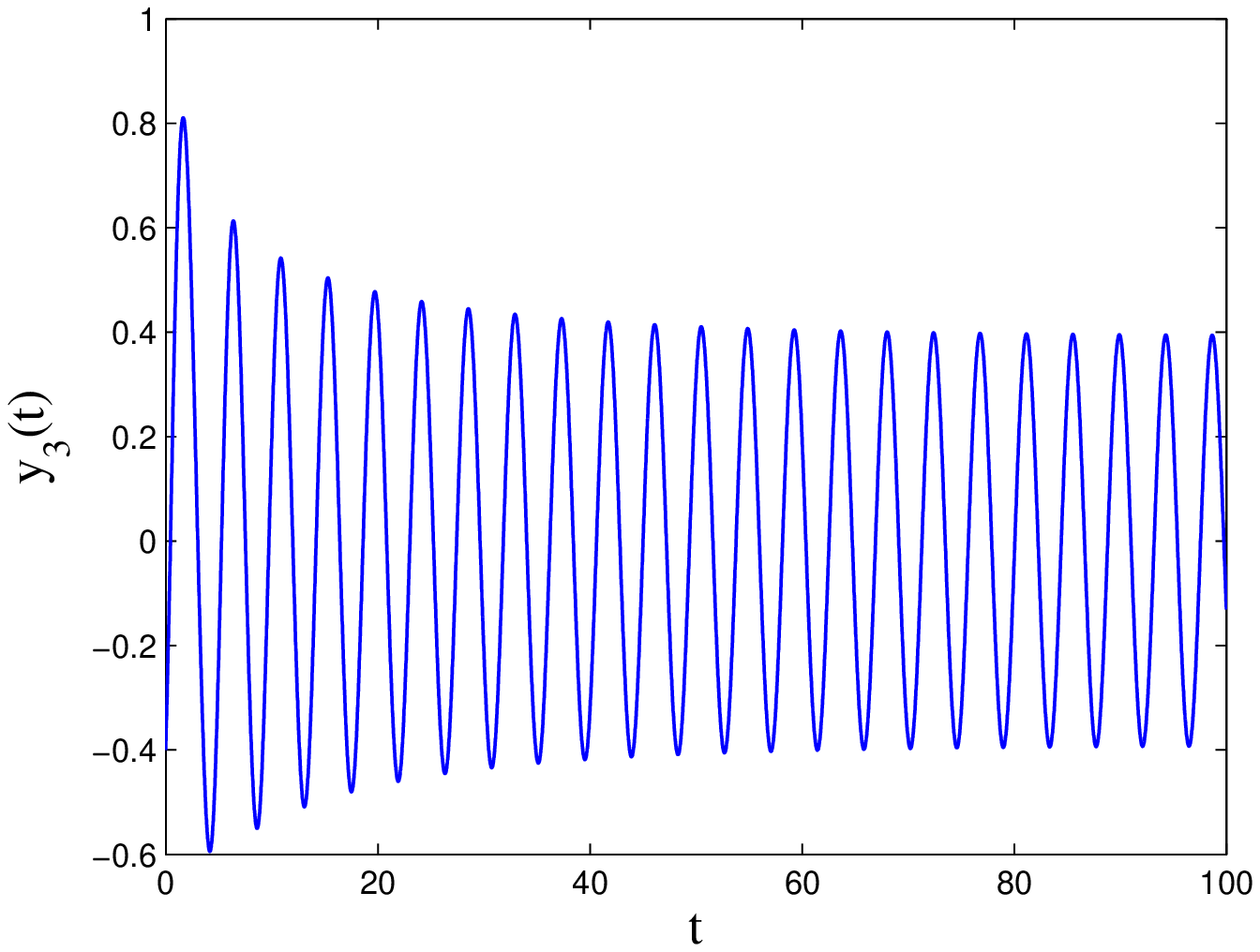}
\caption{\small{The temporal solutions of $x_1(t)$, $x_2(t)$, $x_3(t)$, $y_1(t)$, $y_2(t)$ and $y_3(t)$ versus $t$ of system \eqref{401} with $\theta=0.91$ and $\tau_3=0.15>\tau_0=0.1234$.}}
\label{fig3}
\end{figure}

\subsection{Example 2}

In this example, we fix $\tau_2=0.06<\tau_0=0.0681$ and expound the theoretical results obtained in Section 3.2. The fractional order here is chosen as $\theta=0.91$ and the initial value is selected as $\left(x_1(0),x_2(0),x_3(0),y_1(0),y_2(0),y_3(0)\right)=(-1.4, 0.6,-0.3,0.5,1.2,-0.7)$. From \eqref{318}, \eqref{319} and \eqref{321}, we approximately calculate the bifurcation point $\tau_0^* \approx 0.2613$. In
\begin{figure}[H]
\centering
\includegraphics[width=2in]{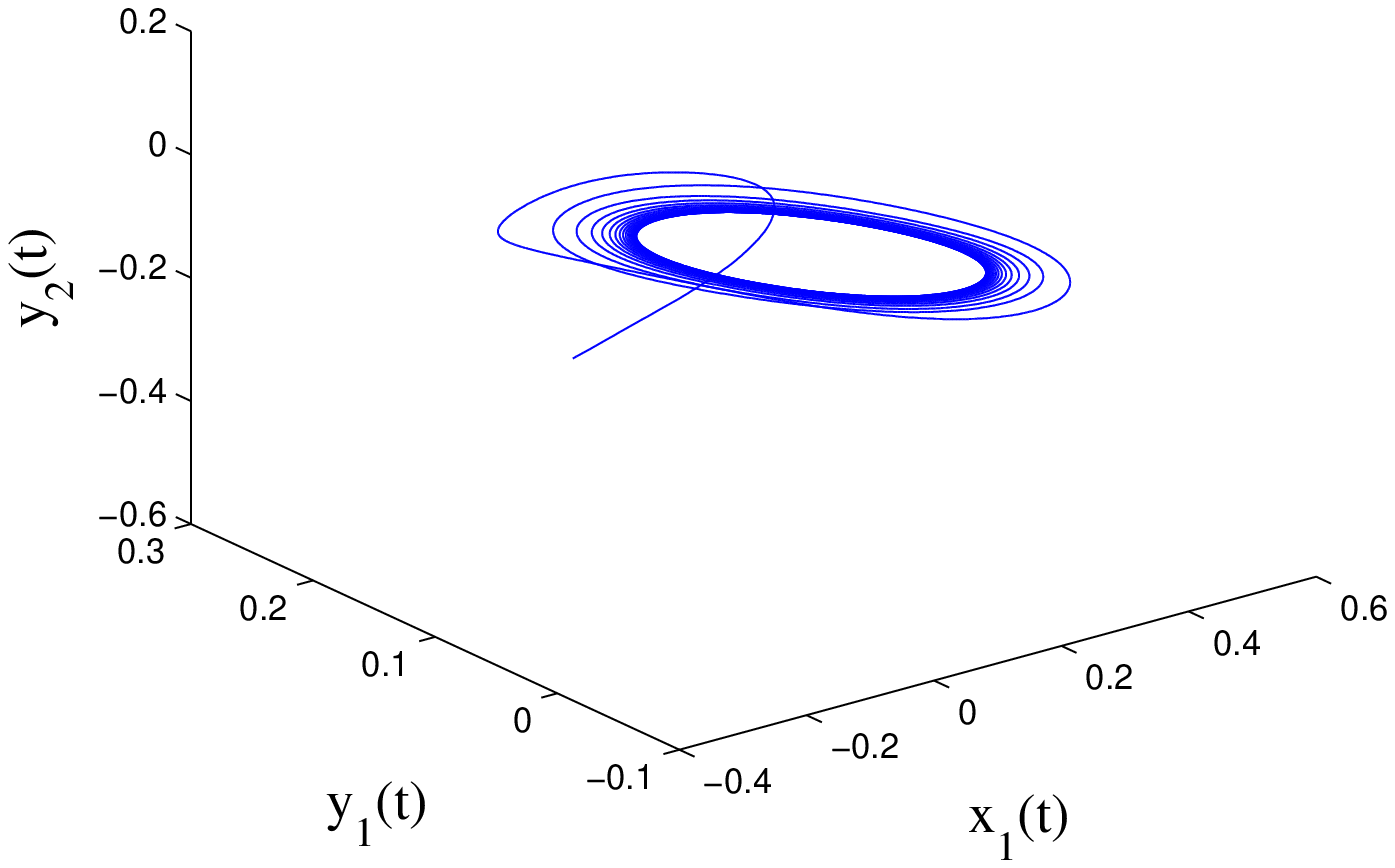}
\includegraphics[width=2in]{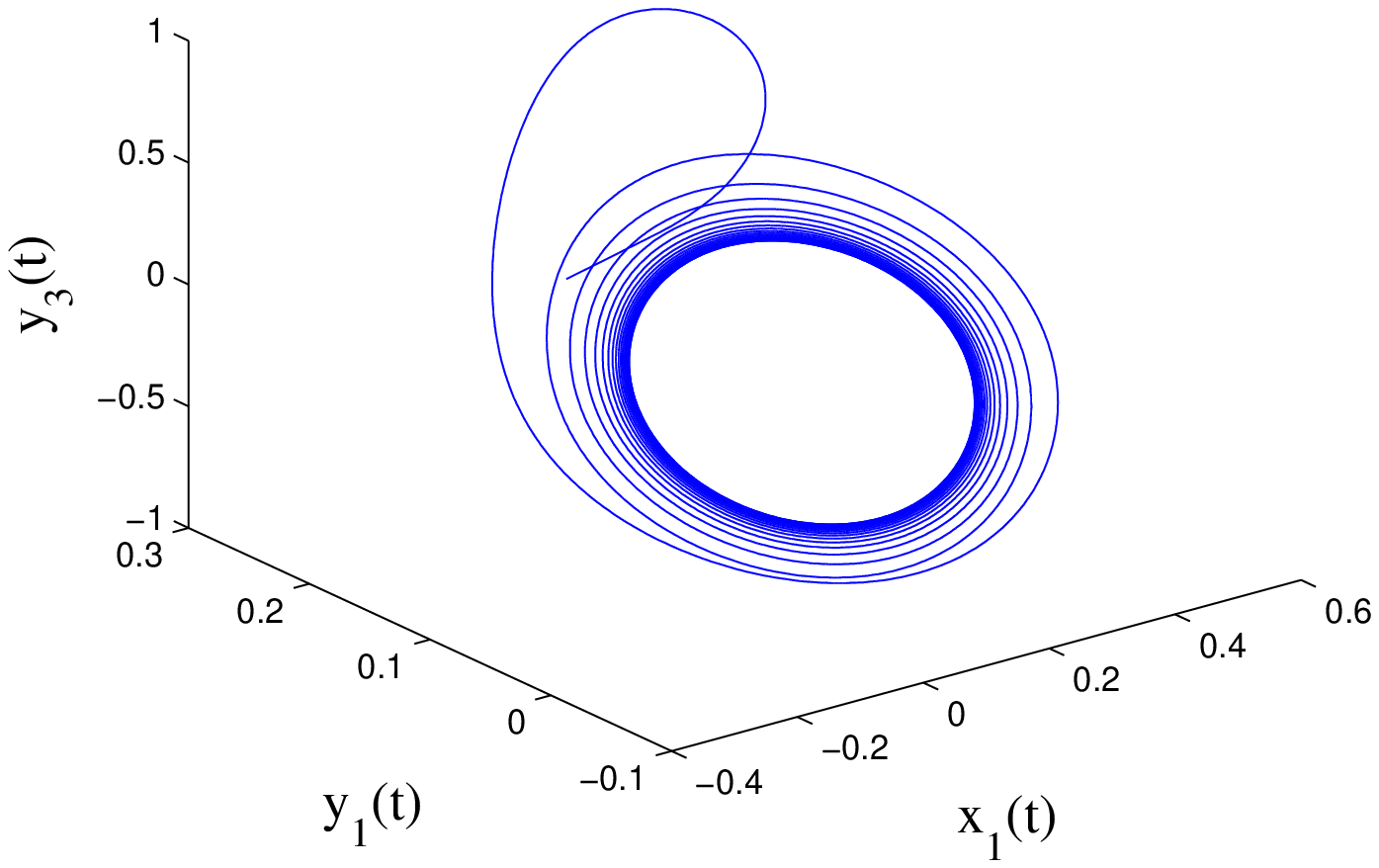}
\includegraphics[width=2in]{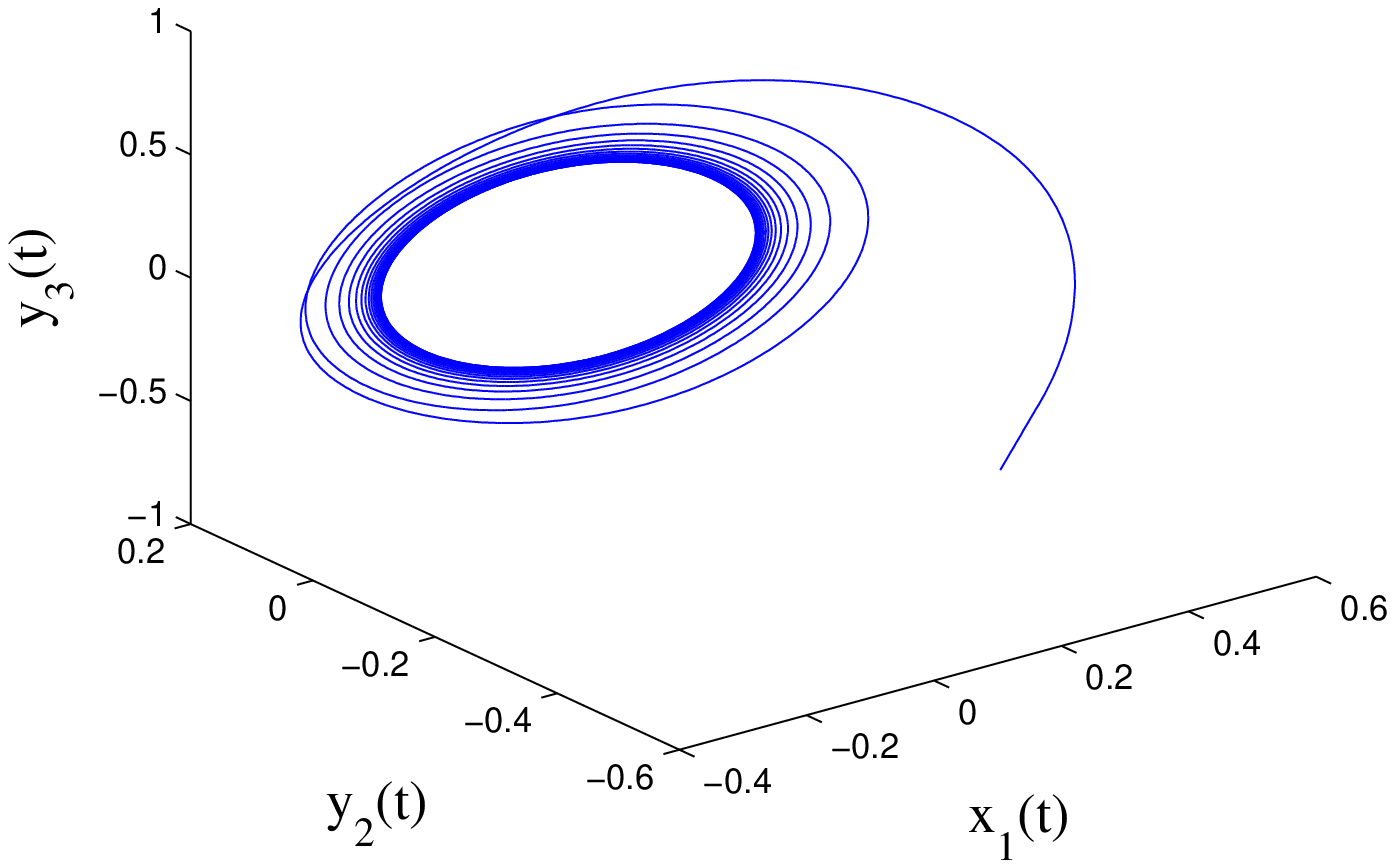}
\includegraphics[width=2in]{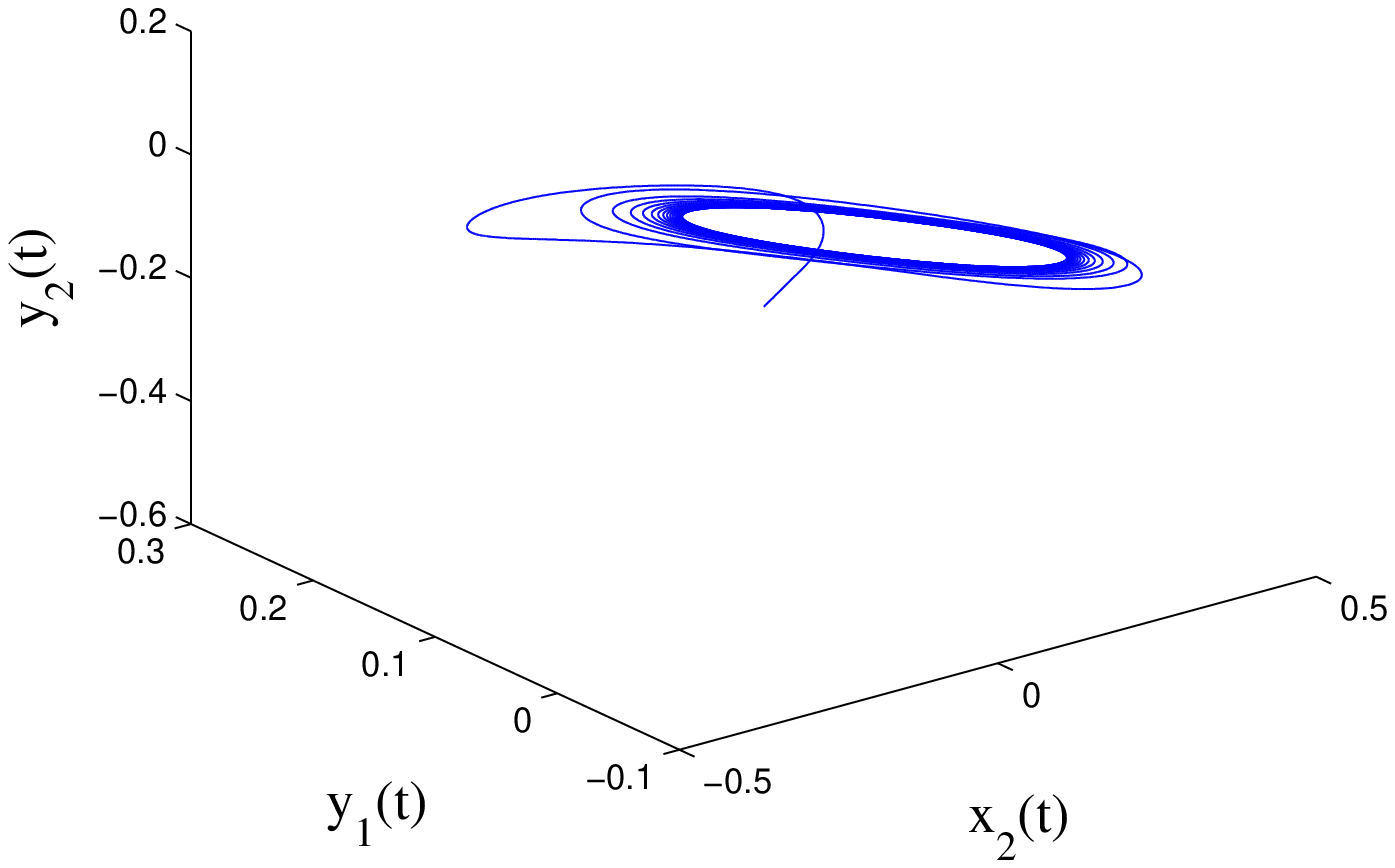}
\includegraphics[width=2in]{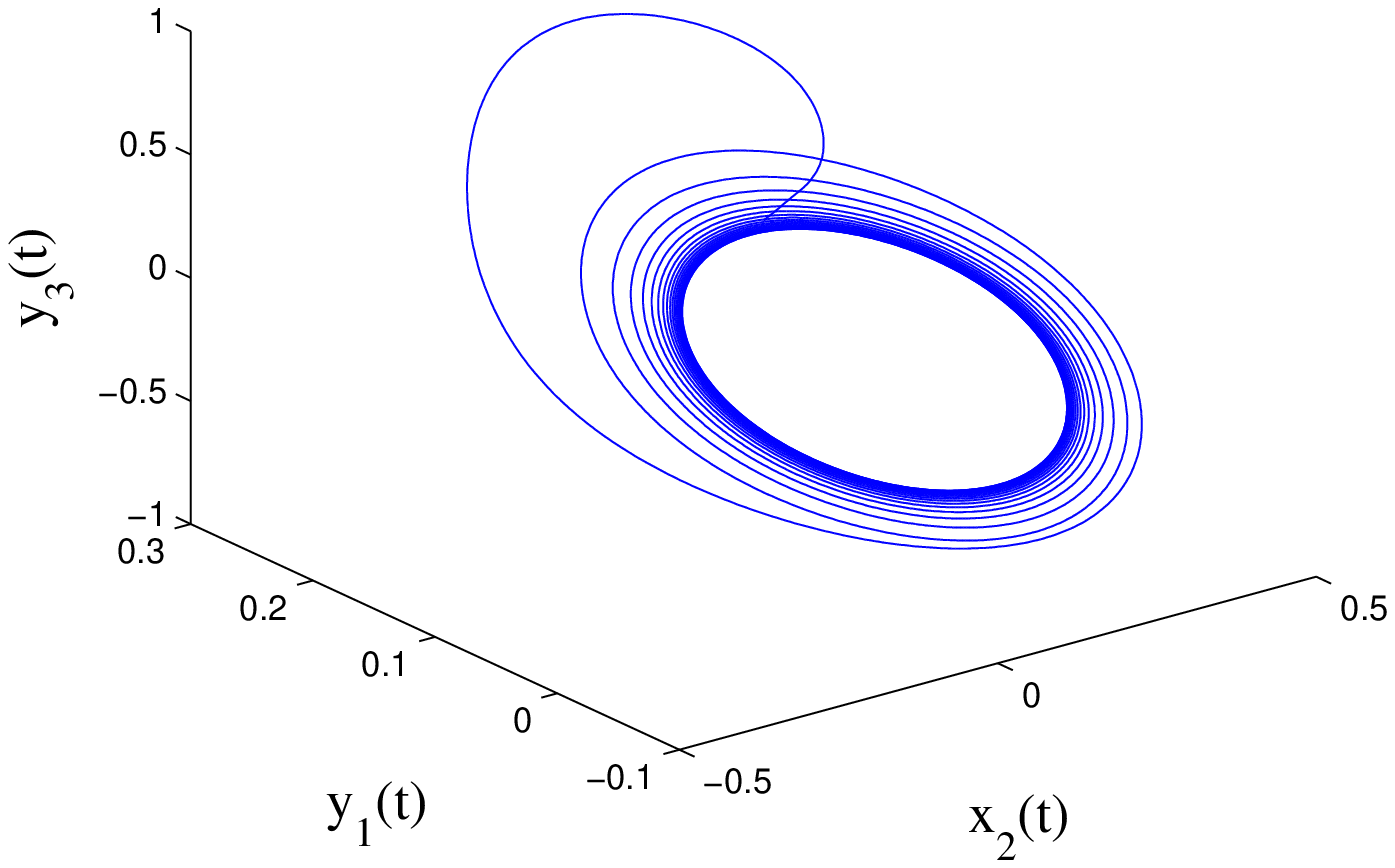}
\includegraphics[width=2in]{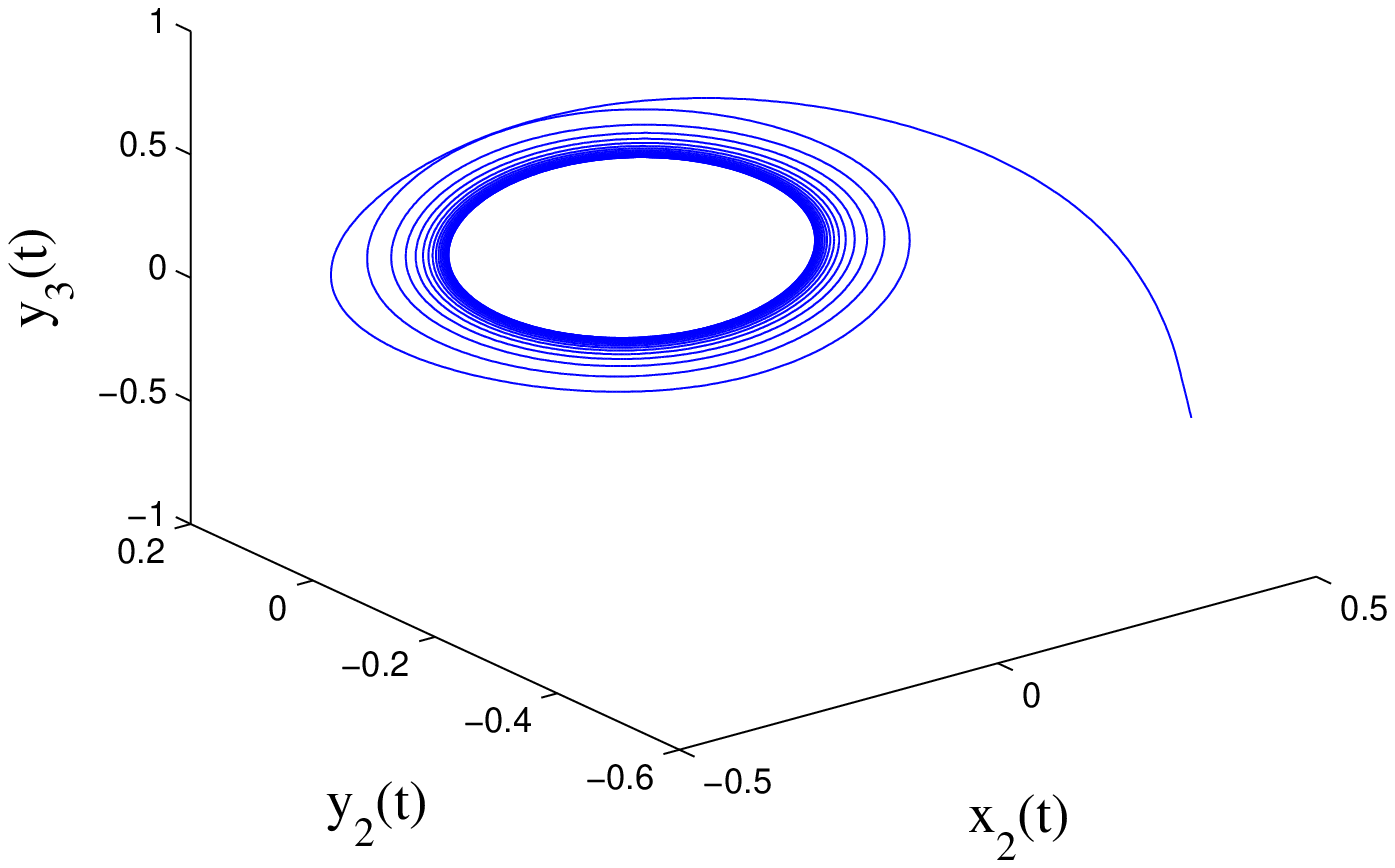}
\caption{\small{Phase diagrams of system \eqref{401} with $\theta=0.91$ and $\tau_3=0.15>\tau_0=0.1234$ projected on $x_1-y_1-y_2$, $x_1-y_1-y_3$, $x_1-y_2-y_3$, $x_2-y_1-y_2$, $x_2-y_1-y_3$ and $x_2-y_2-y_3$, respectively.}}
\label{fig4}
\end{figure}

\begin{figure}[H]
\centering
\includegraphics[width=2in]{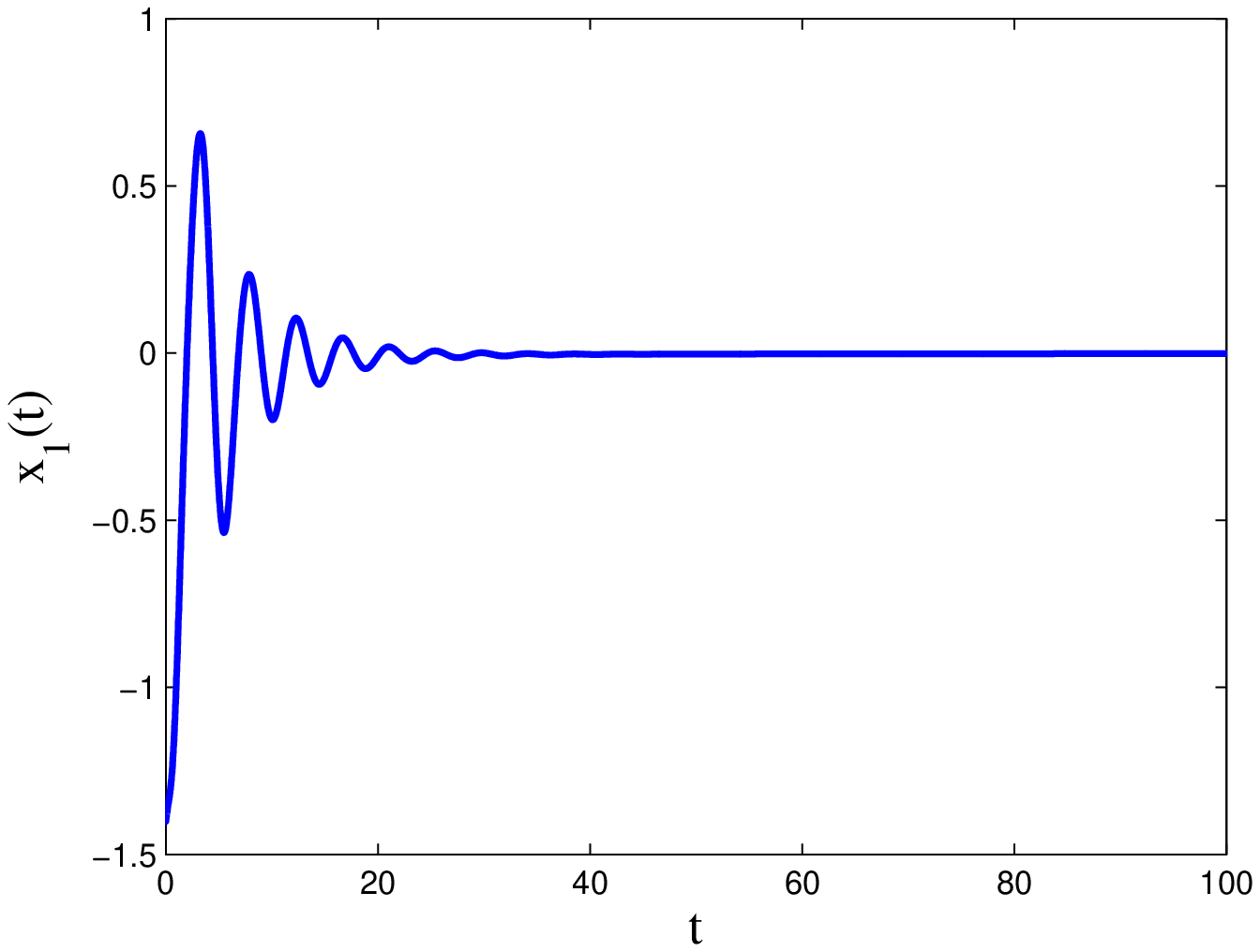}
\includegraphics[width=2in]{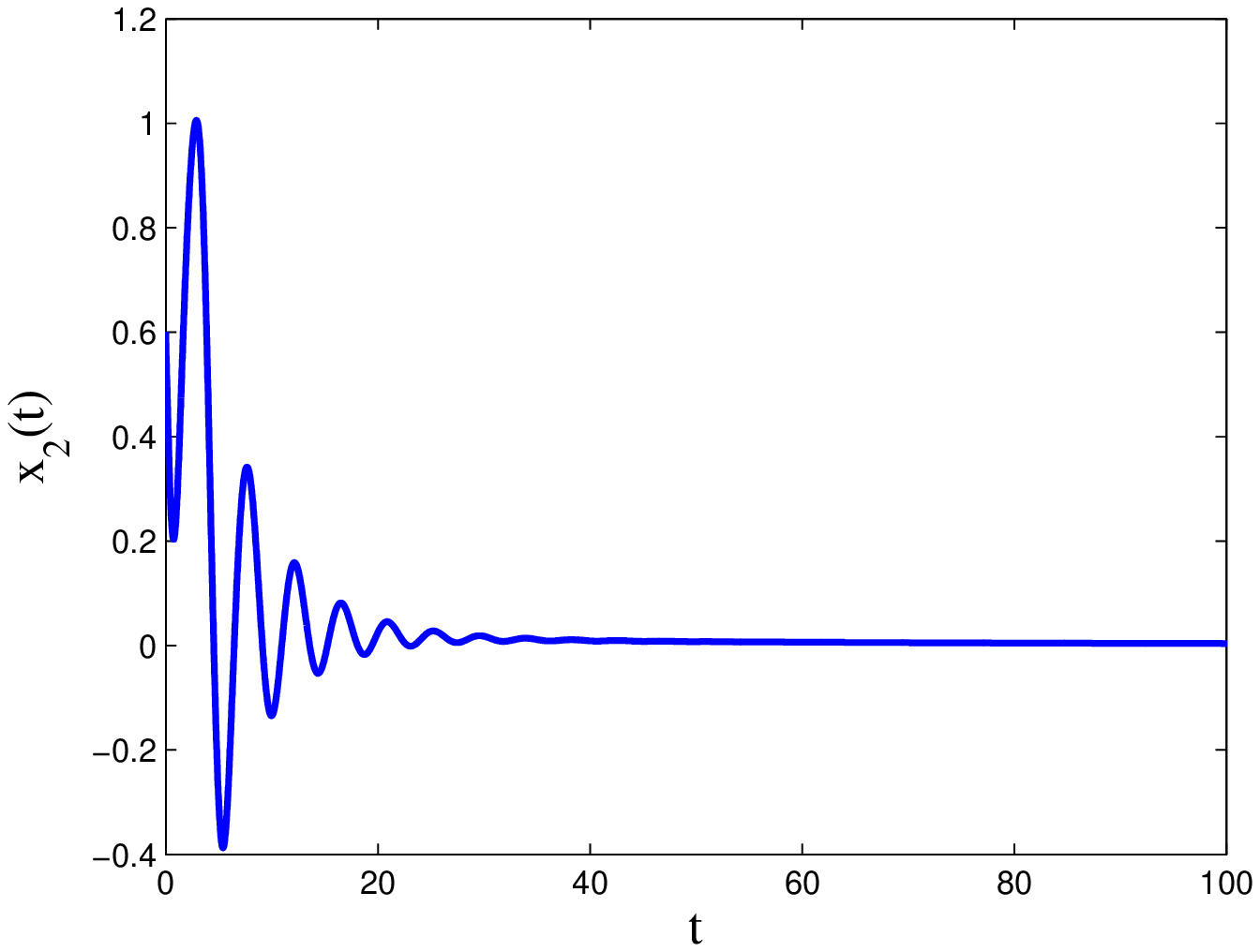}
\includegraphics[width=2in]{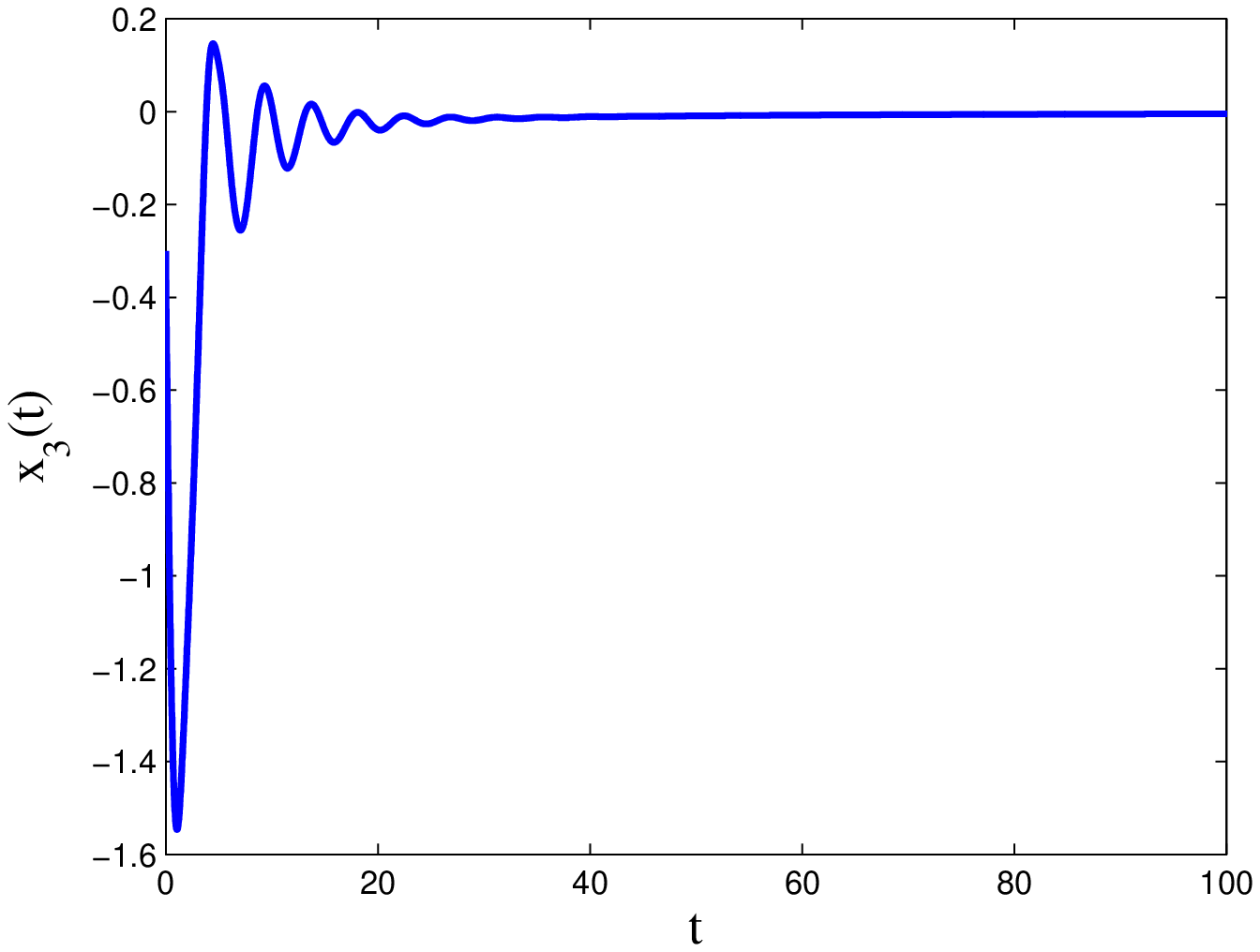}
\includegraphics[width=2in]{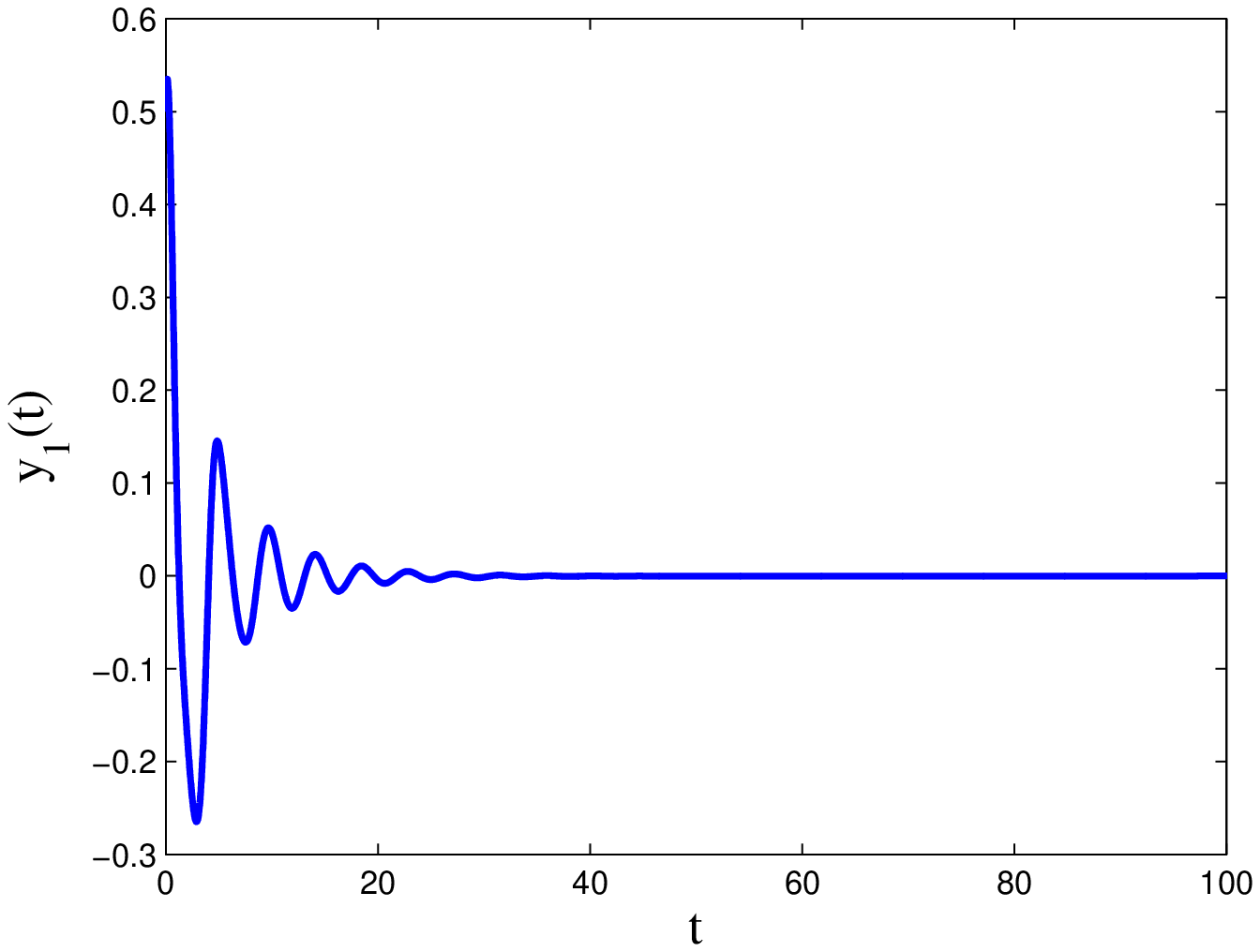}
\includegraphics[width=2in]{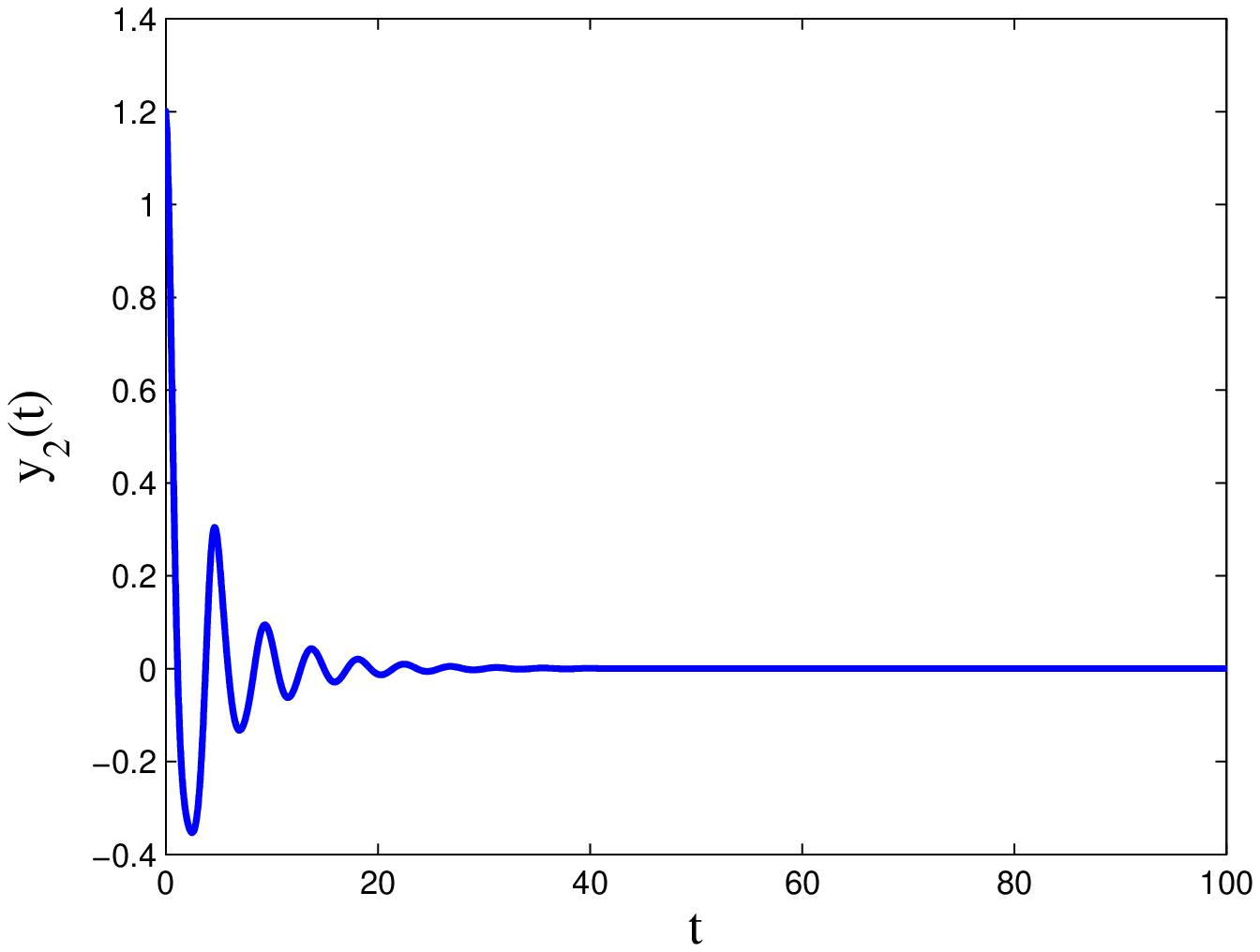}
\includegraphics[width=2in]{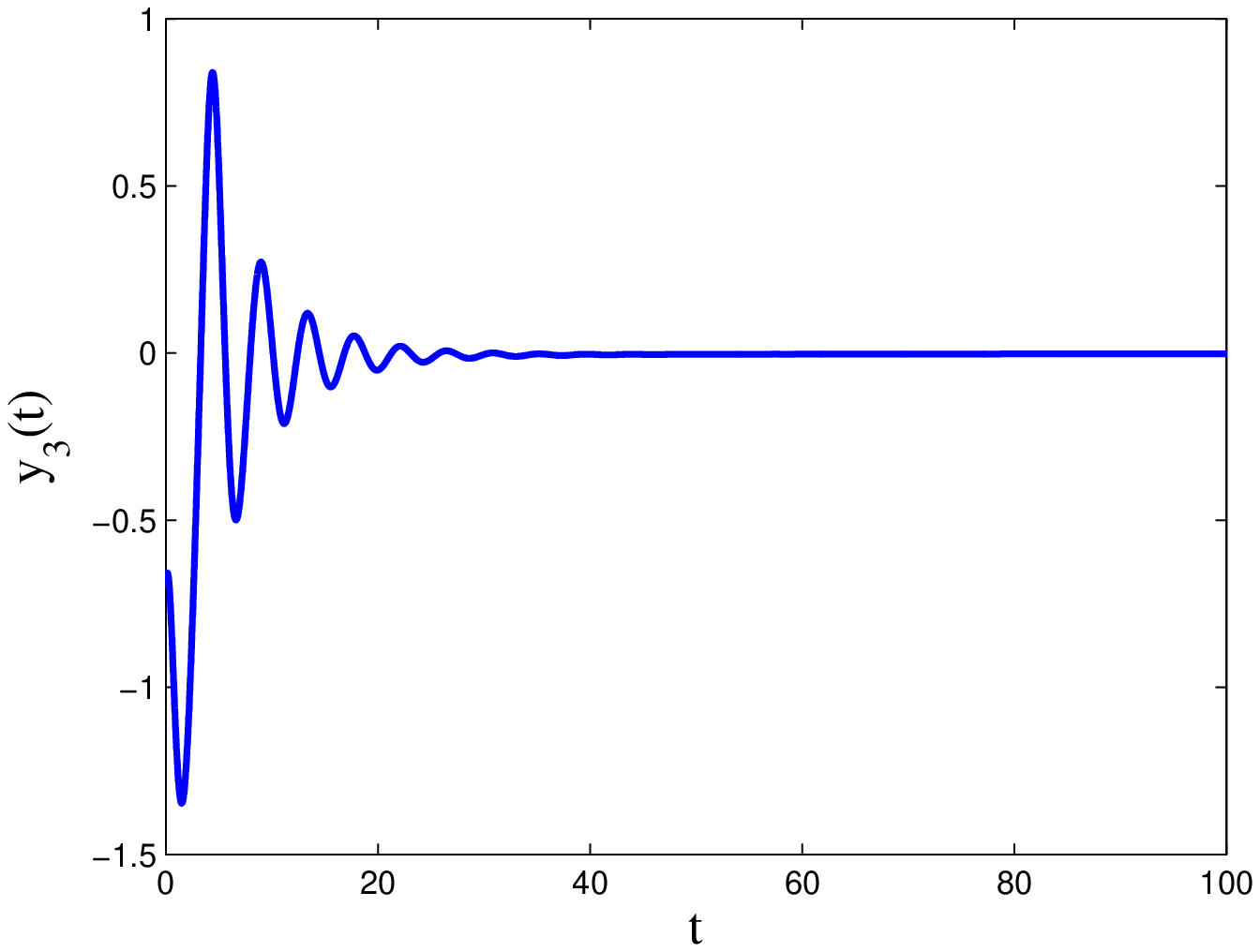}
\caption{\small{The temporal solutions of $x_1(t)$, $x_2(t)$, $x_3(t)$, $y_1(t)$, $y_2(t)$ and $y_3(t)$ versus $t$ of system \eqref{401} with $\theta=0.91$ and $\tau_4=0.2<\tau_0^*=0.2613$.}}
\label{fig5}
\end{figure}

\noindent Fig.\ref{fig5} and Fig.\ref{fig6}, we observe that the zero equilibrium point is locally asymptotically stable when $\tau_4=0.2<\tau_0^*=0.2613$, while, Fig.\ref{fig7} and Fig.\ref{fig8} indicate that the zero equilibrium point lose stability and Hopf bifurcation occurs when $\tau_4=0.36>\tau_0^*=0.2613$. Because $\tau_2=0.06<\tau_0=0.0681$, communication delay only is not strong enough to destabilize system \eqref{401}. But as $\tau_4$ increases from $0.2$ to $0.36$, namely, leakage delay $\tau_1$ increases from $0.46$ to $0.78$, system \eqref{401} occurs a Hopf bifurcation, which illustrates the destabilizing impact of leakage delay directly.

\begin{figure}[H]
\centering
\includegraphics[width=2in]{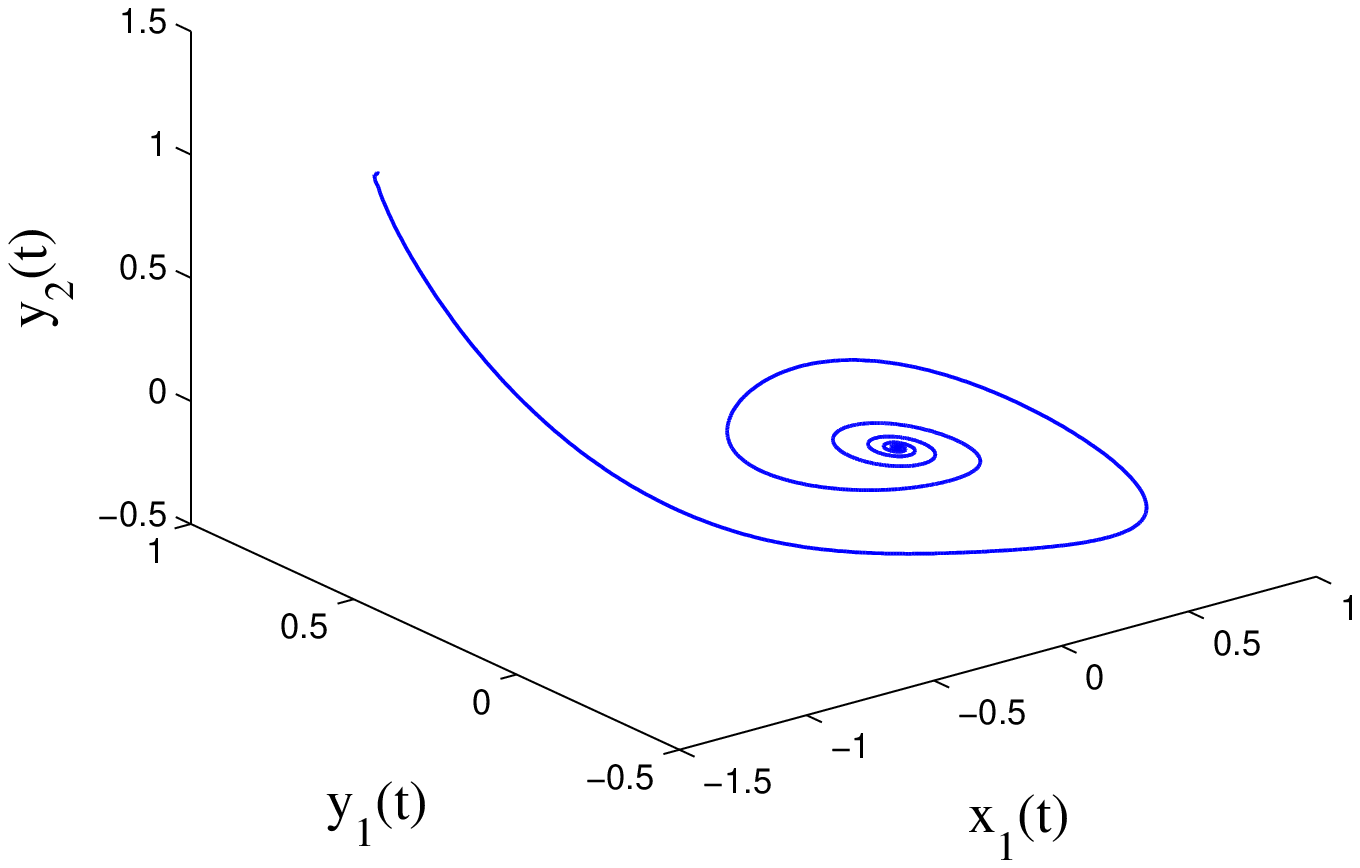}
\includegraphics[width=2in]{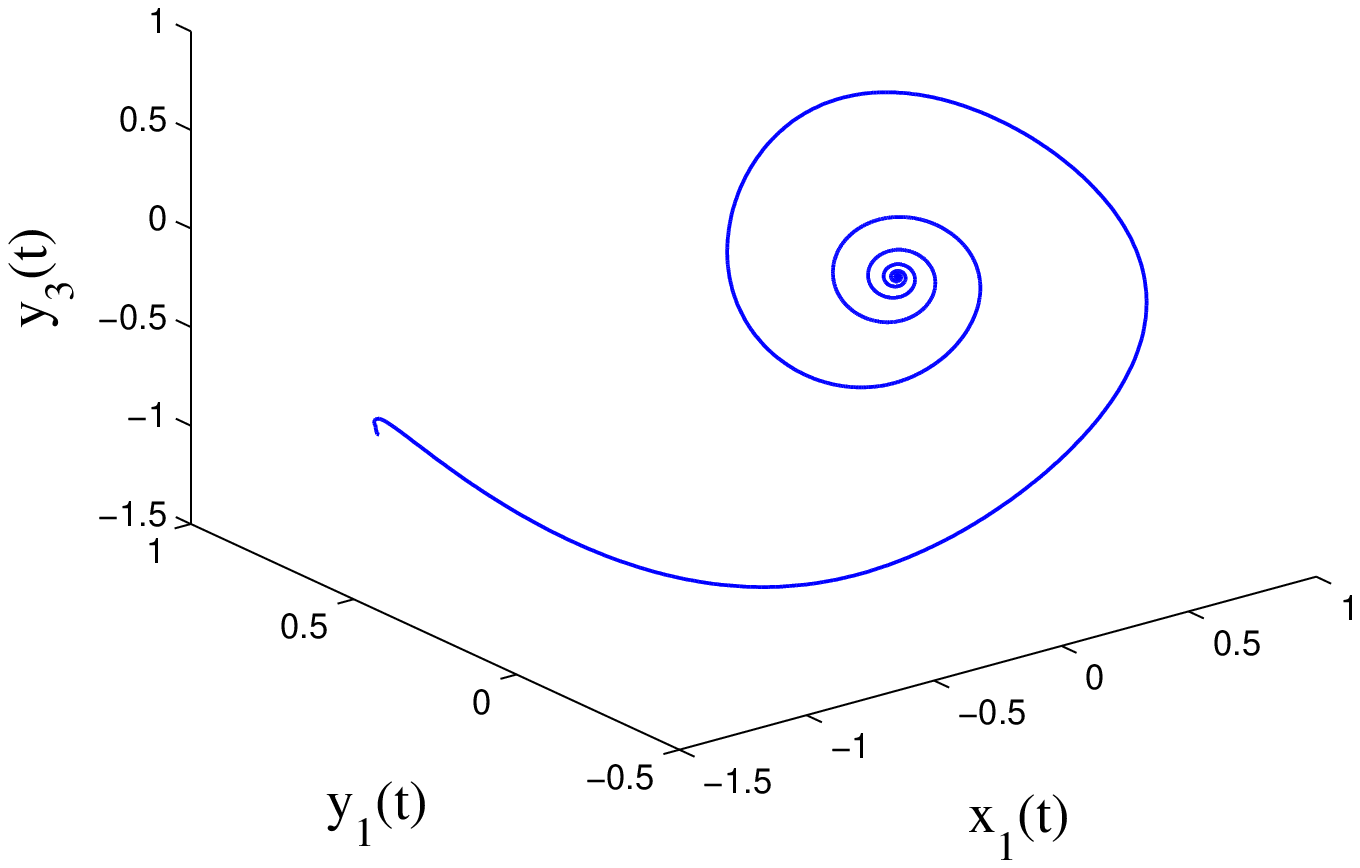}
\includegraphics[width=2in]{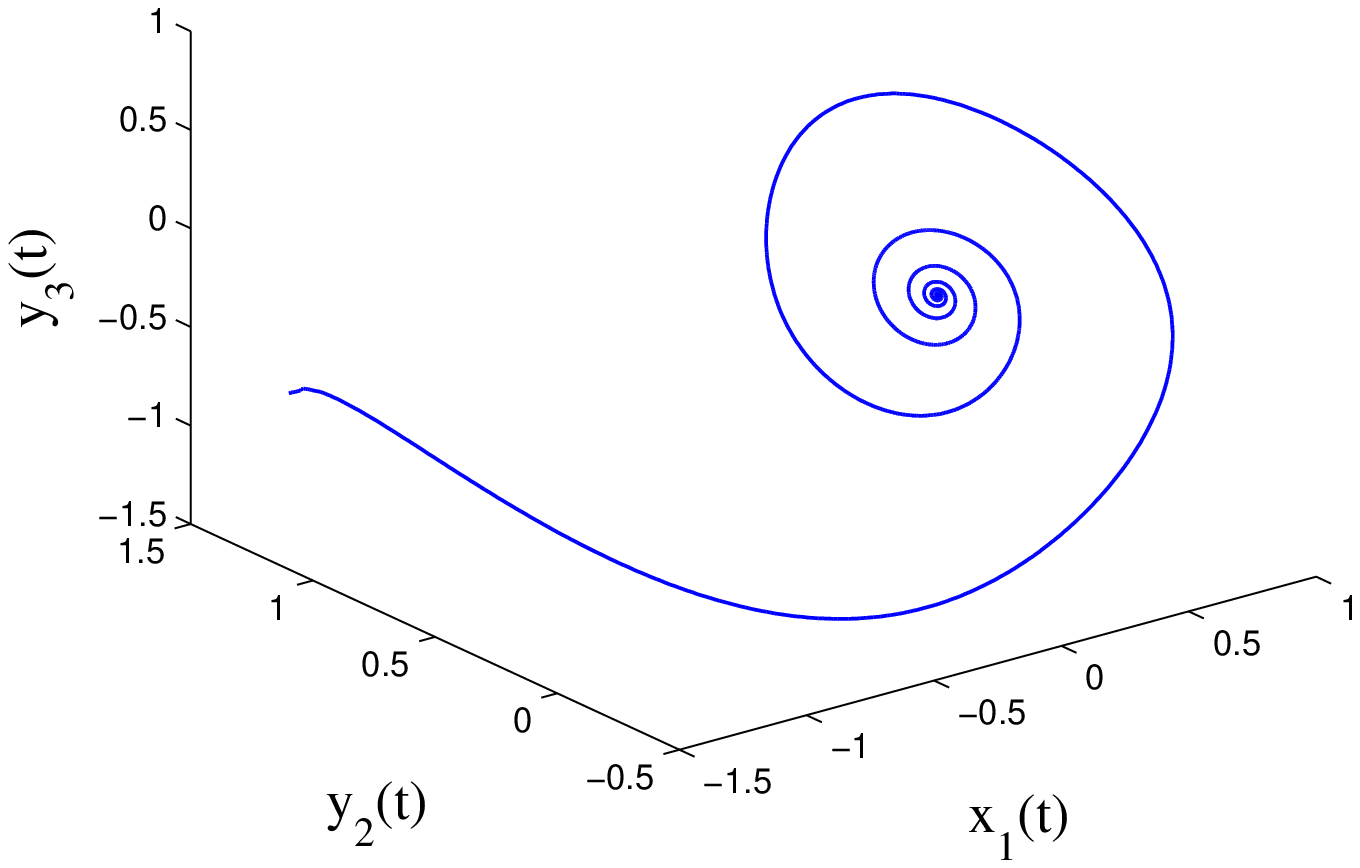}
\includegraphics[width=2in]{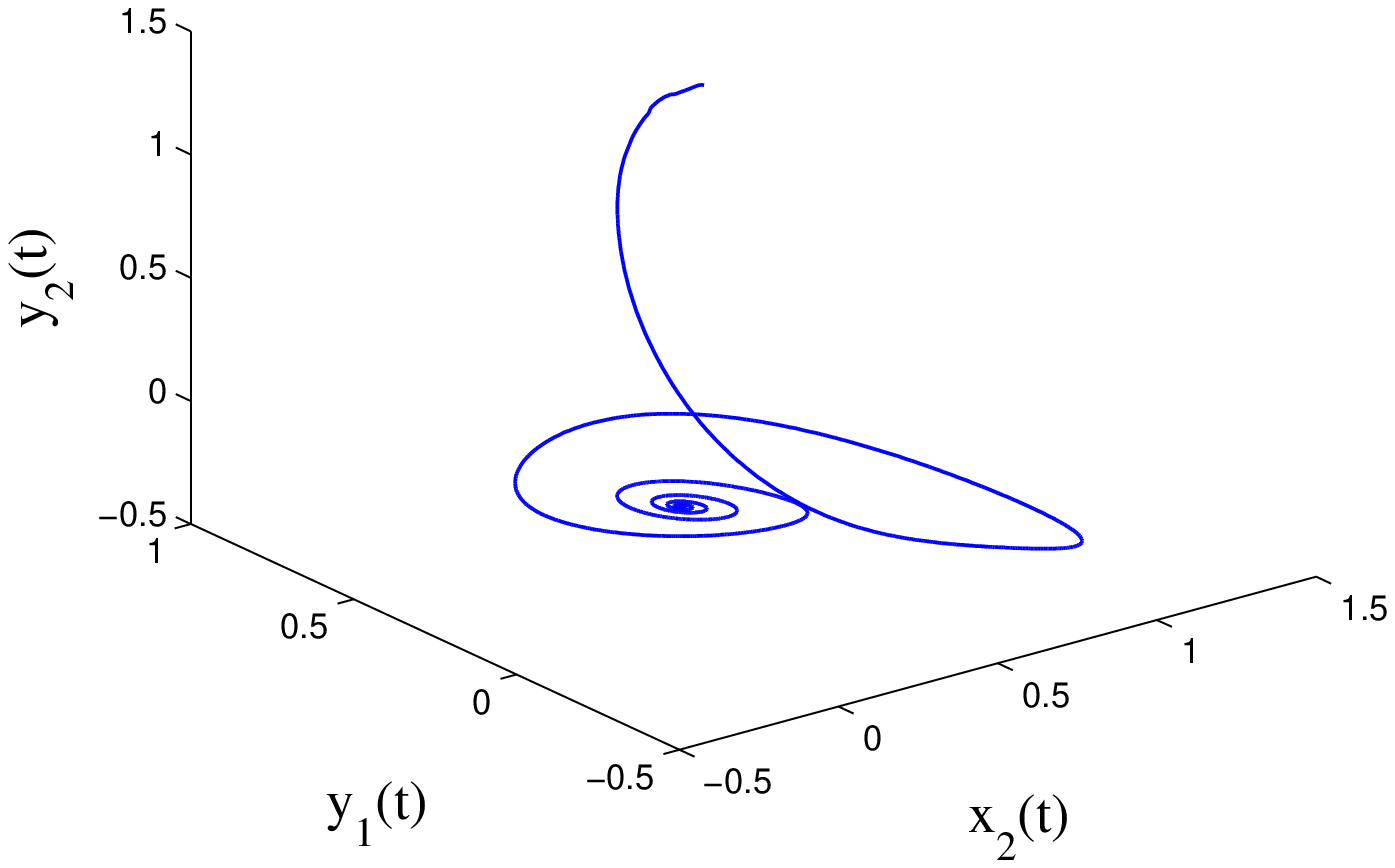}
\includegraphics[width=2in]{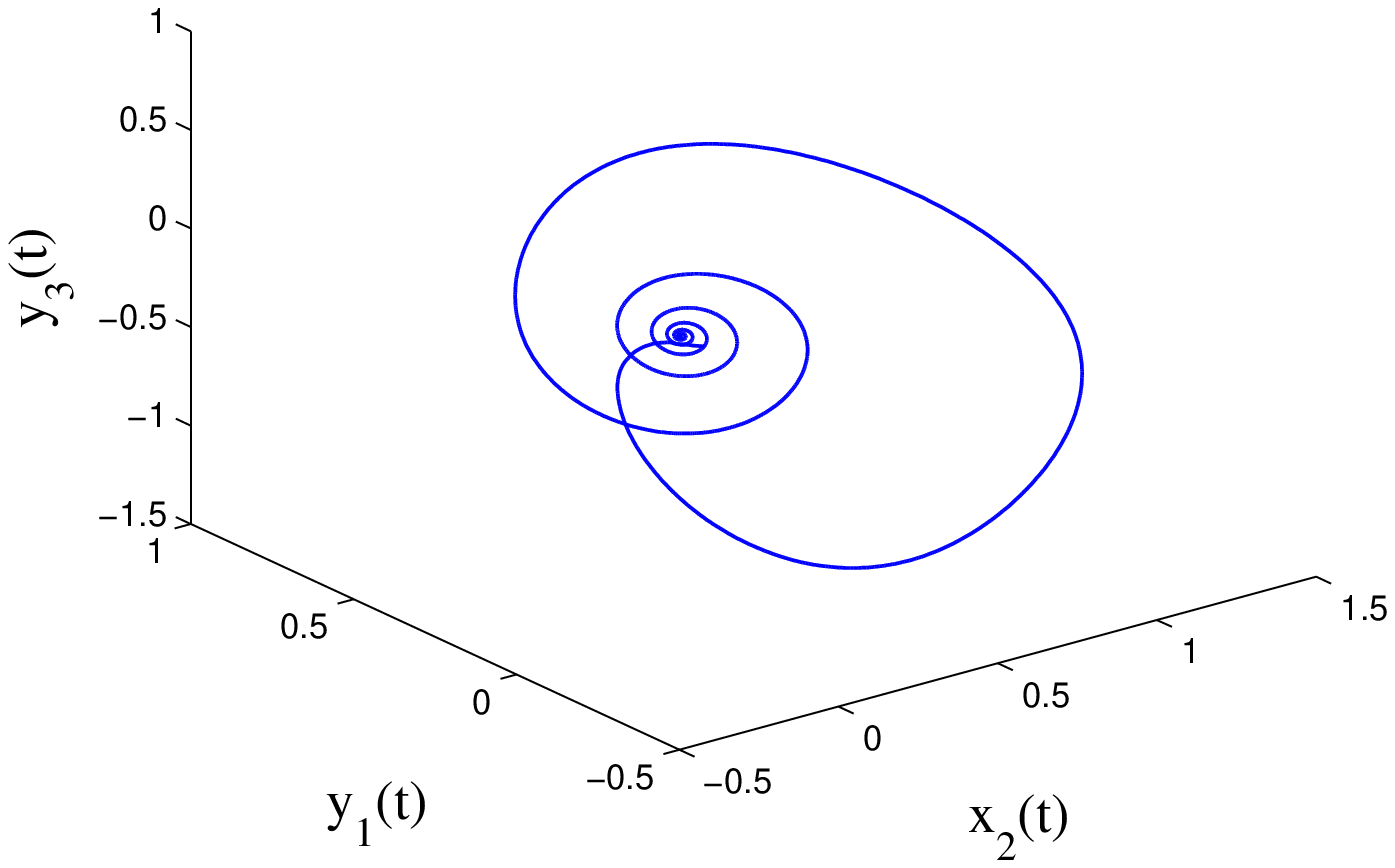}
\includegraphics[width=2in]{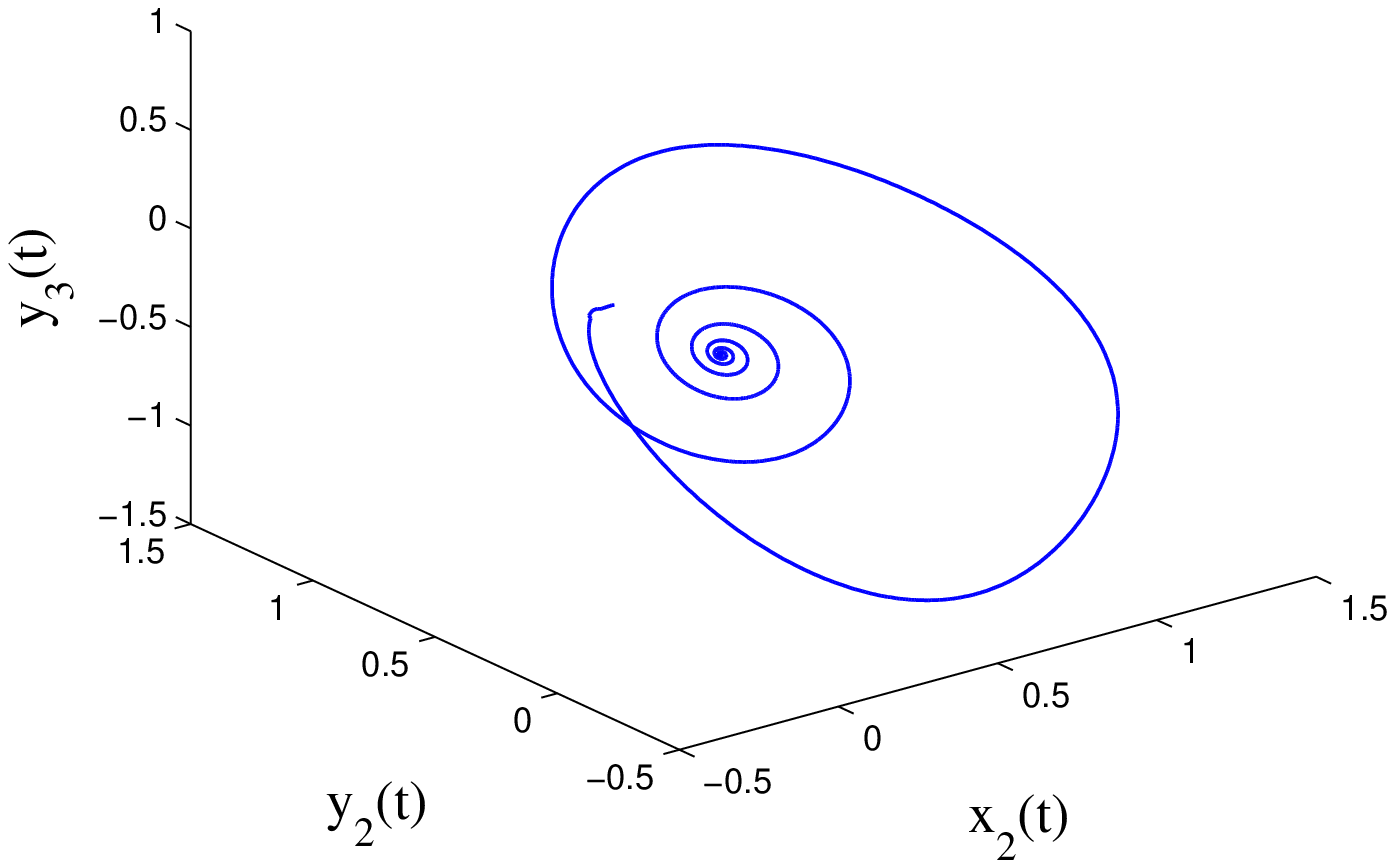}
\caption{\small{Phase diagrams of system \eqref{401} with $\theta=0.91$ and $\tau_4=0.2<\tau_0^*=0.2613$ projected on $x_1-y_1-y_2$, $x_1-y_1-y_3$, $x_1-y_2-y_3$, $x_2-y_1-y_2$, $x_2-y_1-y_3$ and $x_2-y_2-y_3$, respectively.}}
\label{fig6}
\end{figure}

\begin{figure}[H]
\centering
\includegraphics[width=2in]{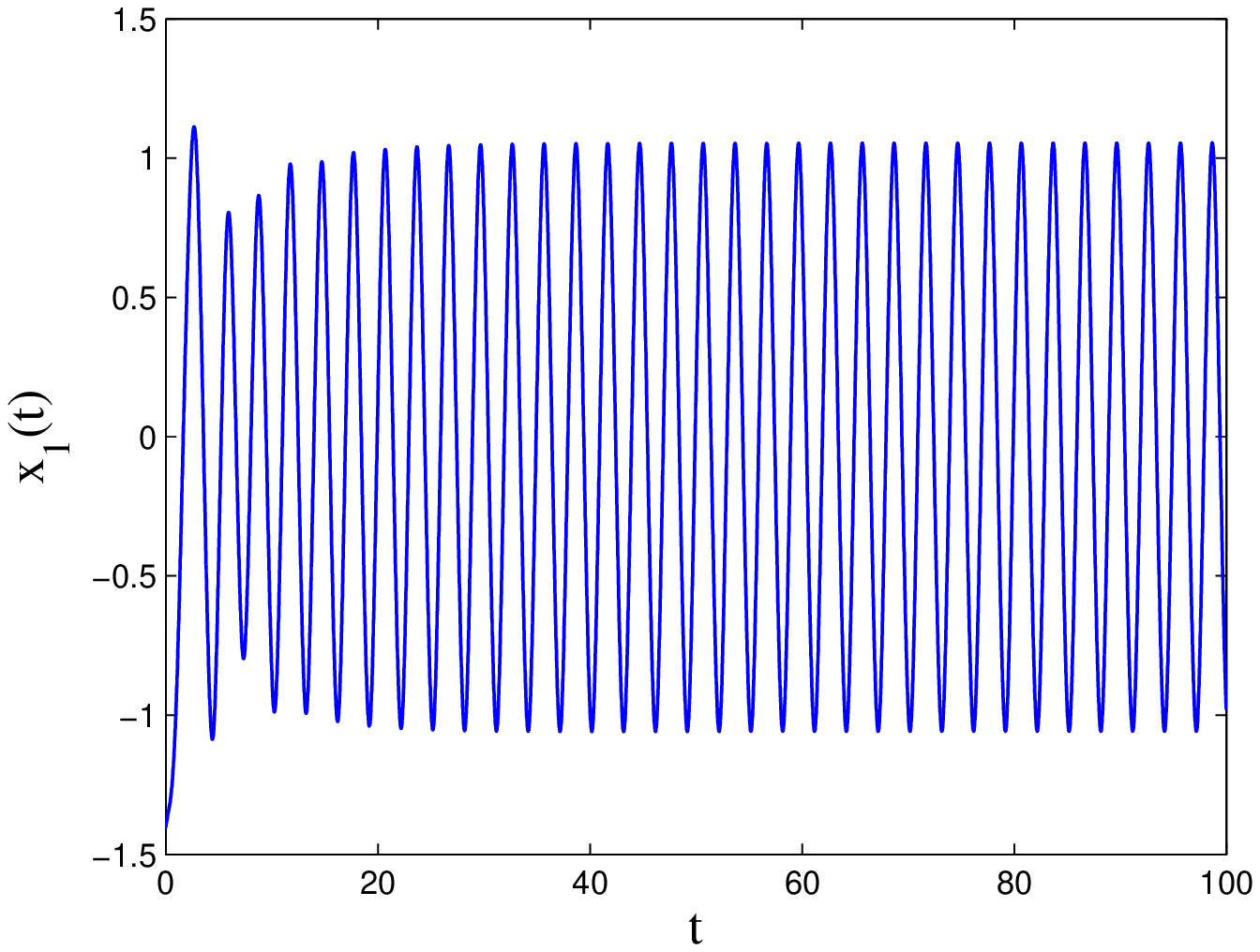}
\includegraphics[width=2in]{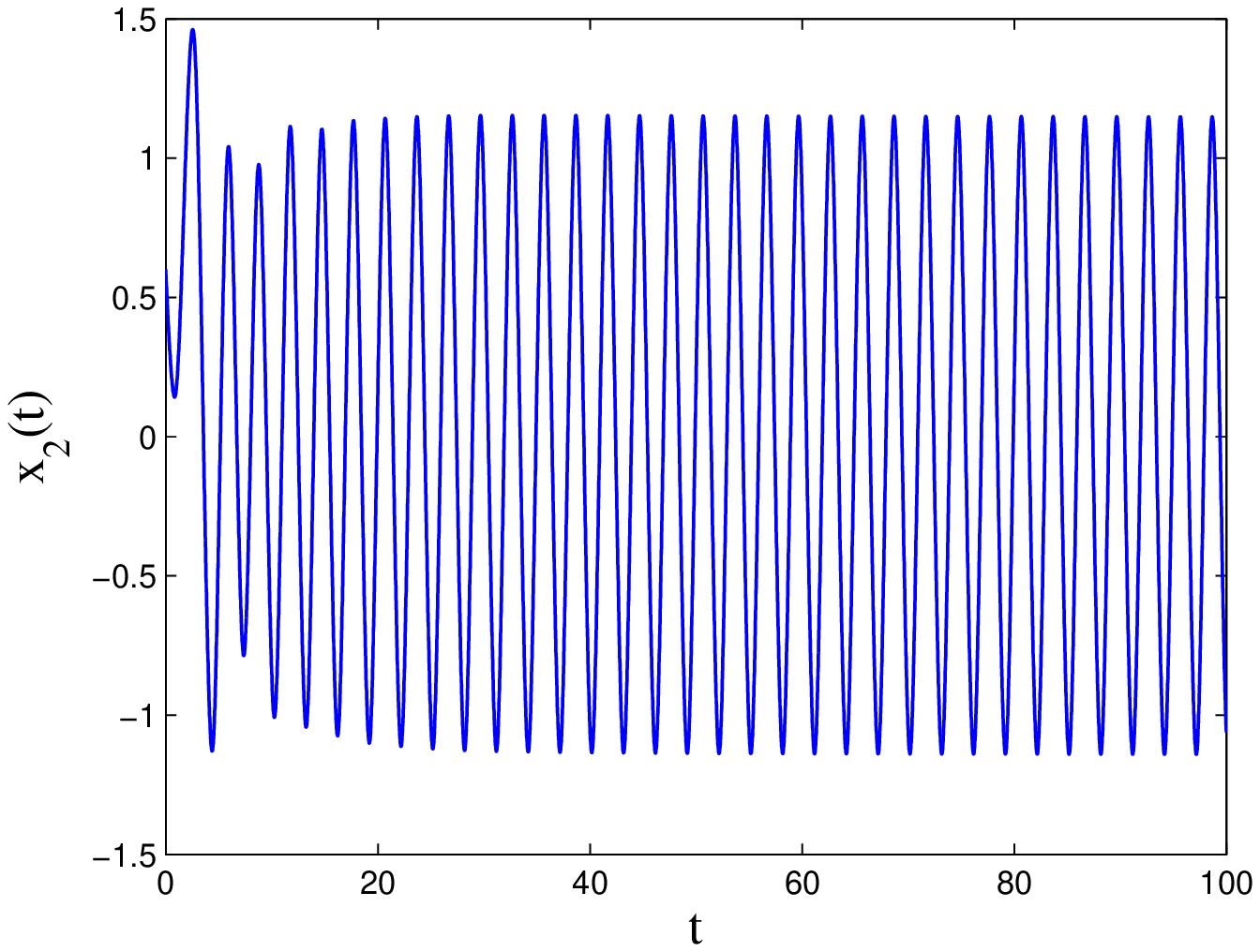}
\includegraphics[width=2in]{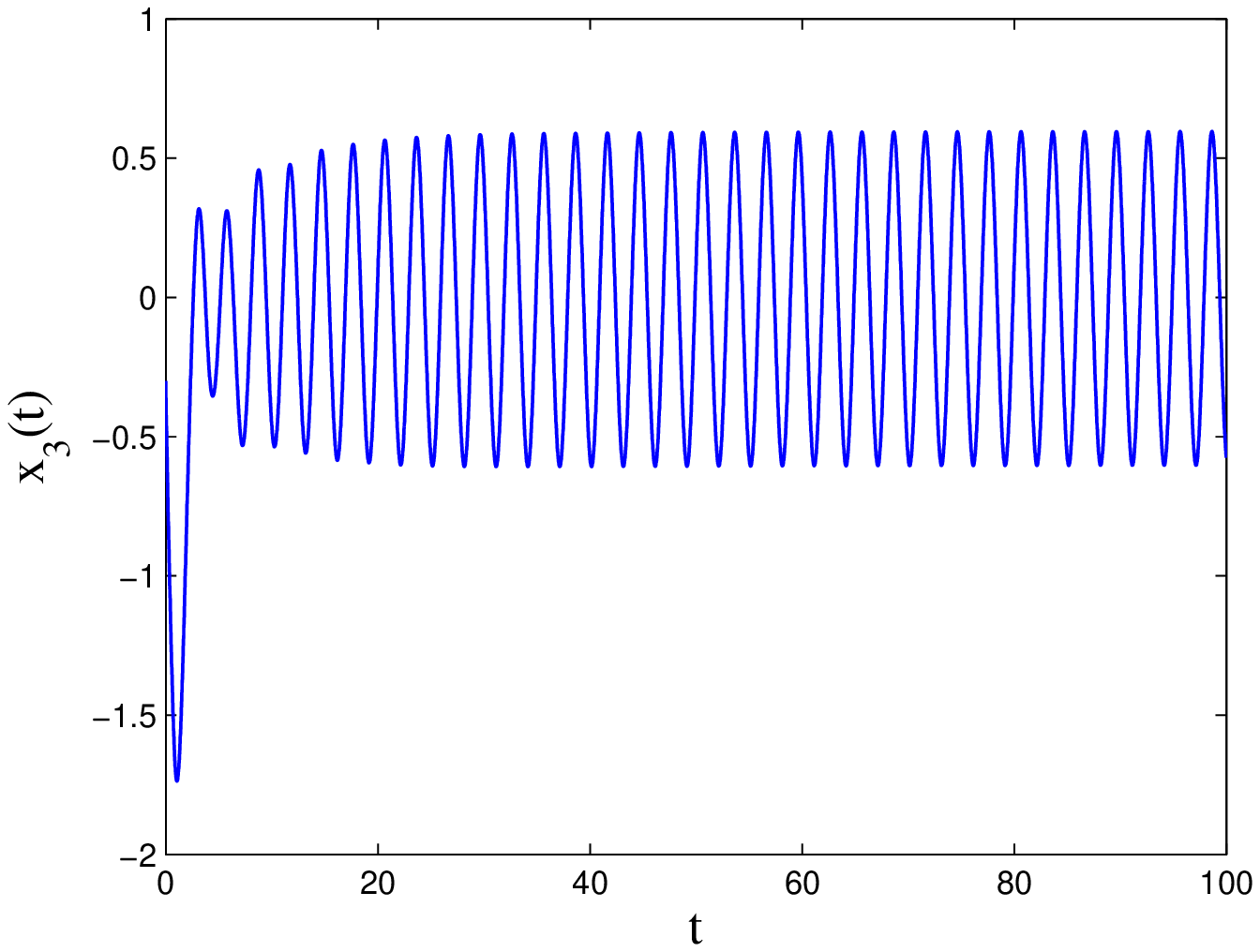}
\includegraphics[width=2in]{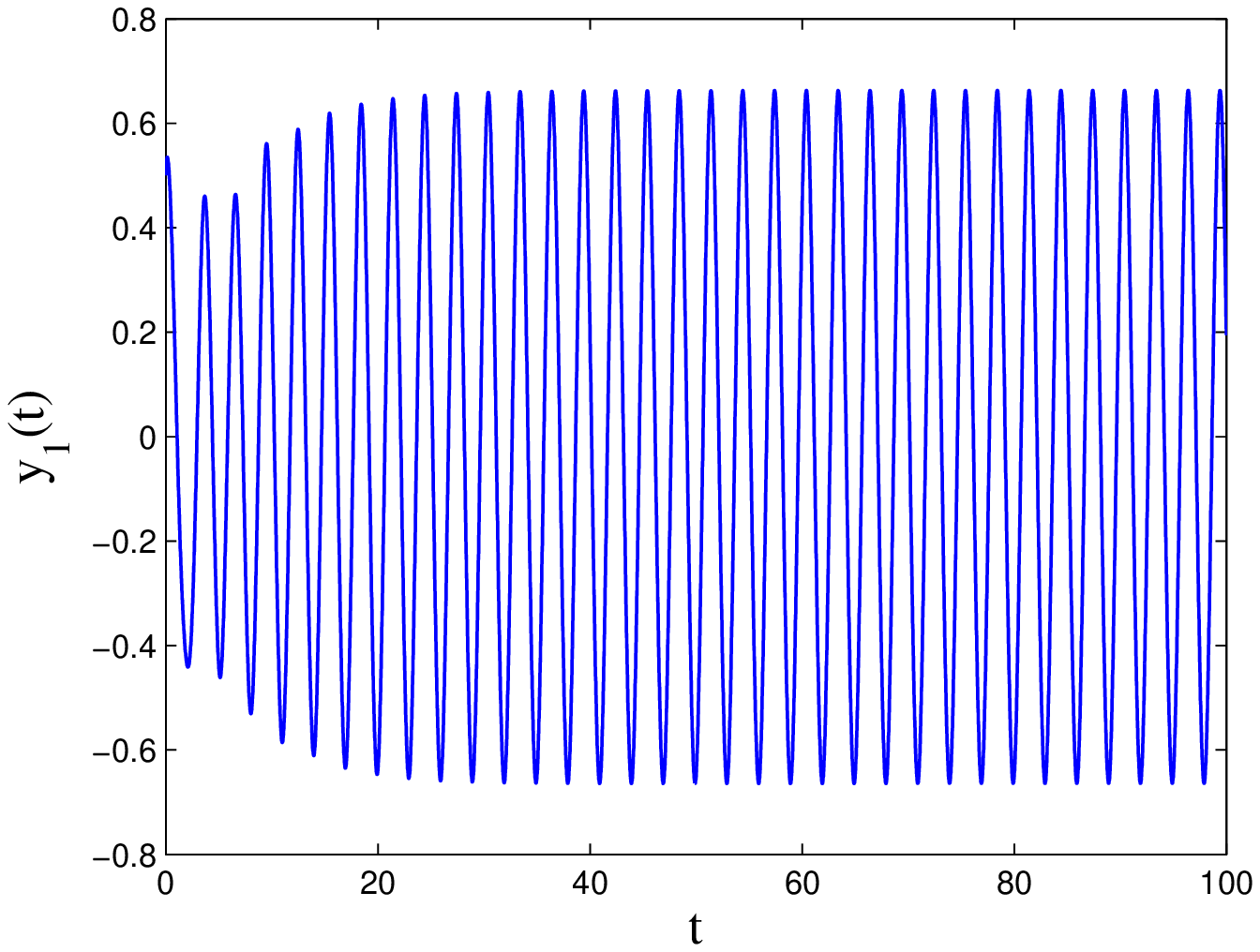}
\includegraphics[width=2in]{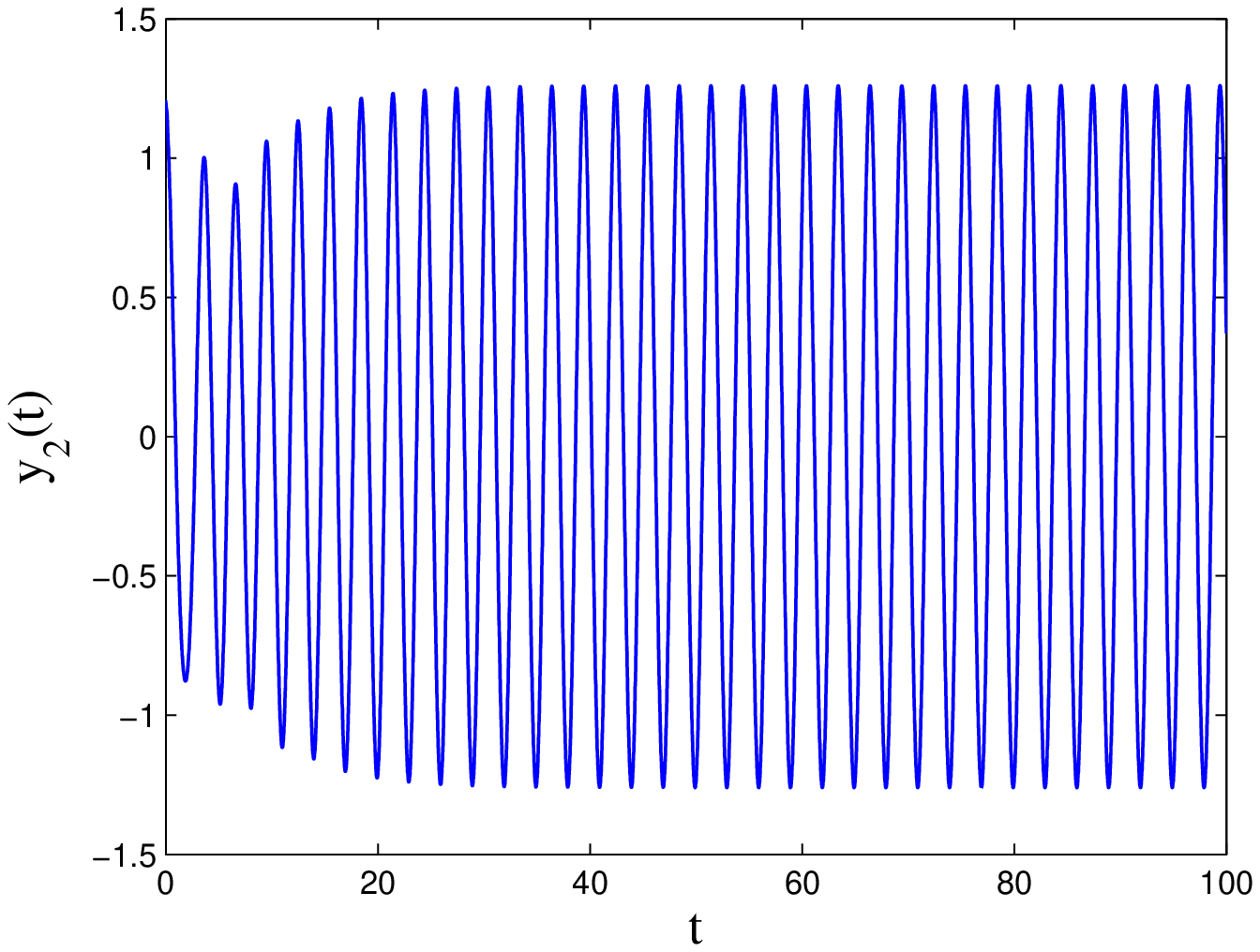}
\includegraphics[width=2in]{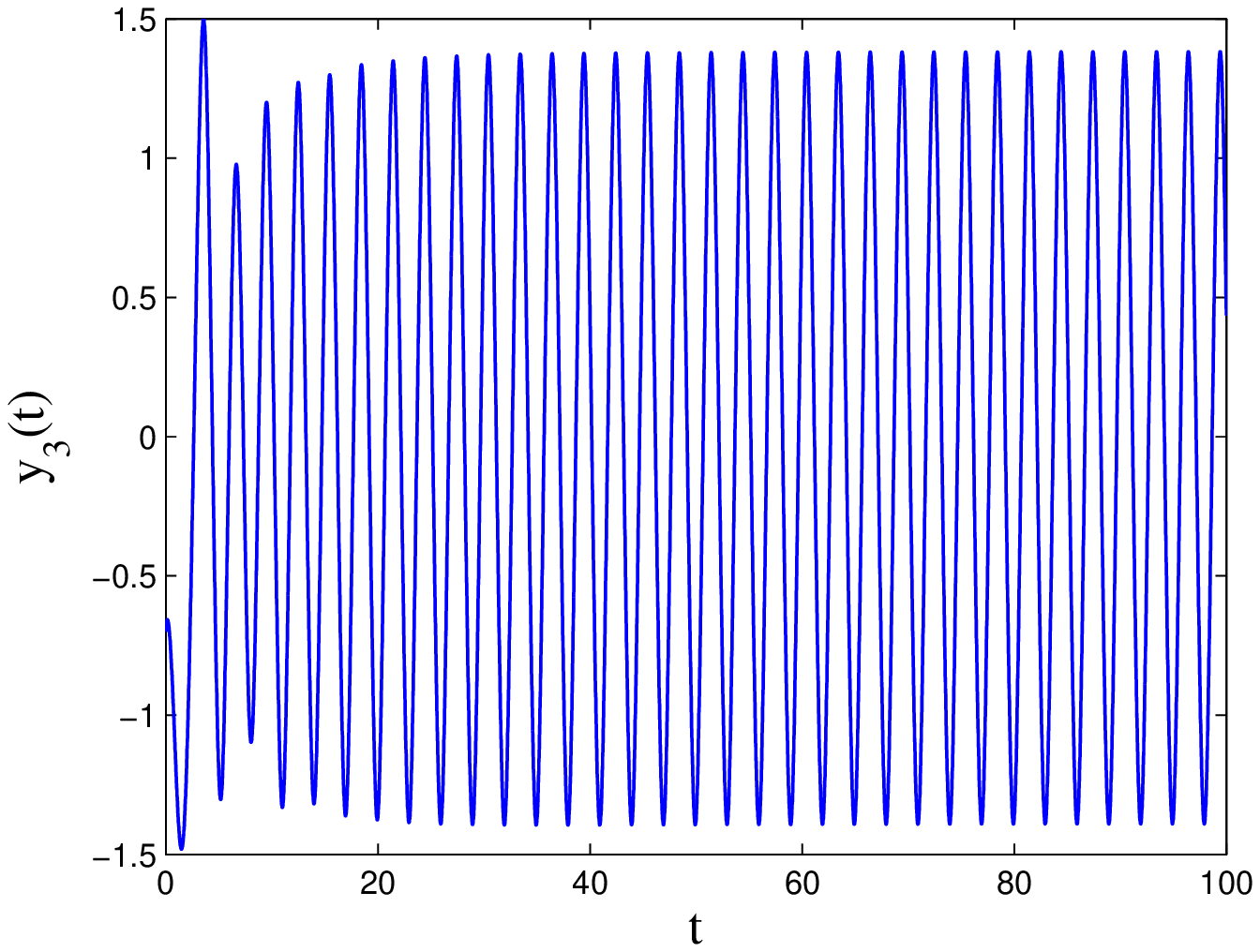}
\caption{\small{The temporal solutions of $x_1(t)$, $x_2(t)$, $x_3(t)$, $y_1(t)$, $y_2(t)$ and $y_3(t)$ versus $t$ of system \eqref{401} with $\theta=0.91$ and $\tau_4=0.36>\tau_0^*=0.2613$.}}
\label{fig7}
\end{figure}

Besides, to reveal the effects of fractional order on the critical frequencies and bifurcation points, we use the system parameters which are the same to those in Example 1. Fig.\ref{fig9} implies that, when the fractional order $\theta$ increases from $0.5$ to $1.0$, the critical frequency $\omega_0$

\begin{figure}[H]
\centering
\includegraphics[width=2in]{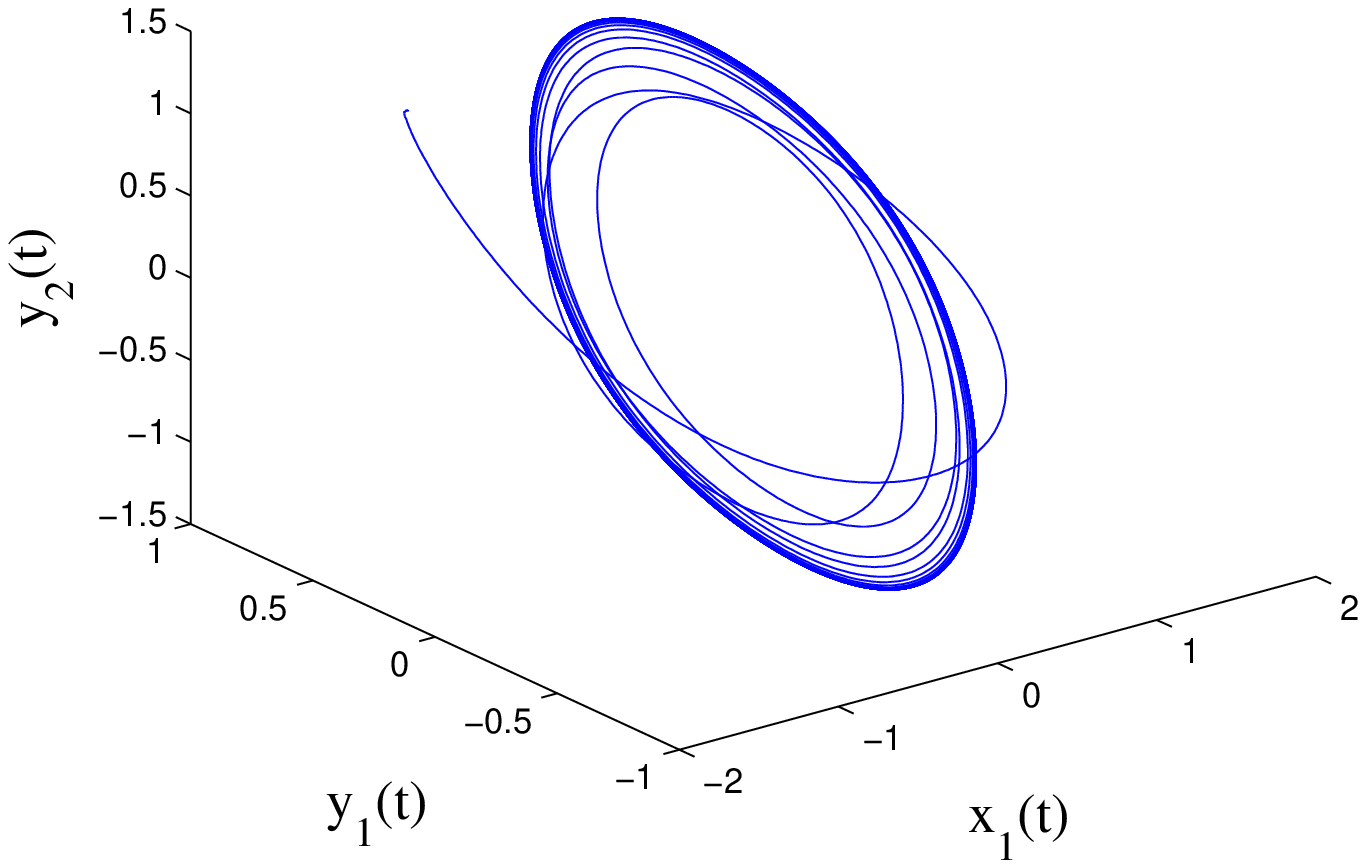}
\includegraphics[width=2in]{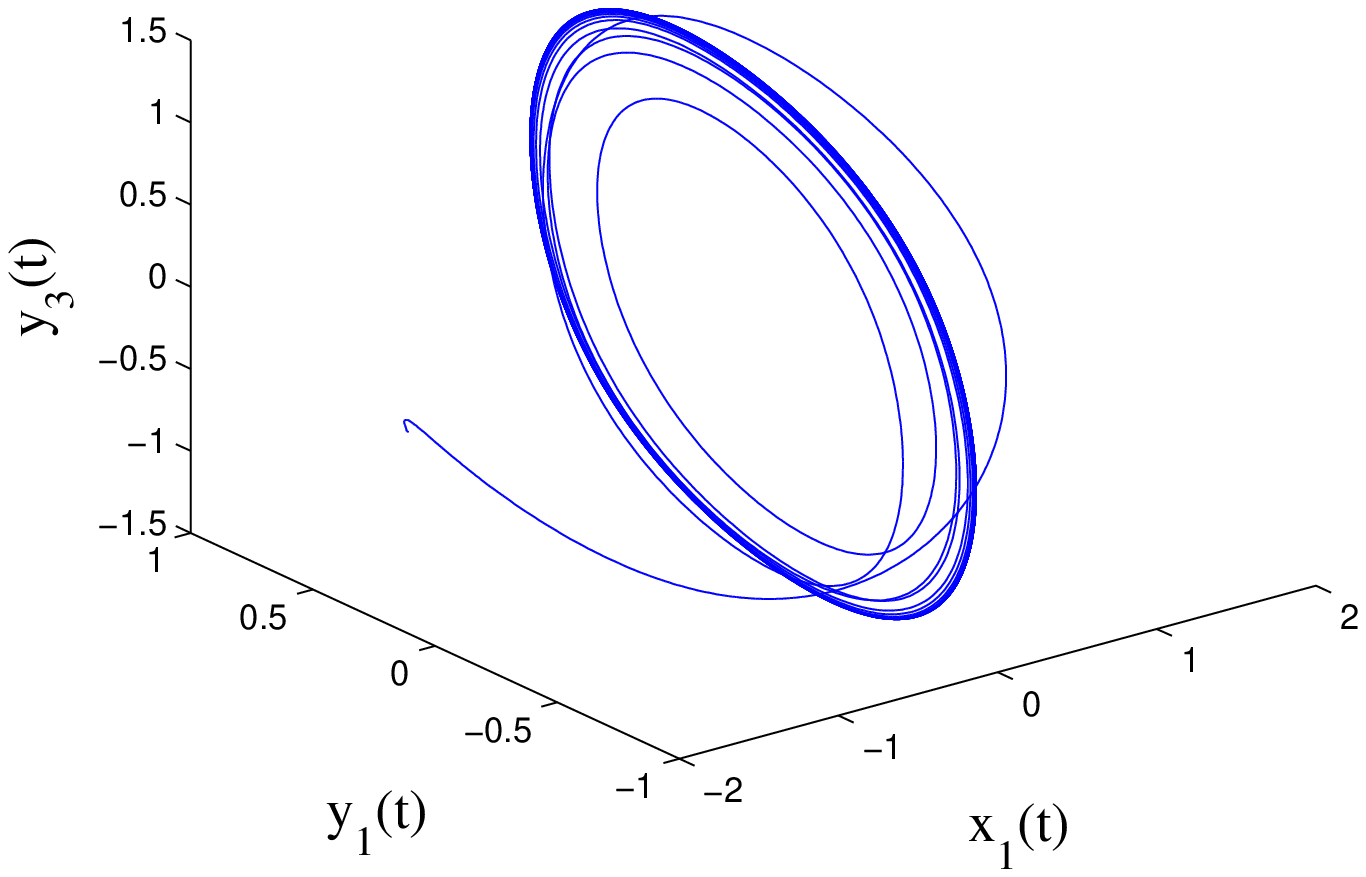}
\includegraphics[width=2in]{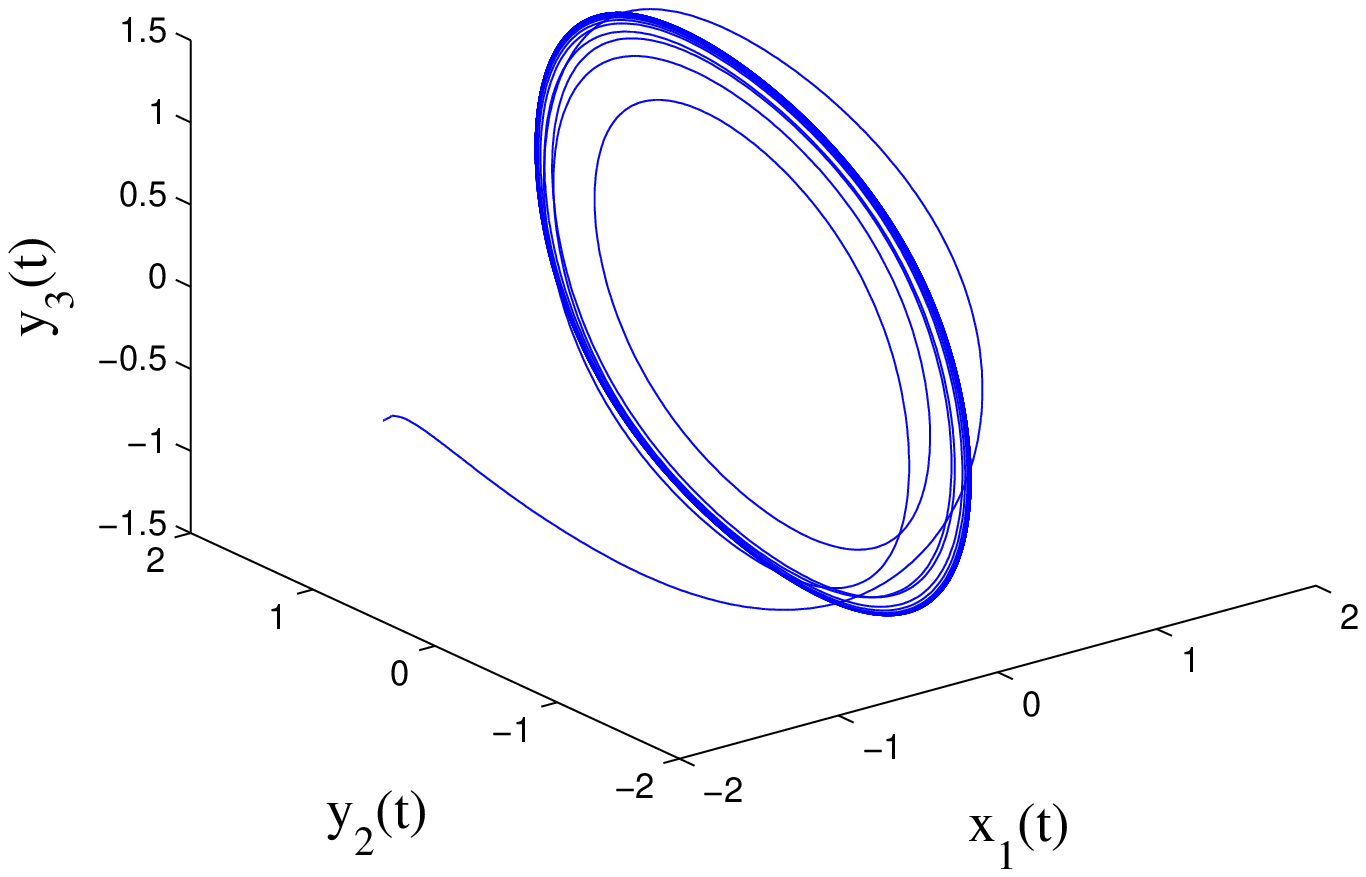}
\includegraphics[width=2in]{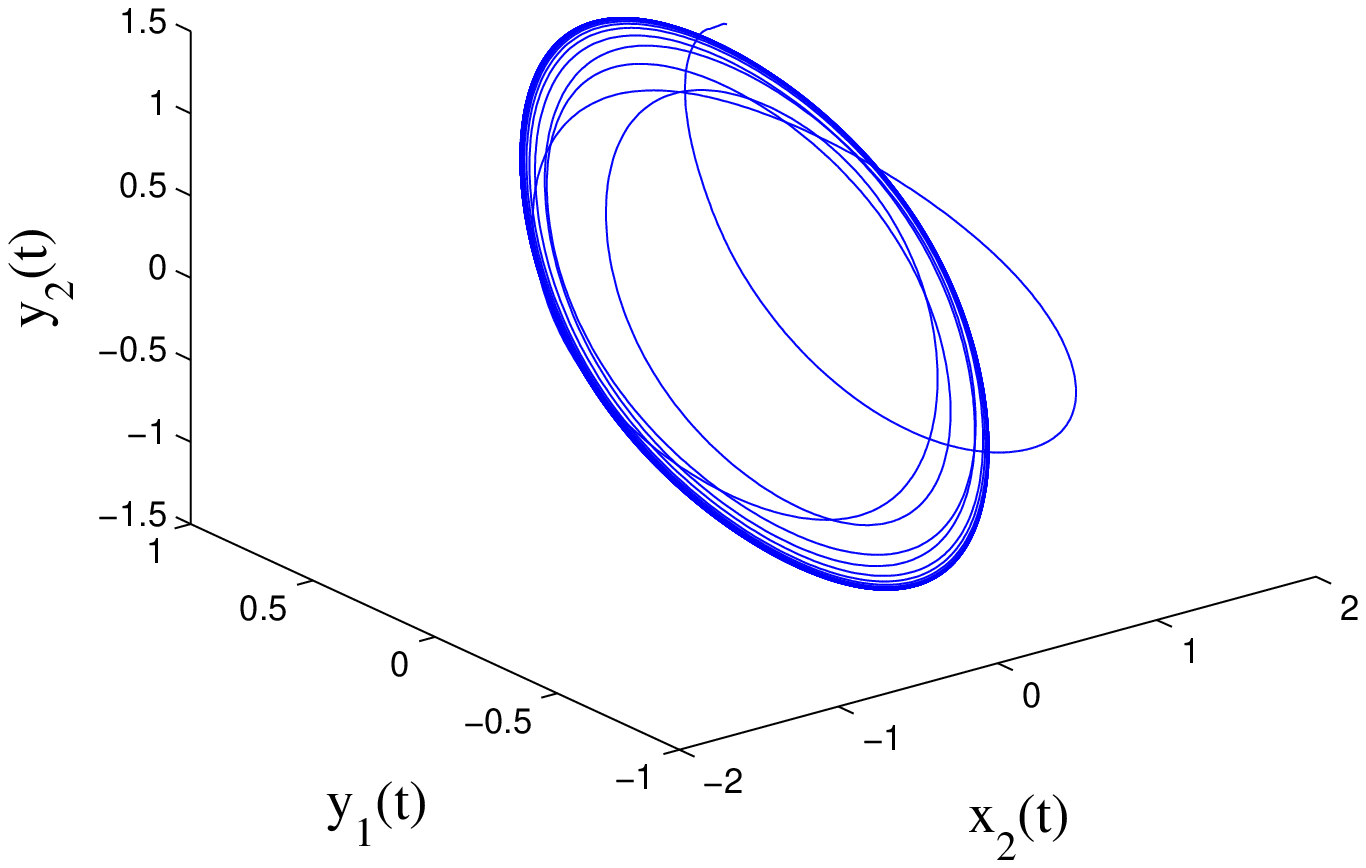}
\includegraphics[width=2in]{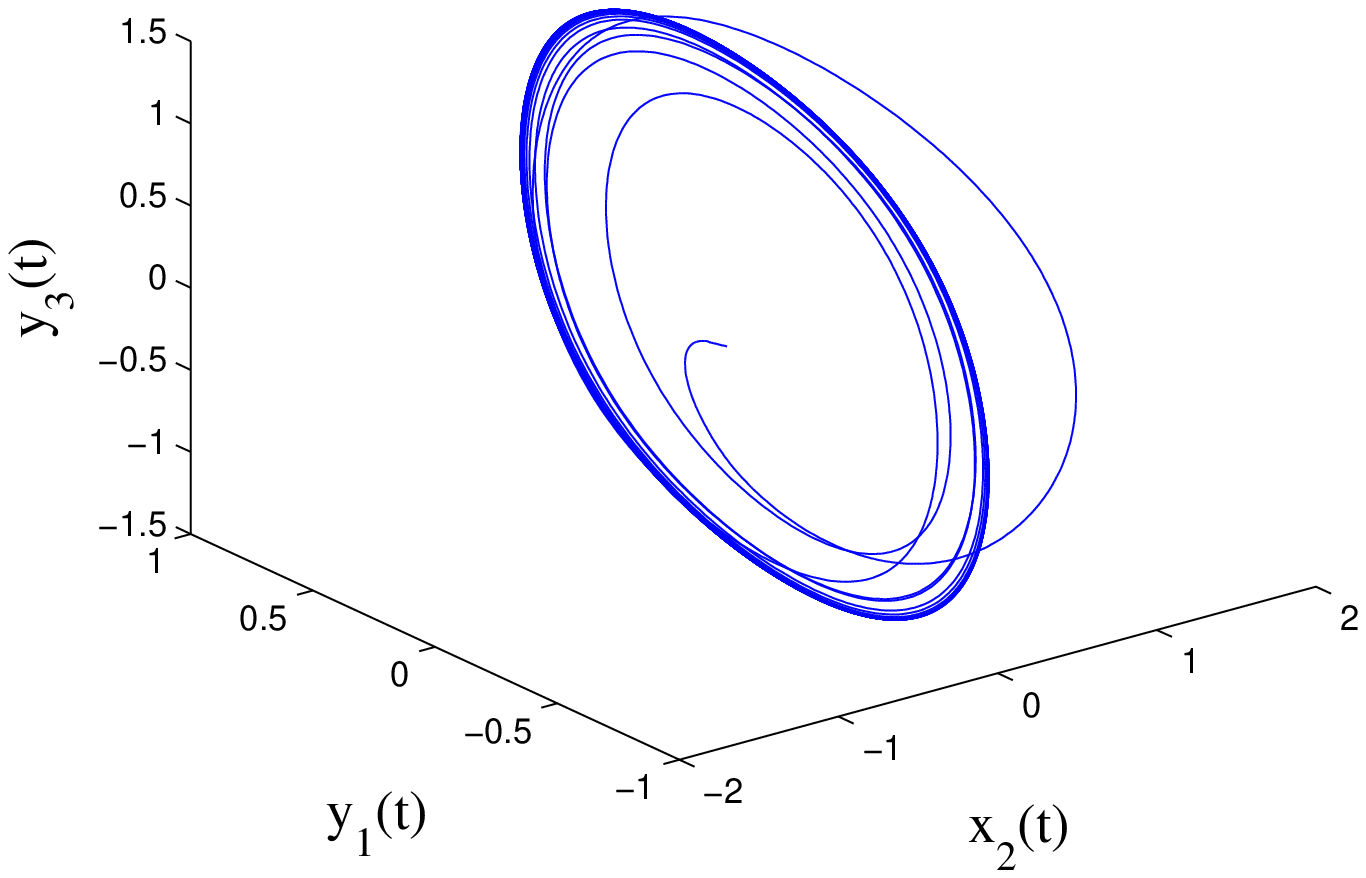}
\includegraphics[width=2in]{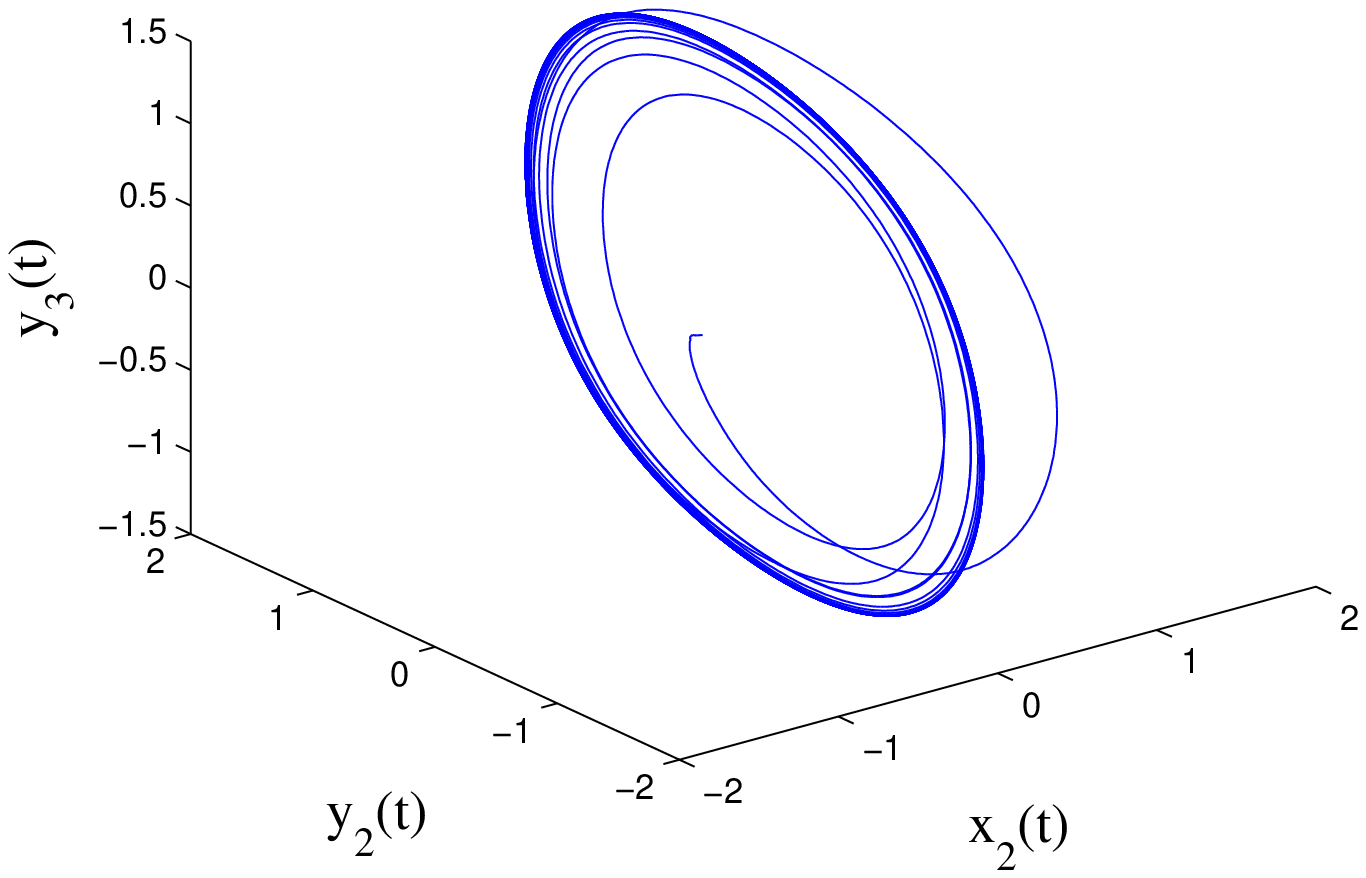}
\caption{\small{Phase diagrams of system \eqref{401} with $\theta=0.91$ and $\tau_4=0.36>\tau_0^*=0.2613$ projected on $x_1-y_1-y_2$, $x_1-y_1-y_3$, $x_1-y_2-y_3$, $x_2-y_1-y_2$, $x_2-y_1-y_3$ and $x_2-y_2-y_3$, respectively.}}
\label{fig8}
\end{figure}

\begin{figure}[H]
\centering
\includegraphics[width=2.5in]{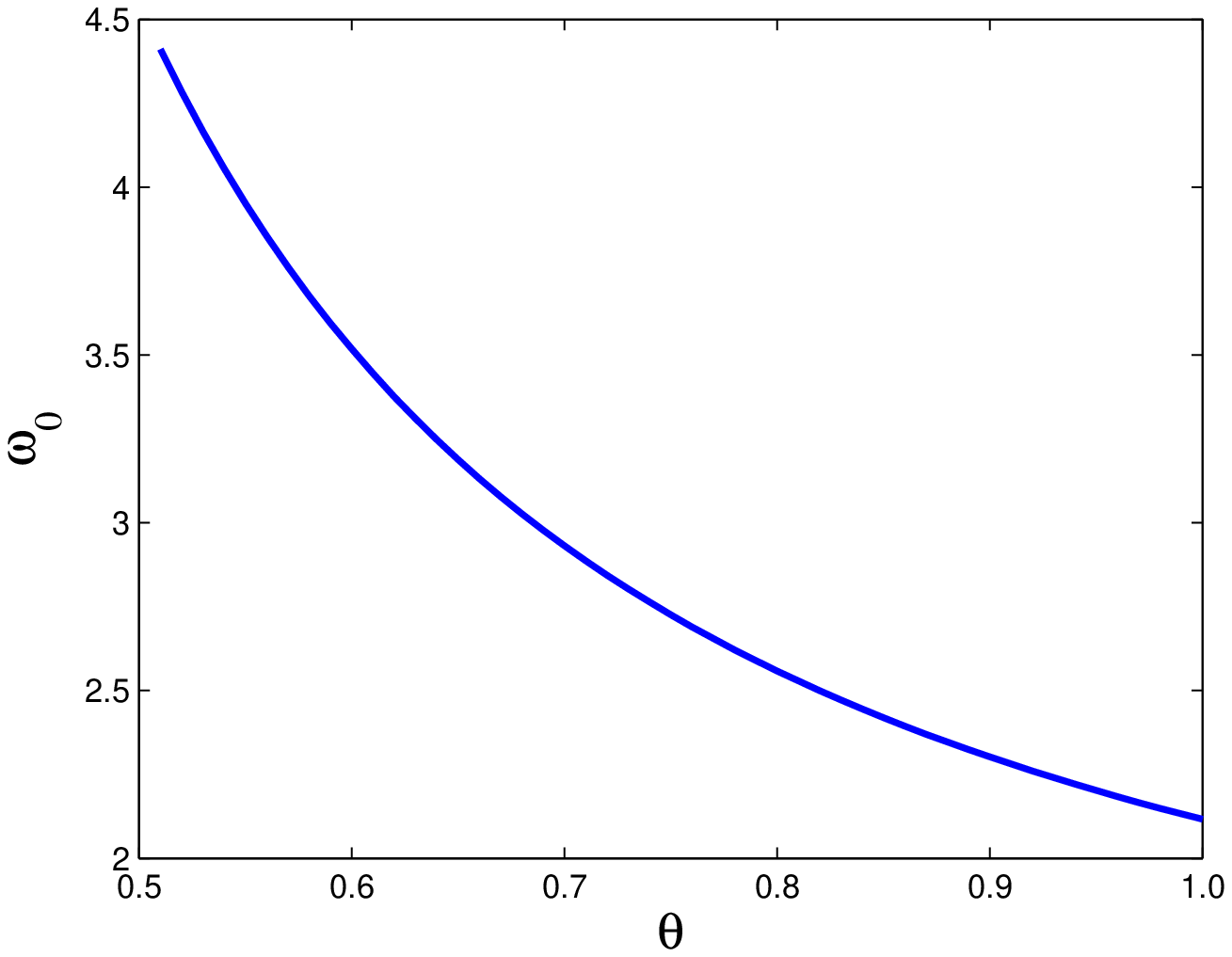}
\includegraphics[width=2.5in]{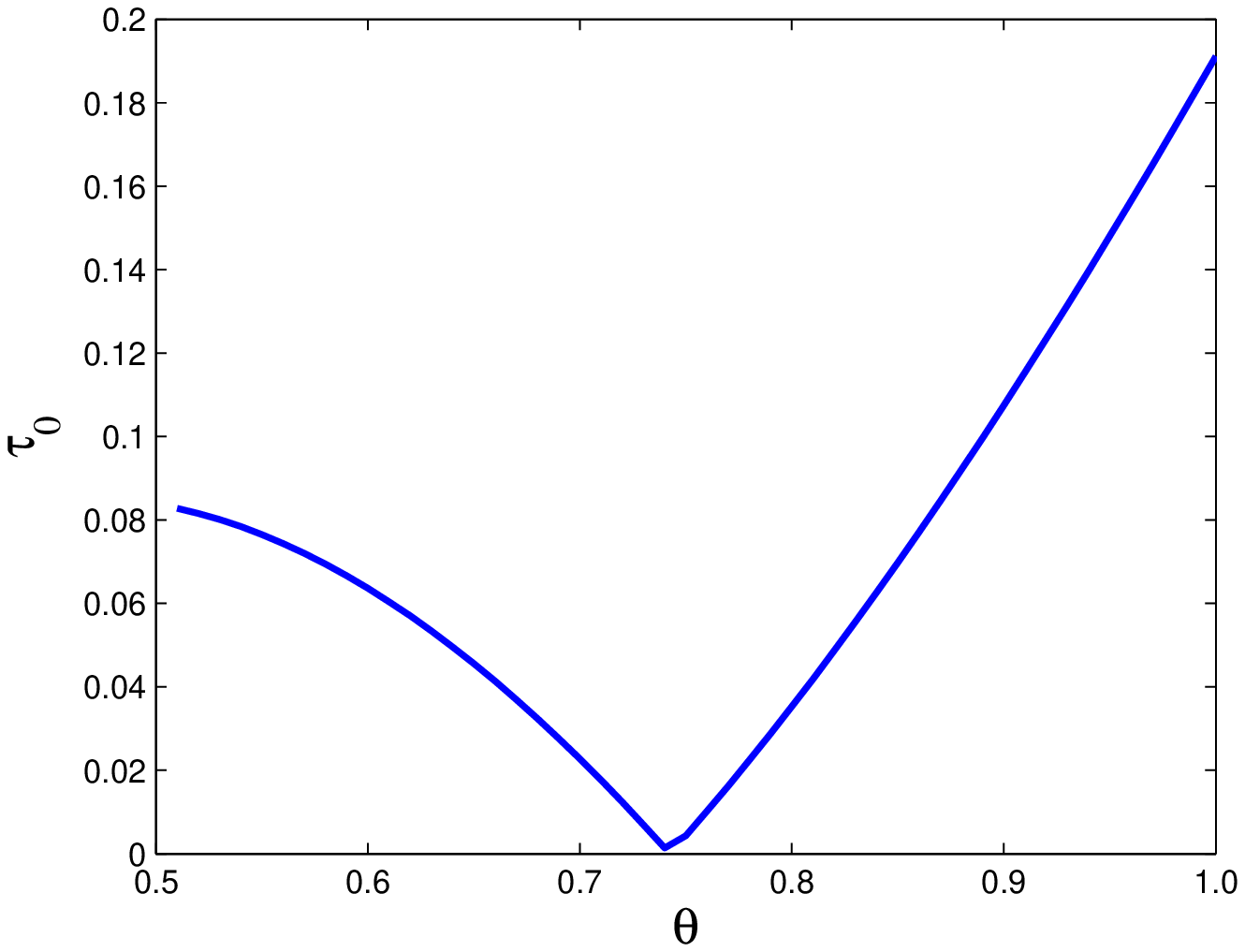}
\caption{\small{Plots of $\omega_0$ versus the order $\theta$ and $\tau_0$ versus the order $\theta$, respectively.}}
\label{fig9}
\end{figure}

\noindent declines and the scale of decrease becomes smaller accordingly. As for the bifurcation point $\tau_0$, it decreases and reach the minimal value $\tau_{min}=0.0014$ at $\theta=0.73$. Later, it increases to the maximum value $\tau_{max}=0.1911$. It is worth mentioning that when $\theta$ is near $0.73$, a little leakage delay or communication delay may lead to the occurrence of Hopf bifurcation. According to the association between the fractional order and the bifurcation point, one can delay the onset of an inherent bifurcation even eliminate the existence of Hopf bifurcation by changing the fractional order, which is insightful for the design and applications of neural networks.

\section{Sensitivity analysis of leakage delay}

Latin hypercube sampling (LHS) belongs to Monte Carlo class of sampling methods, and was introduced by Mckay et al. \cite{MBC}. LHS allows an un-biased estimate of the average model output, with the advantage that it requires fewer samples than simple random sampling to achieve the same accuracy \cite{MBC}. LHS is a so-called stratified sampling without replacement technique, where the random parameter distributions are divided into $N$ equal probability intervals, which are then sampled. $N$ represents the sample size. The choice for $N$ should be at least $k+1$, where $k$ is the number of varied parameters, but usually much larger to ensure accuracy \cite{MHR}.

Scatter plot is a simple but useful tool to give a direct visual indication of sensitivity after sampling the model over its input distributions \cite{HRW}. We set the sample size $N=1000$ and suppose that leakage delay $\tau_1$ and communication delay $\tau_2$ are subject to uniform distribution on $[0.1,0.6]$. Other system parameters are the same to Fig.\ref{fig3} and Fig.\ref{fig4}. Using Adams-Bashforth-Moulton algorithm, we obtain the stable Hopf bifurcation amplitude for each sample. Finally, we obtain $1000$ Hopf bifurcation amplitudes of the six neurons and describe it in forms of scatter plots. Fig.\ref{fig10} and Fig.\ref{fig11} show the scatter plots of the amplitudes of $x_1(t)$, $x_2(t)$, $x_3(t)$, $y_1(t)$, $y_2(t)$ and $y_3(t)$ with respect to $\tau_1$ and $\tau_2$, respectively, which imply that $\tau_1$ and $\tau_2$ are both positively correlative variables with the amplitudes of six-neuron states. However, how much strength and relevance between the delays and the amplitudes is not clear.

Next, we will study PRCCs of the amplitudes of $x_1(t)$, $x_2(t)$, $x_3(t)$, $y_1(t)$, $y_2(t)$ and $y_3(t)$ to $\tau_1$ and $\tau_2$. Marino et al. \cite{MHR} mentioned that PRCCs provide a measure of the strength of a linear association between the inputs and the outputs. The positive or negative of PRCCs respectively denote the positive or negative correlation with $\tau_1$ and $\tau_2$, and the sizes of PRCCs (the height of the bar graph in Fig.\ref{fig12}) measure the strength of the correlation. As we can see in Fig.\ref{fig12}, $\tau_1$ and $\tau_2$ are positively correlative variables to the amplitudes. Specially, we observe that the PRCCs of $\tau_1$ are smaller than the PRCCs of $\tau_2$, which implies that the impact of leakage delay is limited compared with communication delay in system \eqref{401}. Furthermore, we observe that the PRCCs of $\tau_1$ and $\tau_2$ are both more than $0.9$, which means that the dynamical behavior of system \eqref{401} is mainly determined by the leakage delay and the communication delay.

\begin{figure}[H]
\centering
\includegraphics[width=2in]{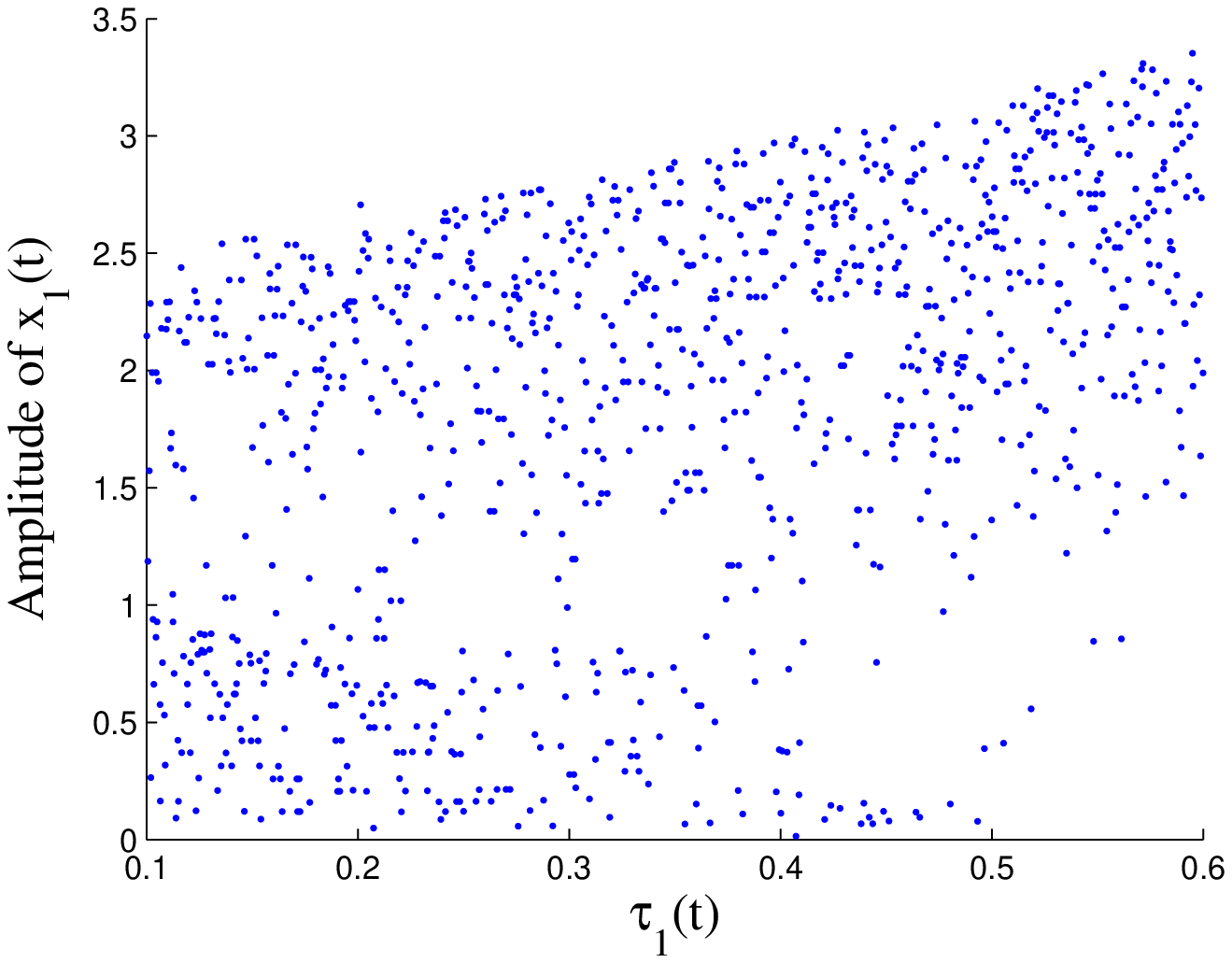}
\includegraphics[width=2in]{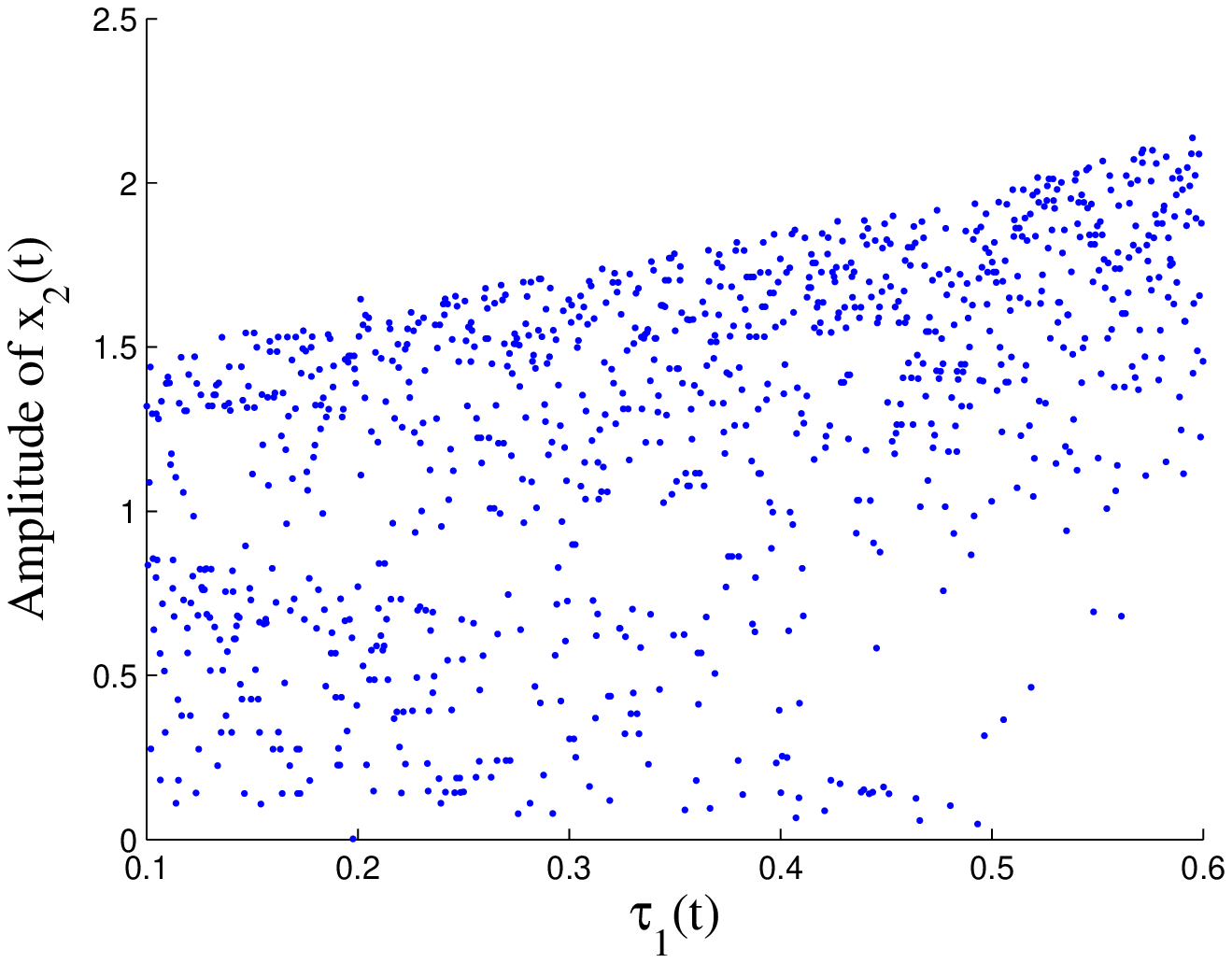}
\includegraphics[width=2in]{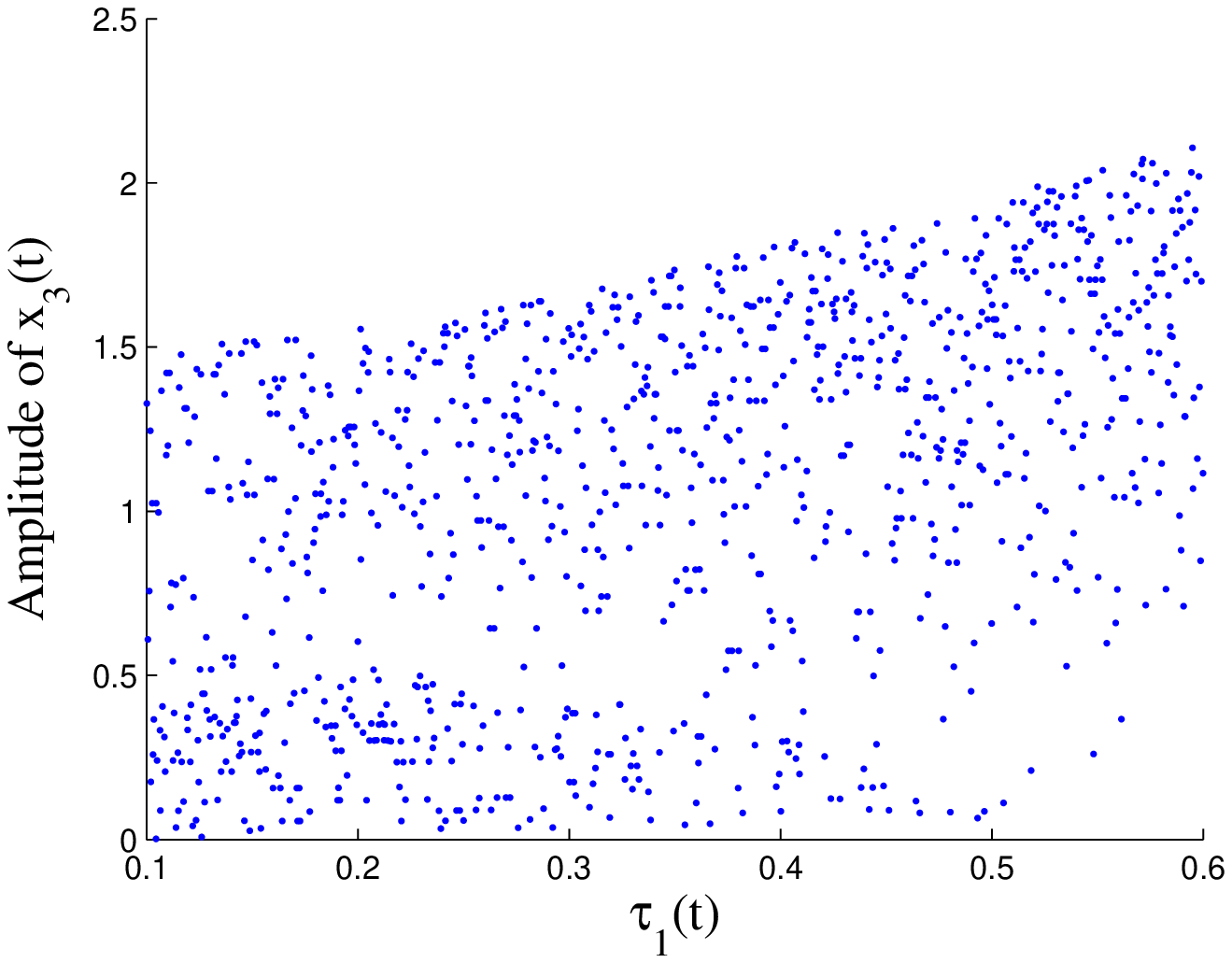}
\includegraphics[width=2in]{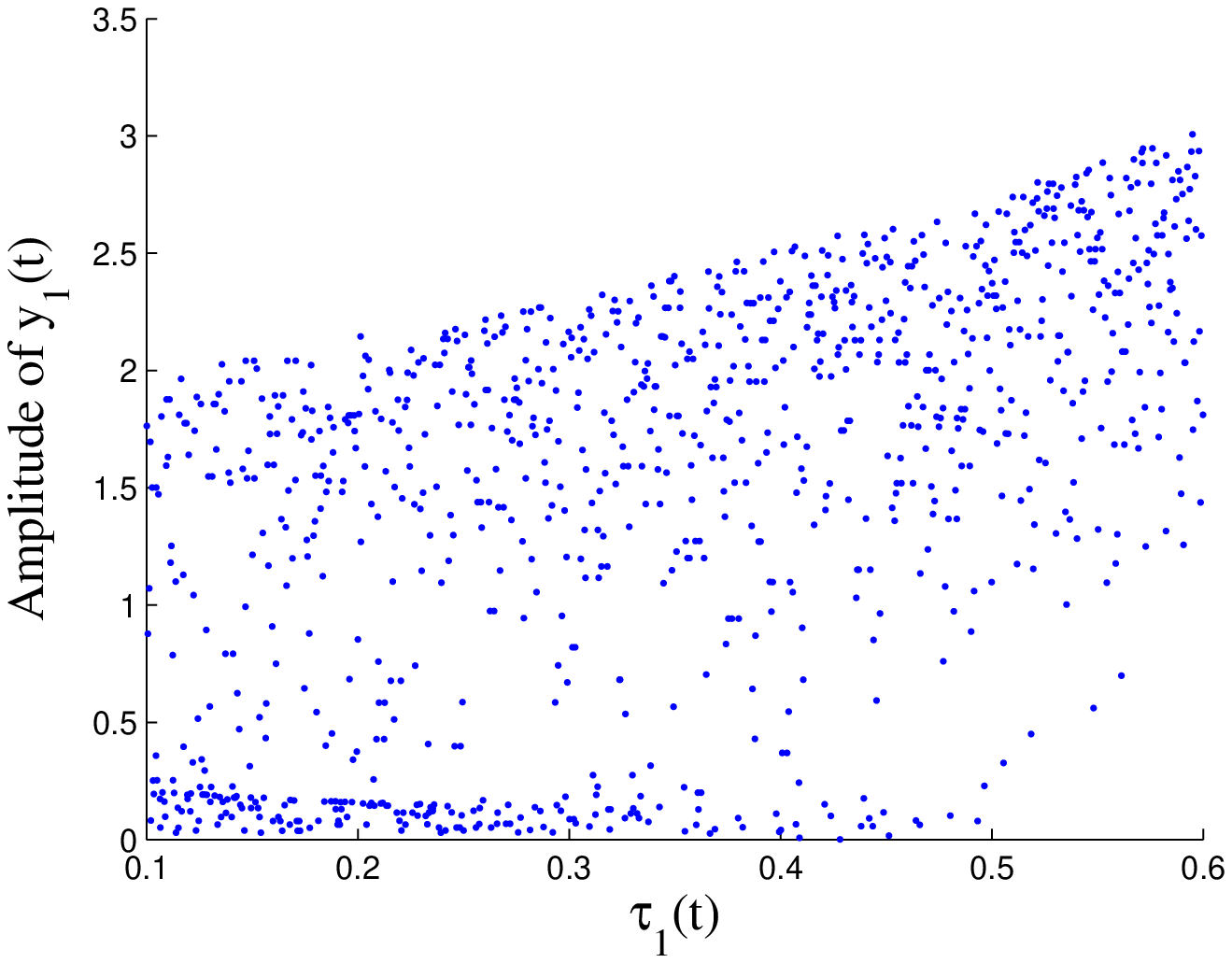}
\includegraphics[width=2in]{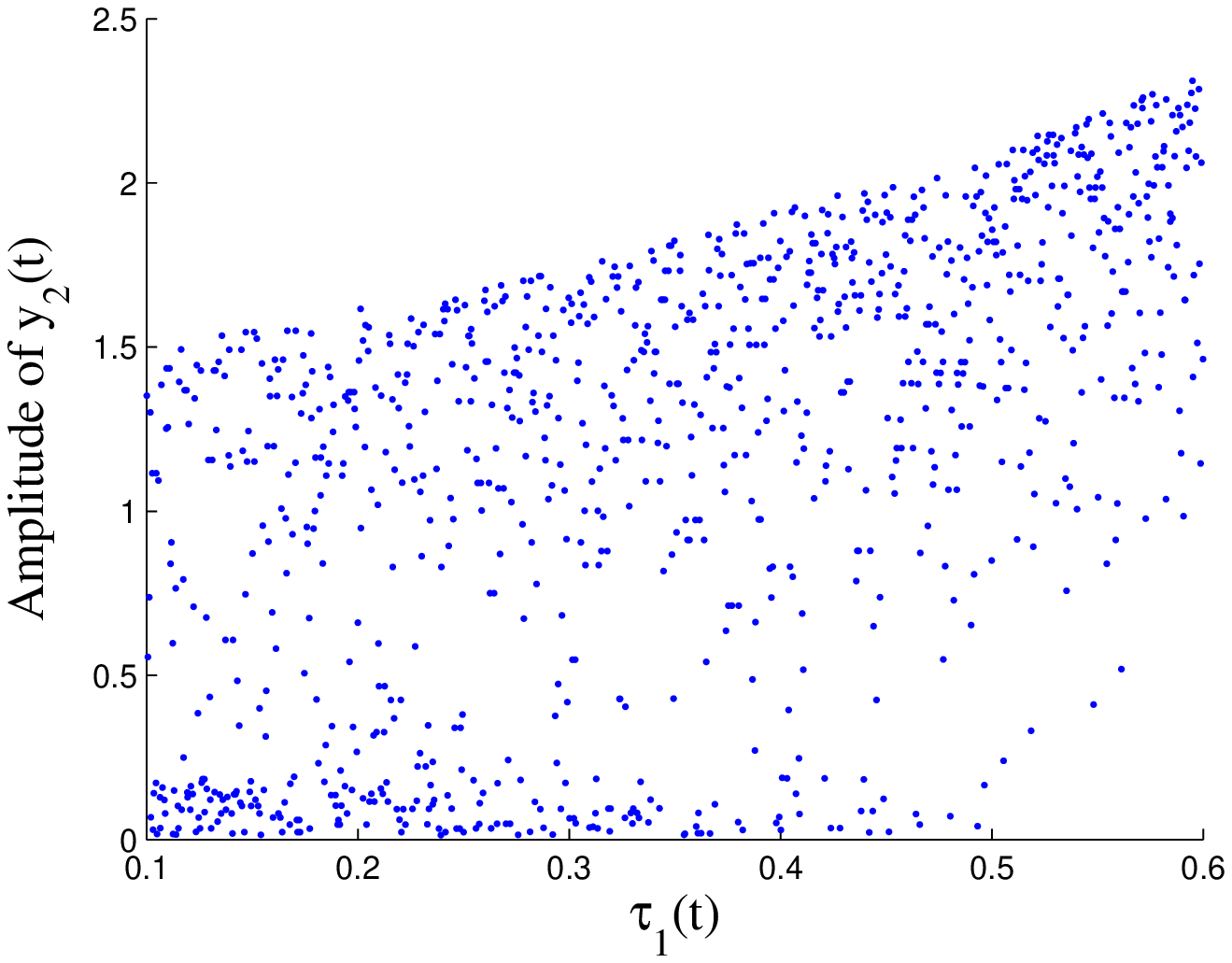}
\includegraphics[width=2in]{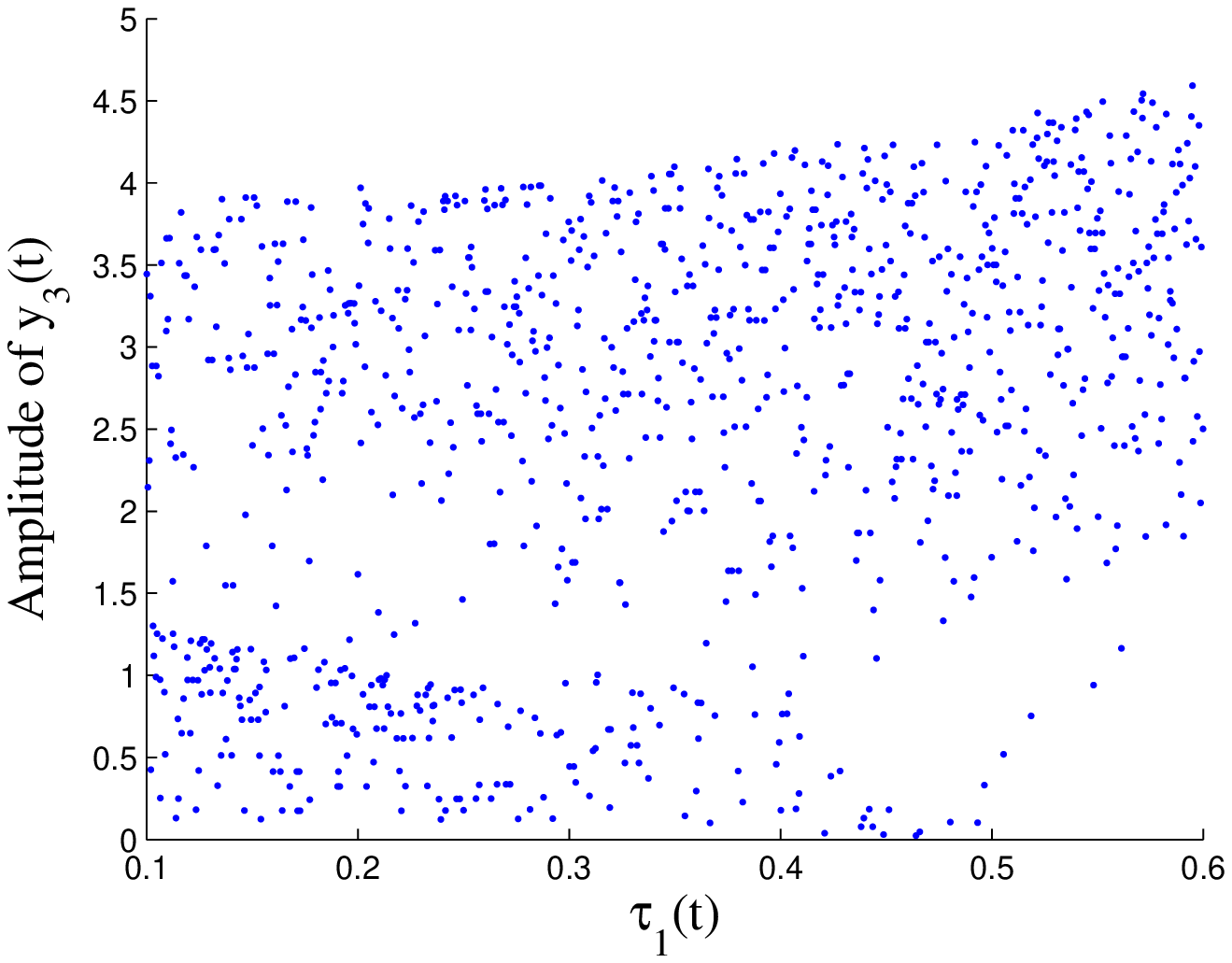}
\caption{\small{Scatter plots of amplitudes of $x_1(t)$, $x_2(t)$, $x_3(t)$, $y_1(t)$, $y_2(t)$ and $y_3(t)$ with respect to $\tau_1$.}}
\label{fig10}
\end{figure}
\begin{figure}[H]
\centering
\includegraphics[width=2in]{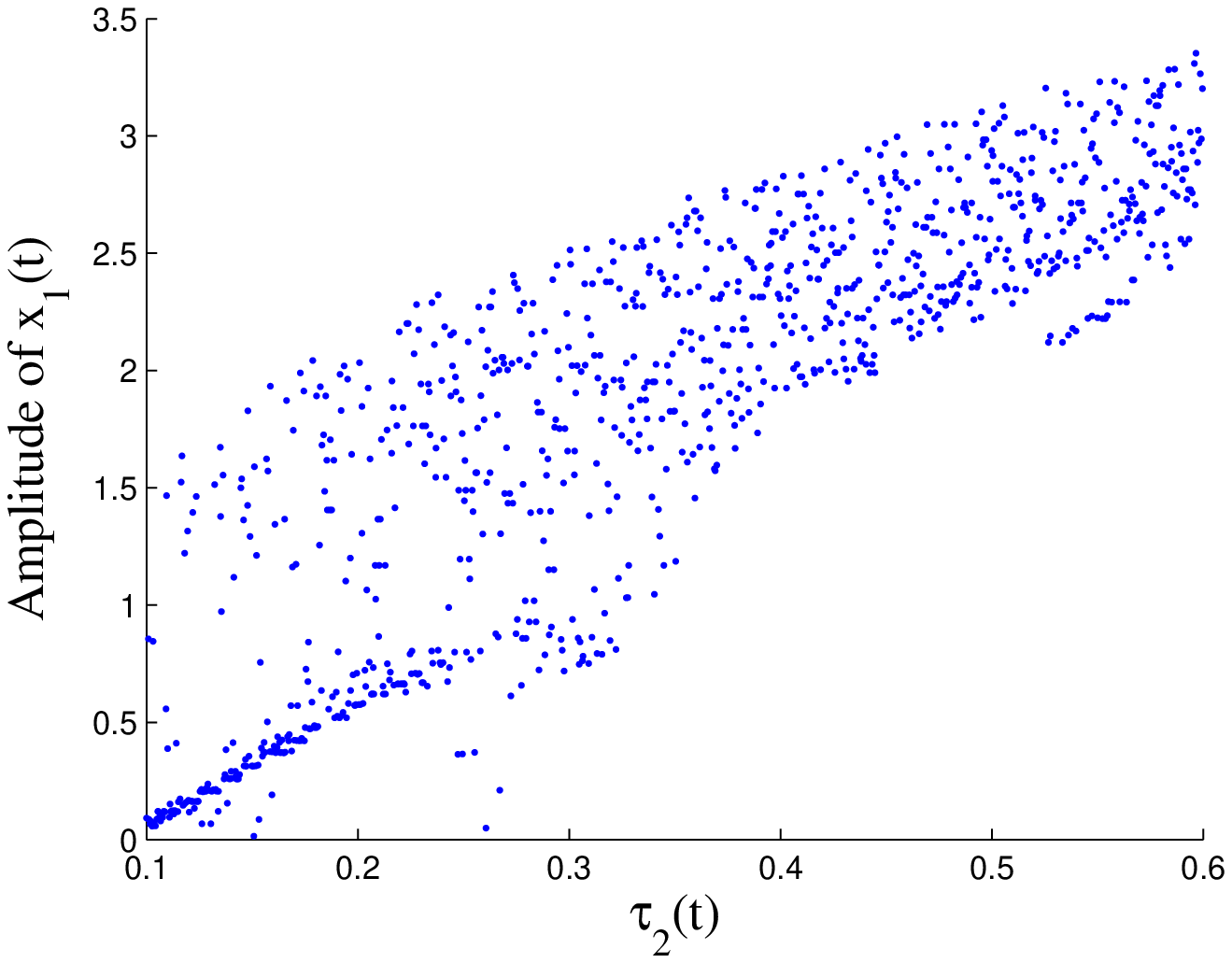}
\includegraphics[width=2in]{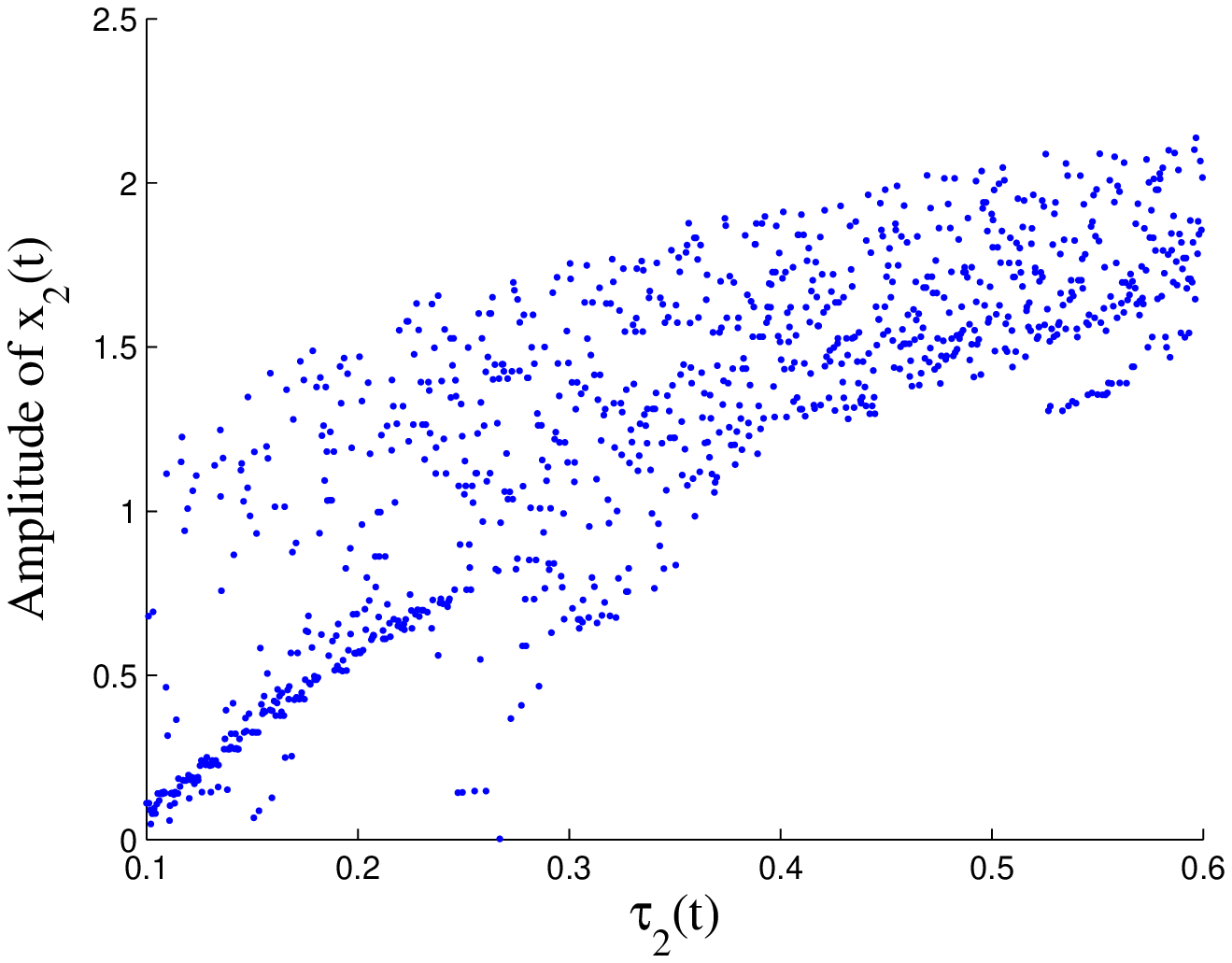}
\includegraphics[width=2in]{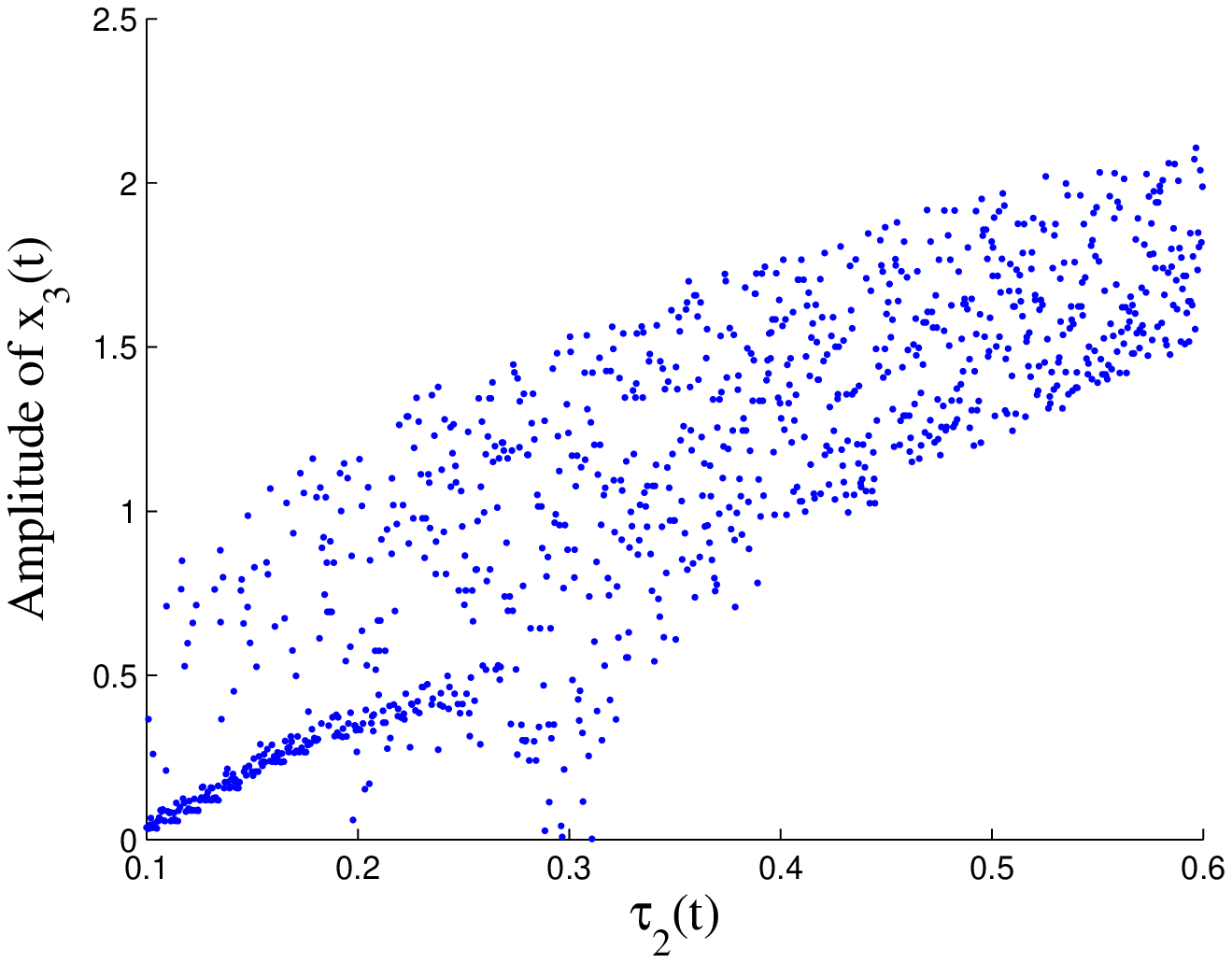}
\includegraphics[width=2in]{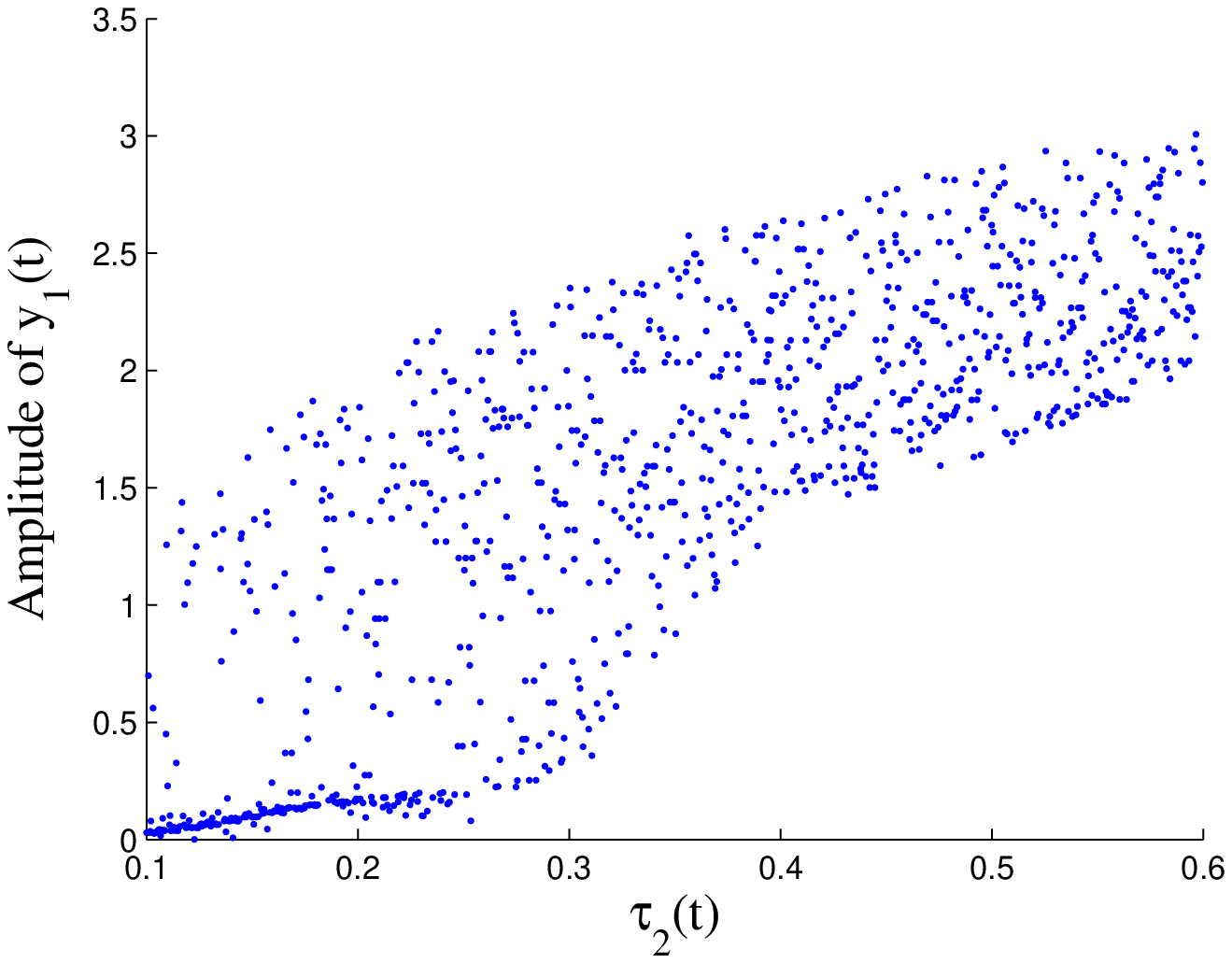}
\includegraphics[width=2in]{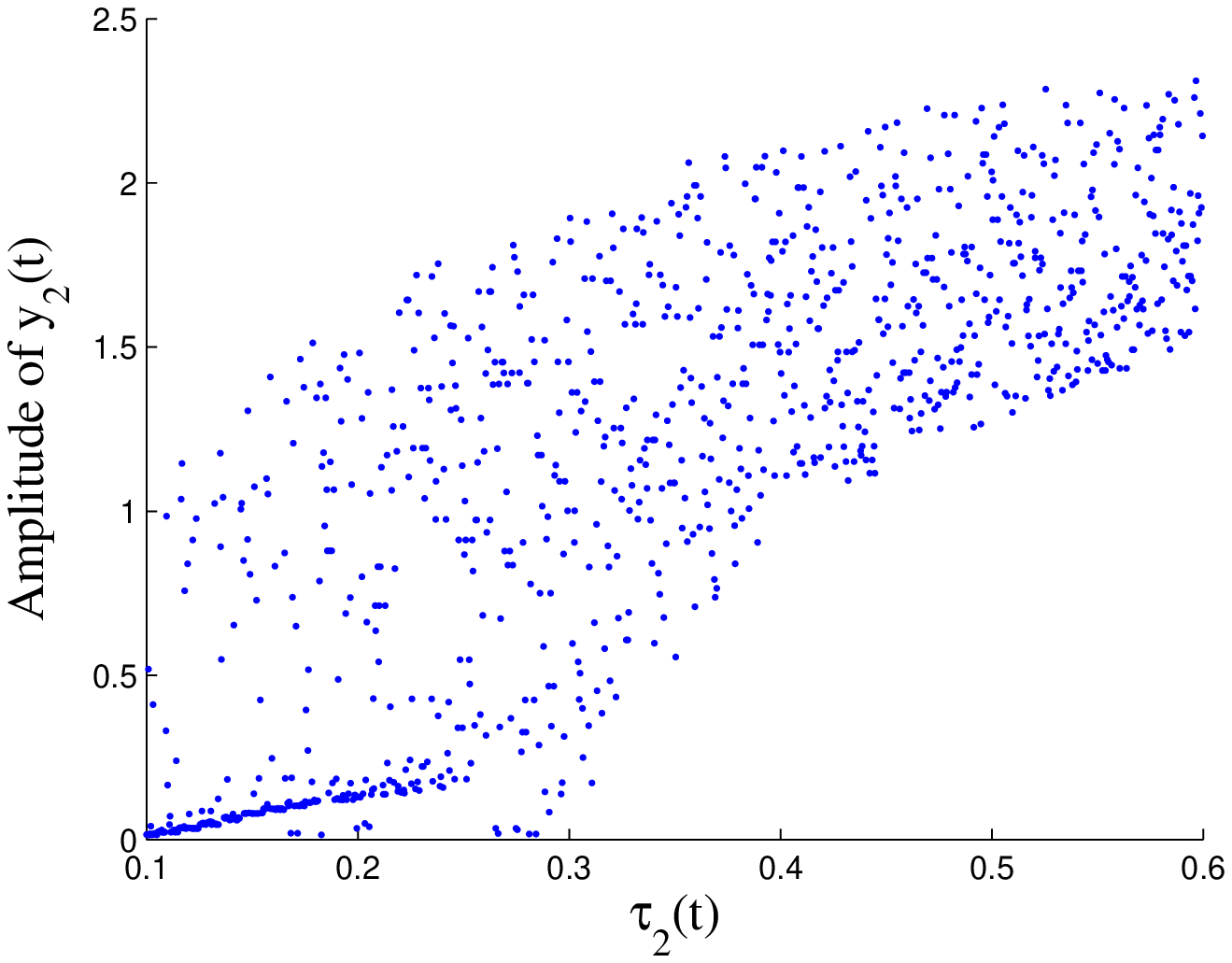}
\includegraphics[width=2in]{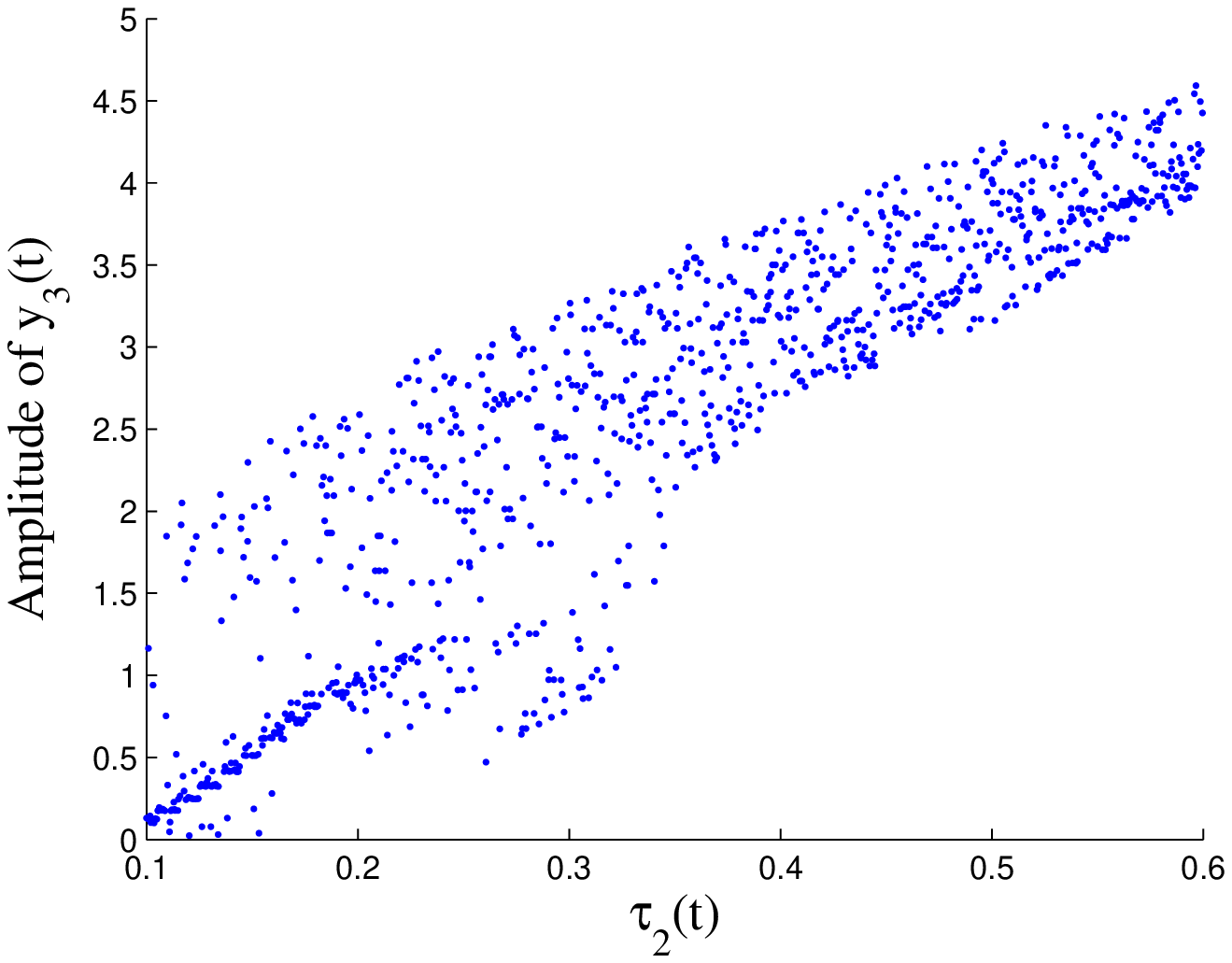}
\caption{\small{Scatter plots of amplitudes of $x_1(t)$, $x_2(t)$, $x_3(t)$, $y_1(t)$, $y_2(t)$ and $y_3(t)$ with respect to $\tau_2$.}}
\label{fig11}
\end{figure}

\begin{Remark}\label{501}
The PRCCs in Fig.\ref{fig12} implies that leakage delay indeed affects the dynamical behavior of the proposed network, but the influence of leakage delay may be not strong enough in contrast with communication delay, which only holds in some specific situations (see, e.g., system \eqref{401}). It is worth mentioning that PRCCs method provides a specific tool to assess different delays which can affect the dynamical behavior of fractional neural networks.
\end{Remark}

\begin{figure}[H]
\centering
\includegraphics[width=3in]{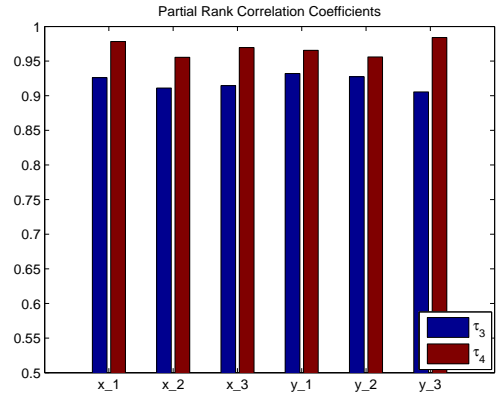}
\caption{\small{PRCCs of $x_1(t)$, $x_2(t)$, $x_3(t)$, $y_1(t)$, $y_2(t)$ and $y_3(t)$ with respect to $\tau_1$ and $\tau_2$.}}
\label{fig12}
\end{figure}

\section{Discussion and conclusion}

In this paper, we proposed a fractional six-neuron BAM neural network and established the existence of Hopf bifurcation. It was verified that the delayed fractional neural network generates a Hopf bifurcation when the bifurcation parameter $\tau_3$ or $\tau_4$ passes through some critical values. By numerical simulations, we demonstrated the effectiveness of our theoretical results and found that leakage delay weakens the stability of the fractional BAM neural network, which is similar to the works in \cite{HMC,HC}. We also investigated the influence of the order on the critical frequencies and bifurcation points numerically. Through sensitivity analysis, we clearly observed that how much impact of leakage delay and communication delay on the dynamical behavior of system \eqref{401}.

In \cite{HMC}, Huang et al. investigated a four-neuron fractional BAM neural network and revealed the effects of the leakage delay on the bifurcation points. For the sake of convenience, Huang et al. assumed that leakage delay equals to communication delay, which is similar to our theoretical analysis in Section 3.1, but didn't consider the general situation, namely, leakage delay differs from communication delay. Besides, in \cite{HMC}, the influence of leakage delay was illustrated by comparing two fractional BAM neural network with and without leakage delay, which cannot clearly indicate the impact. As for the really large neural networks, the association between the fractional order and the bifurcation points is more complicated and our simulations may reflect the really fact more closely in compare with \cite{HMC}.

Besides, neural networks are realized by electronic circuit, and the density of electromagnetic field is generally not uniform. Therefore, in factual modeling, only considering the change of time seems to be not comprehensive when electrons are moving in asymmetric and nonuniform electromagnetic fields. Thus, in our further work, we will introduce the diffusion effect into the fractional BAM neural networks.

\section*{Appendix A}

\begin{align*}
{c_{11}} =& {k_1} + {k_2} + {k_3} + {k_4} + {k_5} + {k_6}, \\
{c_{21}} =& {k_1}{k_2} + {k_1}{k_3} + {k_1}{k_4} + {k_1}{k_5} + {k_1}{k_6} + {k_2}{k_3} + {k_2}{k_4} + {k_2}{k_5} + {k_2}{k_6} + {k_3}{k_4} + {k_3}{k_5} + {k_3}{k_6} \\
&+ {k_4}{k_5} + {k_4}{k_6} + {k_5}{k_6}, \\
{c_{22}} =& {\phi _{11}}{\varphi _{11}} + {\phi _{12}}{\varphi _{21}} + {\phi _{13}}{\varphi _{31}} + {\phi _{21}}{\varphi _{12}} + {\phi _{22}}{\varphi _{22}} + {\phi _{23}}{\varphi _{32}} + {\phi _{31}}{\varphi _{13}} + {\phi _{32}}{\varphi _{23}} + {\phi _{33}}{\varphi _{33}}, \\
{c_{31}} =& {k_1}{k_2}{k_3} + {k_1}{k_2}{k_4} + {k_1}{k_2}{k_5} + {k_1}{k_2}{k_6} + {k_1}{k_3}{k_4} + {k_1}{k_3}{k_5} + {k_1}{k_3}{k_6} + {k_1}{k_4}{k_5} + {k_1}{k_4}{k_6} \\
&+ {k_1}{k_5}{k_6} + {k_2}{k_3}{k_4} + {k_2}{k_3}{k_5} + {k_2}{k_3}{k_6} + {k_2}{k_4}{k_5} + {k_2}{k_4}{k_6} + {k_2}{k_5}{k_6} + {k_3}{k_4}{k_5} + {k_3}{k_4}{k_6} \\
&+ {k_3}{k_5}{k_6} + {k_4}{k_5}{k_6}, \\
{c_{32}} =& {k_1}{\phi _{21}}{\varphi _{12}} + {k_1}{\phi _{22}}{\varphi _{22}} + {k_1}{\phi _{23}}{\varphi _{32}} + {k_1}{\phi _{31}}{\varphi _{13}} + {k_1}{\phi _{32}}{\varphi _{23}} + {k_1}{\phi _{33}}{\varphi _{33}} + {k_2}{\phi _{11}}{\varphi _{11}} \\
&+ {k_2}{\phi _{12}}{\varphi _{21}} + {k_2}{\phi _{13}}{\varphi _{31}} + {k_2}{\phi _{31}}{\varphi _{13}} + {k_2}{\phi _{32}}{\varphi _{23}} + {k_2}{\phi _{33}}{\varphi _{33}} + {k_3}{\phi _{11}}{\varphi _{11}} + {k_3}{\phi _{12}}{\varphi _{21}} \\
&+ {k_3}{\phi _{13}}{\varphi _{31}} + {k_3}{\phi _{21}}{\varphi _{12}} + {k_3}{\phi _{22}}{\varphi _{22}} + {k_3}{\phi _{23}}{\varphi _{32}} + {k_4}{\phi _{12}}{\varphi _{21}} + {k_4}{\phi _{13}}{\varphi _{31}} + {k_4}{\phi _{22}}{\varphi _{22}} \\
&+ {k_4}{\phi _{23}}{\varphi _{32}} + {k_4}{\phi _{32}}{\varphi _{23}} + {k_4}{\phi _{33}}{\varphi _{33}} + {k_5}{\phi _{11}}{\varphi _{11}} + {k_5}{\phi _{13}}{\varphi _{31}} + {k_5}{\phi _{21}}{\varphi _{12}} + {k_5}{\phi _{23}}{\varphi _{32}} \\
&+ {k_5}{\phi _{31}}{\varphi _{13}} + {k_5}{\phi _{33}}{\varphi _{33}} + {k_6}{\phi _{11}}{\varphi _{11}} + {k_6}{\phi _{12}}{\varphi _{21}} + {k_6}{\phi _{21}}{\varphi _{12}} + {k_6}{\phi _{22}}{\varphi _{22}} + {k_6}{\phi _{31}}{\varphi _{13}} \\
&+ {k_6}{\phi _{32}}{\varphi _{23}}, \\
{c_{41}} =& {k_1}{k_2}{k_3}{k_4} + {k_1}{k_2}{k_3}{k_5} + {k_1}{k_2}{k_3}{k_6} + {k_1}{k_2}{k_4}{k_5} + {k_1}{k_2}{k_4}{k_6} + {k_1}{k_2}{k_5}{k_6} + {k_1}{k_3}{k_4}{k_5} \\
&+ {k_1}{k_3}{k_4}{k_6} + {k_1}{k_3}{k_5}{k_6} + {k_1}{k_4}{k_5}{k_6} + {k_2}{k_3}{k_4}{k_5} + {k_2}{k_3}{k_4}{k_6} + {k_2}{k_3}{k_5}{k_6} + {k_2}{k_4}{k_5}{k_6} \\
&+ {k_3}{k_4}{k_5}{k_6},
\end{align*}

\begin{align*}
{c_{42}} =& {k_1}{k_2}{\phi _{31}}{\varphi _{13}} + {k_1}{k_2}{\phi _{32}}{\varphi _{23}} + {k_1}{k_2}{\phi _{33}}{\varphi _{33}} + {k_1}{k_3}{\phi _{21}}{\varphi _{12}} + {k_1}{k_3}{\phi _{22}}{\varphi _{22}} + {k_1}{k_3}{\phi _{23}}{\varphi _{32}} \\
&+ {k_1}{k_4}{\phi _{22}}{\varphi _{22}} + {k_1}{k_4}{\phi _{23}}{\varphi _{32}} + {k_1}{k_4}{\phi _{32}}{\varphi _{23}} + {k_1}{k_4}{\phi _{33}}{\varphi _{33}} + {k_1}{k_5}{\phi _{21}}{\varphi _{12}} + {k_1}{k_5}{\phi _{23}}{\varphi _{32}} \\
&+ {k_1}{k_5}{\phi _{31}}{\varphi _{13}} + {k_1}{k_5}{\phi _{33}}{\varphi _{33}} + {k_1}{k_6}{\phi _{21}}{\varphi _{12}} + {k_1}{k_6}{\phi _{22}}{\varphi _{22}} + {k_1}{k_6}{\phi _{31}}{\varphi _{13}} + {k_1}{k_6}{\phi _{32}}{\varphi _{23}} \\
&+ {k_2}{k_3}{\phi _{11}}{\varphi _{11}} + {k_2}{k_3}{\phi _{12}}{\varphi _{21}} + {k_2}{k_3}{\phi _{13}}{\varphi _{31}} + {k_2}{k_4}{\phi _{12}}{\varphi _{21}} + {k_2}{k_4}{\phi _{13}}{\varphi _{31}} + {k_2}{k_4}{\phi _{32}}{\varphi _{23}} \\
&+ {k_2}{k_4}{\phi _{33}}{\varphi _{33}} + {k_2}{k_5}{\phi _{11}}{\varphi _{11}} + {k_2}{k_5}{\phi _{13}}{\varphi _{31}} + {k_2}{k_5}{\phi _{31}}{\varphi _{13}} + {k_2}{k_5}{\phi _{33}}{\varphi _{33}} + {k_2}{k_6}{\phi _{11}}{\varphi _{11}} \\
&+ {k_2}{k_6}{\phi _{12}}{\varphi _{21}} + {k_2}{k_6}{\phi _{31}}{\varphi _{13}} + {k_2}{k_6}{\phi _{32}}{\varphi _{23}} + {k_3}{k_4}{\phi _{12}}{\varphi _{21}} + {k_3}{k_4}{\phi _{13}}{\varphi _{31}} + {k_3}{k_4}{\phi _{22}}{\varphi _{22}} \\
&+ {k_3}{k_4}{\phi _{23}}{\varphi _{32}} + {k_3}{k_5}{\phi _{11}}{\varphi _{11}} + {k_3}{k_5}{\phi _{13}}{\varphi _{31}} + {k_3}{k_5}{\phi _{21}}{\varphi _{12}} + {k_3}{k_5}{\phi _{23}}{\varphi _{32}} + {k_3}{k_6}{\phi _{11}}{\varphi _{11}} \\
&+ {k_3}{k_6}{\phi _{12}}{\varphi _{21}} + {k_3}{k_6}{\phi _{21}}{\varphi _{12}} + {k_3}{k_6}{\phi _{22}}{\varphi _{22}} + {k_4}{k_5}{\phi _{13}}{\varphi _{31}} + {k_4}{k_5}{\phi _{23}}{\varphi _{32}} + {k_4}{k_5}{\phi _{33}}{\varphi _{33}} \\
&+ {k_4}{k_6}{\phi _{12}}{\varphi _{21}} + {k_4}{k_6}{\phi _{22}}{\varphi _{22}} + {k_4}{k_6}{\phi _{32}}{\varphi _{23}} + {k_5}{k_6}{\phi _{11}}{\varphi _{11}} + {k_5}{k_6}{\phi _{21}}{\varphi _{12}} + {k_5}{k_6}{\phi _{31}}{\varphi _{13}}, \\
{c_{43}} =& {\phi _{11}}{\phi _{22}}{\varphi _{11}}{\varphi _{22}} + {\phi _{11}}{\phi _{23}}{\varphi _{11}}{\varphi _{32}} + {\phi _{11}}{\phi _{32}}{\varphi _{11}}{\varphi _{23}} + {\phi _{11}}{\phi _{33}}{\varphi _{11}}{\varphi _{33}} + {\phi _{12}}{\phi _{21}}{\varphi _{12}}{\varphi _{21}} \\
&+ {\phi _{12}}{\phi _{23}}{\varphi _{21}}{\varphi _{32}} + {\phi _{12}}{\phi _{31}}{\varphi _{13}}{\varphi _{21}} + {\phi _{12}}{\phi _{33}}{\varphi _{21}}{\varphi _{33}} + {\phi _{13}}{\phi _{21}}{\varphi _{12}}{\varphi _{31}} + {\phi _{13}}{\phi _{22}}{\varphi _{22}}{\varphi _{31}} \\
&+ {\phi _{13}}{\phi _{31}}{\varphi _{13}}{\varphi _{31}} + {\phi _{13}}{\phi _{32}}{\varphi _{23}}{\varphi _{31}} + {\phi _{21}}{\phi _{32}}{\varphi _{12}}{\varphi _{23}} + {\phi _{21}}{\phi _{33}}{\varphi _{12}}{\varphi _{33}} + {\phi _{22}}{\phi _{31}}{\varphi _{13}}{\varphi _{22}} \\
&+ {\phi _{22}}{\phi _{33}}{\varphi _{22}}{\varphi _{33}} + {\phi _{23}}{\phi _{31}}{\varphi _{13}}{\varphi _{32}} + {\phi _{23}}{\phi _{32}}{\varphi _{23}}{\varphi _{32}}, \\
{c_{44}} =& {\phi _{11}}{\phi _{22}}{\varphi _{12}}{\varphi _{21}} + {\phi _{11}}{\phi _{23}}{\varphi _{12}}{\varphi _{31}} + {\phi _{11}}{\phi _{32}}{\varphi _{13}}{\varphi _{21}} + {\phi _{11}}{\phi _{33}}{\varphi _{13}}{\varphi _{31}} + {\phi _{12}}{\phi _{21}}{\varphi _{11}}{\varphi _{22}} \\
&+ {\phi _{12}}{\phi _{23}}{\varphi _{22}}{\varphi _{31}} + {\phi _{12}}{\phi _{31}}{\varphi _{11}}{\varphi _{23}} + {\phi _{12}}{\phi _{33}}{\varphi _{23}}{\varphi _{31}} + {\phi _{13}}{\phi _{21}}{\varphi _{11}}{\varphi _{32}} + {\phi _{13}}{\phi _{22}}{\varphi _{21}}{\varphi _{32}} \\
&+ {\phi _{13}}{\phi _{31}}{\varphi _{11}}{\varphi _{33}} + {\phi _{13}}{\phi _{32}}{\varphi _{21}}{\varphi _{33}} + {\phi _{21}}{\phi _{32}}{\varphi _{13}}{\varphi _{22}} + {\phi _{21}}{\phi _{33}}{\varphi _{13}}{\varphi _{32}} + {\phi _{22}}{\phi _{31}}{\varphi _{12}}{\varphi _{23}} \\
&+ {\phi _{22}}{\phi _{33}}{\varphi _{23}}{\varphi _{32}} + {\phi _{23}}{\phi _{31}}{\varphi _{12}}{\varphi _{33}} + {\phi _{23}}{\phi _{32}}{\varphi _{22}}{\varphi _{33}}, \\
{c_{51}} =& {k_1}{k_2}{k_3}{k_4}{k_5} + {k_1}{k_2}{k_3}{k_4}{k_6} + {k_1}{k_2}{k_3}{k_5}{k_6} + {k_1}{k_2}{k_4}{k_5}{k_6} + {k_1}{k_3}{k_4}{k_5}{k_6} + {k_2}{k_3}{k_4}{k_5}{k_6}, \\
{c_{52}} =& {k_1}{k_2}{k_4}{\phi _{32}}{\varphi _{23}} + {k_1}{k_2}{k_4}{\phi _{33}}{\varphi _{33}} + {k_1}{k_2}{k_5}{\phi _{31}}{\varphi _{13}} + {k_1}{k_2}{k_5}{\phi _{33}}{\varphi _{33}} + {k_1}{k_2}{k_6}{\phi _{31}}{\varphi _{13}} \\
&+ {k_1}{k_2}{k_6}{\phi _{32}}{\varphi _{23}} + {k_1}{k_3}{k_4}{\phi _{22}}{\varphi _{22}} + {k_1}{k_3}{k_4}{\phi _{23}}{\varphi _{32}} + {k_1}{k_3}{k_5}{\phi _{21}}{\varphi _{12}} + {k_1}{k_3}{k_5}{\phi _{23}}{\varphi _{32}} \\
&+ {k_1}{k_3}{k_6}{\phi _{21}}{\varphi _{12}} + {k_1}{k_3}{k_6}{\phi _{22}}{\varphi _{22}} + {k_1}{k_4}{k_5}{\phi _{23}}{\varphi _{32}} + {k_1}{k_4}{k_5}{\phi _{33}}{\varphi _{33}} + {k_1}{k_4}{k_6}{\phi _{22}}{\varphi _{22}} \\
&+ {k_1}{k_4}{k_6}{\phi _{32}}{\varphi _{23}} + {k_1}{k_5}{k_6}{\phi _{21}}{\varphi _{12}} + {k_1}{k_5}{k_6}{\phi _{31}}{\varphi _{13}} + {k_2}{k_3}{k_4}{\phi _{12}}{\varphi _{21}} + {k_2}{k_3}{k_4}{\phi _{13}}{\varphi _{31}} \\
&+ {k_2}{k_3}{k_5}{\phi _{11}}{\varphi _{11}} + {k_2}{k_3}{k_5}{\phi _{13}}{\varphi _{31}} + {k_2}{k_3}{k_6}{\phi _{11}}{\varphi _{11}} + {k_2}{k_3}{k_6}{\phi _{12}}{\varphi _{21}} + {k_2}{k_4}{k_5}{\phi _{13}}{\varphi _{31}} \\
&+ {k_2}{k_4}{k_5}{\phi _{33}}{\varphi _{33}} + {k_2}{k_4}{k_6}{\phi _{12}}{\varphi _{21}} + {k_2}{k_4}{k_6}{\phi _{32}}{\varphi _{23}} + {k_2}{k_5}{k_6}{\phi _{11}}{\varphi _{11}} + {k_2}{k_5}{k_6}{\phi _{31}}{\varphi _{13}} \\
&+ {k_3}{k_4}{k_5}{\phi _{13}}{\varphi _{31}} + {k_3}{k_4}{k_5}{\phi _{23}}{\varphi _{32}} + {k_3}{k_4}{k_6}{\phi _{12}}{\varphi _{21}} + {k_3}{k_4}{k_6}{\phi _{22}}{\varphi _{22}} + {k_3}{k_5}{k_6}{\phi _{11}}{\varphi _{11}} \\
&+ {k_3}{k_5}{k_6}{\phi _{21}}{\varphi _{12}}, \\
{c_{53}} =& {k_1}{\phi _{21}}{\phi _{32}}{\varphi _{12}}{\varphi _{23}} + {k_1}{\phi _{21}}{\phi _{33}}{\varphi _{12}}{\varphi _{33}} + {k_1}{\phi _{22}}{\phi _{31}}{\varphi _{13}}{\varphi _{22}} + {k_1}{\phi _{22}}{\phi _{33}}{\varphi _{22}}{\varphi _{33}} \\
&+ {k_1}{\phi _{23}}{\phi _{31}}{\varphi _{13}}{\varphi _{32}} + {k_1}{\phi _{23}}{\phi _{32}}{\varphi _{23}}{\varphi _{32}} + {k_2}{\phi _{11}}{\phi _{32}}{\varphi _{11}}{\varphi _{23}} + {k_2}{\phi _{11}}{\phi _{33}}{\varphi _{11}}{\varphi _{33}} \\
&+ {k_2}{\phi _{12}}{\phi _{31}}{\varphi _{13}}{\varphi _{21}} + {k_2}{\phi _{12}}{\phi _{33}}{\varphi _{21}}{\varphi _{33}} + {k_2}{\phi _{13}}{\phi _{31}}{\varphi _{13}}{\varphi _{31}} + {k_2}{\phi _{13}}{\phi _{32}}{\varphi _{23}}{\varphi _{31}} \\
&+ {k_3}{\phi _{11}}{\phi _{22}}{\varphi _{11}}{\varphi _{22}} + {k_3}{\phi _{11}}{\phi _{23}}{\varphi _{11}}{\varphi _{32}} + {k_3}{\phi _{12}}{\phi _{21}}{\varphi _{12}}{\varphi _{21}} + {k_3}{\phi _{12}}{\phi _{23}}{\varphi _{21}}{\varphi _{32}} \\
&+ {k_3}{\phi _{13}}{\phi _{21}}{\varphi _{12}}{\varphi _{31}} + {k_3}{\phi _{13}}{\phi _{22}}{\varphi _{22}}{\varphi _{31}} + {k_4}{\phi _{12}}{\phi _{23}}{\varphi _{21}}{\varphi _{32}} + {k_4}{\phi _{12}}{\phi _{33}}{\varphi _{21}}{\varphi _{33}} \\
&+ {k_4}{\phi _{13}}{\phi _{22}}{\varphi _{22}}{\varphi _{31}} + {k_4}{\phi _{13}}{\phi _{32}}{\varphi _{23}}{\varphi _{31}} + {k_4}{\phi _{22}}{\phi _{33}}{\varphi _{22}}{\varphi _{33}} + {k_4}{\phi _{23}}{\phi _{32}}{\varphi _{23}}{\varphi _{32}} \\
&+ {k_5}{\phi _{11}}{\phi _{23}}{\varphi _{11}}{\varphi _{32}} + {k_5}{\phi _{11}}{\phi _{33}}{\varphi _{11}}{\varphi _{33}} + {k_5}{\phi _{13}}{\phi _{21}}{\varphi _{12}}{\varphi _{31}} + {k_5}{\phi _{13}}{\phi _{31}}{\varphi _{13}}{\varphi _{31}} \\
&+ {k_5}{\phi _{21}}{\phi _{33}}{\varphi _{12}}{\varphi _{33}} + {k_5}{\phi _{23}}{\phi _{31}}{\varphi _{13}}{\varphi _{32}} + {k_6}{\phi _{11}}{\phi _{22}}{\varphi _{11}}{\varphi _{22}} + {k_6}{\phi _{11}}{\phi _{32}}{\varphi _{11}}{\varphi _{23}} \\
&+ {k_6}{\phi _{12}}{\phi _{21}}{\varphi _{12}}{\varphi _{21}} + {k_6}{\phi _{12}}{\phi _{31}}{\varphi _{13}}{\varphi _{21}} + {k_6}{\phi _{21}}{\phi _{32}}{\varphi _{12}}{\varphi _{23}} + {k_6}{\phi _{22}}{\phi _{31}}{\varphi _{13}}{\varphi _{22}},
\end{align*}

\begin{align*}
{c_{54}} =& {k_1}{\phi _{21}}{\phi _{32}}{\varphi _{13}}{\varphi _{22}} + {k_1}{\phi _{21}}{\phi _{33}}{\varphi _{13}}{\varphi _{32}} + {k_1}{\phi _{22}}{\phi _{31}}{\varphi _{12}}{\varphi _{23}} + {k_1}{\phi _{22}}{\phi _{33}}{\varphi _{23}}{\varphi _{32}} \\
&+ {k_1}{\phi _{23}}{\phi _{31}}{\varphi _{12}}{\varphi _{33}} + {k_1}{\phi _{23}}{\phi _{32}}{\varphi _{22}}{\varphi _{33}} + {k_2}{\phi _{11}}{\phi _{32}}{\varphi _{13}}{\varphi _{21}} + {k_2}{\phi _{11}}{\phi _{33}}{\varphi _{13}}{\varphi _{31}} \\
&+ {k_2}{\phi _{12}}{\phi _{31}}{\varphi _{11}}{\varphi _{23}} + {k_2}{\phi _{12}}{\phi _{33}}{\varphi _{23}}{\varphi _{31}} + {k_2}{\phi _{13}}{\phi _{31}}{\varphi _{11}}{\varphi _{33}} + {k_2}{\phi _{13}}{\phi _{32}}{\varphi _{21}}{\varphi _{33}} \\
&+ {k_3}{\phi _{11}}{\phi _{22}}{\varphi _{12}}{\varphi _{21}} + {k_3}{\phi _{11}}{\phi _{23}}{\varphi _{12}}{\varphi _{31}} + {k_3}{\phi _{12}}{\phi _{21}}{\varphi _{11}}{\varphi _{22}} + {k_3}{\phi _{12}}{\phi _{23}}{\varphi _{22}}{\varphi _{31}} \\
&+ {k_3}{\phi _{13}}{\phi _{21}}{\varphi _{11}}{\varphi _{32}} + {k_3}{\phi _{13}}{\phi _{22}}{\varphi _{21}}{\varphi _{32}} + {k_4}{\phi _{12}}{\phi _{23}}{\varphi _{22}}{\varphi _{31}} + {k_4}{\phi _{12}}{\phi _{33}}{\varphi _{23}}{\varphi _{31}} \\
&+ {k_4}{\phi _{13}}{\phi _{22}}{\varphi _{21}}{\varphi _{32}} + {k_4}{\phi _{13}}{\phi _{32}}{\varphi _{21}}{\varphi _{33}} + {k_4}{\phi _{22}}{\phi _{33}}{\varphi _{23}}{\varphi _{32}} + {k_4}{\phi _{23}}{\phi _{32}}{\varphi _{22}}{\varphi _{33}} \\
&+ {k_5}{\phi _{11}}{\phi _{23}}{\varphi _{12}}{\varphi _{31}} + {k_5}{\phi _{11}}{\phi _{33}}{\varphi _{13}}{\varphi _{31}} + {k_5}{\phi _{13}}{\phi _{21}}{\varphi _{11}}{\varphi _{32}} + {k_5}{\phi _{13}}{\phi _{31}}{\varphi _{11}}{\varphi _{33}} \\
&+ {k_5}{\phi _{21}}{\phi _{33}}{\varphi _{13}}{\varphi _{32}} + {k_5}{\phi _{23}}{\phi _{31}}{\varphi _{12}}{\varphi _{33}} + {k_6}{\phi _{11}}{\phi _{22}}{\varphi _{12}}{\varphi _{21}} + {k_6}{\phi _{11}}{\phi _{32}}{\varphi _{13}}{\varphi _{21}} \\
&+ {k_6}{\phi _{12}}{\phi _{21}}{\varphi _{11}}{\varphi _{22}} + {k_6}{\phi _{12}}{\phi _{31}}{\varphi _{11}}{\varphi _{23}} + {k_6}{\phi _{21}}{\phi _{32}}{\varphi _{13}}{\varphi _{22}} + {k_6}{\phi _{22}}{\phi _{31}}{\varphi _{12}}{\varphi _{23}}, \\
{c_{61}} =& {k_1}{k_2}{k_3}{k_4}{k_5}{k_6} + {\phi _{11}}{\phi _{22}}{\phi _{33}}{\varphi _{11}}{\varphi _{23}}{\varphi _{32}} + {\phi _{11}}{\phi _{22}}{\phi _{33}}{\varphi _{12}}{\varphi _{21}}{\varphi _{33}} \\
&+ {\phi _{11}}{\phi _{22}}{\phi _{33}}{\varphi _{13}}{\varphi _{22}}{\varphi _{31}} + {\phi _{11}}{\phi _{23}}{\phi _{32}}{\varphi _{11}}{\varphi _{22}}{\varphi _{33}} + {\phi _{11}}{\phi _{23}}{\phi _{32}}{\varphi _{12}}{\varphi _{23}}{\varphi _{31}} \\
&+ {\phi _{11}}{\phi _{23}}{\phi _{32}}{\varphi _{13}}{\varphi _{21}}{\varphi _{32}} + {\phi _{12}}{\phi _{21}}{\phi _{33}}{\varphi _{11}}{\varphi _{22}}{\varphi _{33}} + {\phi _{12}}{\phi _{21}}{\phi _{33}}{\varphi _{12}}{\varphi _{23}}{\varphi _{31}} \\
&+ {\phi _{12}}{\phi _{21}}{\phi _{33}}{\varphi _{13}}{\varphi _{21}}{\varphi _{32}} + {\phi _{12}}{\phi _{23}}{\phi _{31}}{\varphi _{11}}{\varphi _{23}}{\varphi _{32}} + {\phi _{12}}{\phi _{23}}{\phi _{31}}{\varphi _{12}}{\varphi _{21}}{\varphi _{33}} \\
&+ {\phi _{12}}{\phi _{23}}{\phi _{31}}{\varphi _{13}}{\varphi _{22}}{\varphi _{31}} + {\phi _{13}}{\phi _{21}}{\phi _{32}}{\varphi _{11}}{\varphi _{23}}{\varphi _{32}} + {\phi _{13}}{\phi _{21}}{\phi _{32}}{\varphi _{12}}{\varphi _{21}}{\varphi _{33}}  \\
&+ {\phi _{13}}{\phi _{21}}{\phi _{32}}{\varphi _{13}}{\varphi _{22}}{\varphi _{31}} + {\phi _{13}}{\phi _{22}}{\phi _{31}}{\varphi _{11}}{\varphi _{22}}{\varphi _{33}} + {\phi _{13}}{\phi _{22}}{\phi _{31}}{\varphi _{12}}{\varphi _{23}}{\varphi _{31}} \\
&+ {\phi _{13}}{\phi _{22}}{\phi _{31}}{\varphi _{13}}{\varphi _{21}}{\varphi _{32}}, \\
{c_{62}} =& {k_1}{k_2}{k_4}{k_5}{\phi _{33}}{\varphi _{33}} + {k_1}{k_2}{k_4}{k_6}{\phi _{32}}{\varphi _{23}} + {k_1}{k_2}{k_5}{k_6}{\phi _{31}}{\varphi _{13}} + {k_1}{k_3}{k_4}{k_5}{\phi _{23}}{\varphi _{32}} \\
&+ {k_1}{k_3}{k_4}{k_6}{\phi _{22}}{\varphi _{22}} + {k_1}{k_3}{k_5}{k_6}{\phi _{21}}{\varphi _{12}} + {k_2}{k_3}{k_4}{k_5}{\phi _{13}}{\varphi _{31}} + {k_2}{k_3}{k_4}{k_6}{\phi _{12}}{\varphi _{21}} \\
&+ {k_2}{k_3}{k_5}{k_6}{\phi _{11}}{\varphi _{11}}, \\
{c_{63}} =& {k_1}{k_4}{\phi _{22}}{\phi _{33}}{\varphi _{22}}{\varphi _{33}} + {k_1}{k_4}{\phi _{23}}{\phi _{32}}{\varphi _{23}}{\varphi _{32}} + {k_1}{k_5}{\phi _{21}}{\phi _{33}}{\varphi _{12}}{\varphi _{33}} \\
&+ {k_1}{k_5}{\phi _{23}}{\phi _{31}}{\varphi _{13}}{\varphi _{32}} + {k_1}{k_6}{\phi _{21}}{\phi _{32}}{\varphi _{12}}{\varphi _{23}} + {k_1}{k_6}{\phi _{22}}{\phi _{31}}{\varphi _{13}}{\varphi _{22}} \\
&+ {k_2}{k_4}{\phi _{12}}{\phi _{33}}{\varphi _{21}}{\varphi _{33}} + {k_2}{k_4}{\phi _{13}}{\phi _{32}}{\varphi _{23}}{\varphi _{31}} + {k_2}{k_5}{\phi _{11}}{\phi _{33}}{\varphi _{11}}{\varphi _{33}} \\
&+ {k_2}{k_5}{\phi _{13}}{\phi _{31}}{\varphi _{13}}{\varphi _{31}} + {k_2}{k_6}{\phi _{11}}{\phi _{32}}{\varphi _{11}}{\varphi _{23}} + {k_2}{k_6}{\phi _{12}}{\phi _{31}}{\varphi _{13}}{\varphi _{21}} \\
&+ {k_3}{k_4}{\phi _{12}}{\phi _{23}}{\varphi _{21}}{\varphi _{32}} + {k_3}{k_4}{\phi _{13}}{\phi _{22}}{\varphi _{22}}{\varphi _{31}} + {k_3}{k_5}{\phi _{11}}{\phi _{23}}{\varphi _{11}}{\varphi _{32}}\\
&+ {k_3}{k_5}{\phi _{13}}{\phi _{21}}{\varphi _{12}}{\varphi _{31}} + {k_3}{k_6}{\phi _{11}}{\phi _{22}}{\varphi _{11}}{\varphi _{22}} + {k_3}{k_6}{\phi _{12}}{\phi _{21}}{\varphi _{12}}{\varphi _{21}}, \\
{c_{64}} =& {k_1}{k_4}{\phi _{22}}{\phi _{33}}{\varphi _{23}}{\varphi _{32}} + {k_1}{k_4}{\phi _{23}}{\phi _{32}}{\varphi _{22}}{\varphi _{33}} + {k_1}{k_5}{\phi _{21}}{\phi _{33}}{\varphi _{13}}{\varphi _{32}} \\
&+ {k_1}{k_5}{\phi _{23}}{\phi _{31}}{\varphi _{12}}{\varphi _{33}} + {k_1}{k_6}{\phi _{21}}{\phi _{32}}{\varphi _{13}}{\varphi _{22}} + {k_1}{k_6}{\phi _{22}}{\phi _{31}}{\varphi _{12}}{\varphi _{23}} \\
&+ {k_2}{k_4}{\phi _{12}}{\phi _{33}}{\varphi _{23}}{\varphi _{31}} + {k_2}{k_4}{\phi _{13}}{\phi _{32}}{\varphi _{21}}{\varphi _{33}} + {k_2}{k_5}{\phi _{11}}{\phi _{33}}{\varphi _{13}}{\varphi _{31}} \\
&+ {k_2}{k_5}{\phi _{13}}{\phi _{31}}{\varphi _{11}}{\varphi _{33}} + {k_2}{k_6}{\phi _{11}}{\phi _{32}}{\varphi _{13}}{\varphi _{21}} + {k_2}{k_6}{\phi _{12}}{\phi _{31}}{\varphi _{11}}{\varphi _{23}} \\
&+ {k_3}{k_4}{\phi _{12}}{\phi _{23}}{\varphi _{22}}{\varphi _{31}} + {k_3}{k_4}{\phi _{13}}{\phi _{22}}{\varphi _{21}}{\varphi _{32}} + {k_3}{k_5}{\phi _{11}}{\phi _{23}}{\varphi _{12}}{\varphi _{31}} \\
&+ {k_3}{k_5}{\phi _{13}}{\phi _{21}}{\varphi _{11}}{\varphi _{32}} + {k_3}{k_6}{\phi _{11}}{\phi _{22}}{\varphi _{12}}{\varphi _{21}} + {k_3}{k_6}{\phi _{12}}{\phi _{21}}{\varphi _{11}}{\varphi _{22}}, \\
{c_{65}} =& {\phi _{11}}{\phi _{22}}{\phi _{33}}{\varphi _{11}}{\varphi _{22}}{\varphi _{33}} + {\phi _{11}}{\phi _{22}}{\phi _{33}}{\varphi _{12}}{\varphi _{23}}{\varphi _{31}} + {\phi _{11}}{\phi _{22}}{\phi _{33}}{\varphi _{13}}{\varphi _{21}}{\varphi _{32}} \\
&+ {\phi _{11}}{\phi _{23}}{\phi _{32}}{\varphi _{11}}{\varphi _{23}}{\varphi _{32}} + {\phi _{11}}{\phi _{23}}{\phi _{32}}{\varphi _{12}}{\varphi _{21}}{\varphi _{33}} + {\phi _{11}}{\phi _{23}}{\phi _{32}}{\varphi _{13}}{\varphi _{22}}{\varphi _{31}} \\
&+ {\phi _{12}}{\phi _{21}}{\phi _{33}}{\varphi _{11}}{\varphi _{23}}{\varphi _{32}} + {\phi _{12}}{\phi _{21}}{\phi _{33}}{\varphi _{12}}{\varphi _{21}}{\varphi _{33}} + {\phi _{12}}{\phi _{21}}{\phi _{33}}{\varphi _{13}}{\varphi _{22}}{\varphi _{31}} \\
&+ {\phi _{12}}{\phi _{23}}{\phi _{31}}{\varphi _{11}}{\varphi _{22}}{\varphi _{33}} + {\phi _{12}}{\phi _{23}}{\phi _{31}}{\varphi _{12}}{\varphi _{23}}{\varphi _{31}} + {\phi _{12}}{\phi _{23}}{\phi _{31}}{\varphi _{13}}{\varphi _{21}}{\varphi _{32}} \\
&+ {\phi _{13}}{\phi _{21}}{\phi _{32}}{\varphi _{11}}{\varphi _{22}}{\varphi _{33}} + {\phi _{13}}{\phi _{21}}{\phi _{32}}{\varphi _{12}}{\varphi _{23}}{\varphi _{31}} + {\phi _{13}}{\phi _{21}}{\phi _{32}}{\varphi _{13}}{\varphi _{21}}{\varphi _{32}} \\
&+ {\phi _{13}}{\phi _{22}}{\phi _{31}}{\varphi _{11}}{\varphi _{23}}{\varphi _{32}} + {\phi _{13}}{\phi _{22}}{\phi _{31}}{\varphi _{12}}{\varphi _{21}}{\varphi _{33}} + {\phi _{13}}{\phi _{22}}{\phi _{31}}{\varphi _{13}}{\varphi _{22}}{\varphi _{31}}.
\end{align*}

\section*{Appendix B}

\begin{align*}
{\Phi _1} =& {c_{11}}{\omega _0}^{5\theta  + 1}\sin \left( {{\omega _0}{\tau _0} - \frac{{5\theta \pi }}{2}} \right) + 2\left( {{c_{21}} - {c_{22}}} \right){\omega _0}^{4\theta  + 1}\sin \left( {2{\omega _0}{\tau _0} - 2\theta \pi } \right) \\
&+ 3{\omega _0}\left( {{c_{31}} - {c_{32}}} \right){\omega _0}^{3\theta  + 1}\sin \left( {3{\omega _0}{\tau _0} - \frac{{3\theta \pi }}{2}} \right) \\
&+ 4\left( {{c_{41}} + {c_{43}} - {c_{42}} - {c_{44}}} \right){\omega _0}^{2\theta  + 1}\sin \left( {4{\omega _0}{\tau _0} - \theta \pi } \right) \\
&+ 5{\omega _0}\left( {{c_{51}} + {c_{53}} - {c_{52}} - {c_{54}}} \right){\omega _0}^{\theta  + 1}\sin \left( {5{\omega _0}{\tau _0} - \frac{{\theta \pi }}{2}} \right) \\
&+ 6{\omega _0}\left( {{c_{61}} + {c_{63}} - {c_{62}} - {c_{64}} - {c_{65}}} \right)\sin( 6{\omega _0}{\tau _0}), \\
{\Phi _2} =& {c_{11}}{\omega _0}^{5\theta  + 1}\cos \left( {{\omega _0}{\tau _0} - \frac{{5\theta \pi }}{2}} \right) + 2\left( {{c_{21}} - {c_{22}}} \right){\omega _0}^{4\theta  + 1}\cos \left( {2{\omega _0}{\tau _0} - 2\theta \pi } \right) \\
&+ 3\left( {{c_{31}} - {c_{32}}} \right){\omega _0}^{3\theta  + 1}\cos \left( {3{\omega _0}{\tau _0} - \frac{{3\theta \pi }}{2}} \right) \\
&+ 4\left( {{c_{41}} + {c_{43}} - {c_{42}} - {c_{44}}} \right){\omega _0}^{2\theta  + 1}\cos \left( {4{\omega _0}{\tau _0} - \theta \pi } \right) \\
&+ 5\left( {{c_{51}} + {c_{53}} - {c_{52}} - {c_{54}}} \right){\omega _0}^{\theta  + 1}\cos \left( {5{\omega _0}{\tau _0} - \frac{{\theta \pi }}{2}} \right) \\
&+ 6\left( {{c_{61}} + {c_{63}} - {c_{62}} - {c_{64}} - {c_{65}}} \right){\omega _0}\cos( 6{\omega _0}{\tau _0}), \\
{\Psi _1} =& 6\theta {\omega _0}^{6\theta  - 1}\cos \frac{{\left( {6\theta  - 1} \right)\pi }}{2} \\
&+ {c_{11}}\cos {\omega _0}{\tau _0}\left( {5\theta {\omega _0}^{5\theta  - 1}\cos \frac{{\left( {5\theta  - 1} \right)\pi }}{2} - {\tau _0}{\omega _0}^{5\theta }\cos \frac{{5\theta \pi }}{2}} \right) \\
&+ {c_{11}}\sin {\omega _0}{\tau _0}\left( {5\theta {\omega _0}^{5\theta  - 1}\sin \frac{{\left( {5\theta  - 1} \right)\pi }}{2} - {\tau _0}{\omega _0}^{5\theta }\sin \frac{{5\theta \pi }}{2}} \right) \\
&+ \left( {{c_{21}} - {c_{22}}} \right)\cos (2{\omega _0}{\tau _0})\left( {4\theta {\omega _0}^{4\theta  - 1}\cos \frac{{\left( {4\theta  - 1} \right)\pi }}{2} - 2{\tau _0}{\omega _0}^{4\theta }\cos (2\theta \pi) } \right) \\
&+ \left( {{c_{21}} - {c_{22}}} \right)\sin (2{\omega _0}{\tau _0})\left( {4\theta {\omega _0}^{4\theta  - 1}\sin \frac{{\left( {4\theta  - 1} \right)\pi }}{2} - 2{\tau _0}{\omega _0}^{4\theta }\sin (2\theta \pi) } \right) \\
&+ \left( {{c_{31}} - {c_{32}}} \right)\cos (3{\omega _0}{\tau _0})\left( {3\theta {\omega _0}^{3\theta  - 1}\cos \frac{{\left( {3\theta  - 1} \right)\pi }}{2} - 3{\tau _0}{\omega _0}^{3\theta }\cos \frac{{3\theta \pi }}{2}} \right) \\
&+ \left( {{c_{31}} - {c_{32}}} \right)\sin (3{\omega _0}{\tau _0})\left( {3\theta {\omega _0}^{3\theta  - 1}\sin \frac{{\left( {3\theta  - 1} \right)\pi }}{2} - 3{\tau _0}{\omega _0}^{3\theta }\sin \frac{{3\theta \pi }}{2}} \right) \\
&+ \left( {{c_{41}} + {c_{43}} - {c_{42}} - {c_{44}}} \right)\cos (4{\omega _0}{\tau _0})\left( {2\theta {\omega _0}^{2\theta  - 1}\cos \frac{{\left( {2\theta  - 1} \right)\pi }}{2} - 4{\tau _0}{\omega _0}^{2\theta }\cos \theta \pi } \right) \\
&+ \left( {{c_{41}} + {c_{43}} - {c_{42}} - {c_{44}}} \right)\sin (4{\omega _0}{\tau _0})\left( {2\theta {\omega _0}^{2\theta  - 1}\sin \frac{{\left( {2\theta  - 1} \right)\pi }}{2} - 4{\tau _0}{\omega _0}^{2\theta }\sin \theta \pi } \right) \\
&+ \left( {{c_{51}} + {c_{53}} - {c_{52}} - {c_{54}}} \right)\cos (5{\omega _0}{\tau _0})\left( {\theta {\omega _0}^{\theta  - 1}\cos \frac{{\left( {\theta  - 1} \right)\pi }}{2} - 5{\tau _0}{\omega _0}^\theta \cos \frac{{\theta \pi }}{2}} \right) \\
&+ \left( {{c_{51}} + {c_{53}} - {c_{52}} - {c_{54}}} \right)\sin (5{\omega _0}{\tau _0})\left( {\theta {\omega _0}^{\theta  - 1}\sin \frac{{\left( {\theta  - 1} \right)\pi }}{2} - 5{\tau _0}{\omega _0}^\theta \sin \frac{{\theta \pi }}{2}} \right) \\
&- 6\left( {{c_{61}} + {c_{63}} - {c_{62}} - {c_{64}} - {c_{65}}} \right){\tau _0}\cos (6{\omega _0}{\tau _0}),
\end{align*}

\begin{align*}
{\Psi _2} =& 6\theta {\omega _0}^{6\theta  - 1}\sin \frac{{\left( {6\theta  - 1} \right)\pi }}{2} \\
&+ {c_{11}}\cos {\omega _0}{\tau _0}\left( {5\theta {\omega _0}^{5\theta  - 1}\sin \frac{{\left( {5\theta  - 1} \right)\pi }}{2} - {\tau _0}{\omega _0}^{5\theta }\sin \frac{{5\theta \pi }}{2}} \right) \\
&- {c_{11}}\sin {\omega _0}{\tau _0}\left( {5\theta {\omega _0}^{5\theta  - 1}\cos \frac{{\left( {5\theta  - 1} \right)\pi }}{2} - {\tau _0}{\omega _0}^{5\theta }\cos \frac{{5\theta \pi }}{2}} \right) \\
&+ \left( {{c_{21}} - {c_{22}}} \right)\cos (2{\omega _0}{\tau _0})\left( {4\theta {\omega _0}^{4\theta  - 1}\sin \frac{{\left( {4\theta  - 1} \right)\pi }}{2} - 2{\tau _0}{\omega _0}^{4\theta }\sin (2\theta \pi) } \right) \\
&- \left( {{c_{21}} - {c_{22}}} \right)\sin (2{\omega _0}{\tau _0})\left( {4\theta {\omega _0}^{4\theta  - 1}\cos \frac{{\left( {4\theta  - 1} \right)\pi }}{2} - 2{\tau _0}{\omega _0}^{4\theta }\cos (2\theta \pi) } \right) \\
&+ \left( {{c_{31}} - {c_{32}}} \right)\cos (3{\omega _0}{\tau _0})\left( {3\theta {\omega _0}^{3\theta  - 1}\sin \frac{{\left( {3\theta  - 1} \right)\pi }}{2} - 3{\tau _0}{\omega _0}^{3\theta }\sin \frac{{3\theta \pi }}{2}} \right) \\
&- \left( {{c_{31}} - {c_{32}}} \right)\sin (3{\omega _0}{\tau _0})\left( {3\theta {\omega _0}^{3\theta  - 1}\cos \frac{{\left( {3\theta  - 1} \right)\pi }}{2} - 3{\tau _0}{\omega _0}^{3\theta }\cos \frac{{3\theta \pi }}{2}} \right) \\
&+ \left( {{c_{41}} + {c_{43}} - {c_{42}} - {c_{44}}} \right)\cos (4{\omega _0}{\tau _0})\left( {2\theta {\omega _0}^{2\theta  - 1}\sin \frac{{\left( {2\theta  - 1} \right)\pi }}{2} - 4{\tau _0}{\omega _0}^{2\theta }\sin \theta \pi } \right) \\
&- \left( {{c_{41}} + {c_{43}} - {c_{42}} - {c_{44}}} \right)\sin (4{\omega _0}{\tau _0})\left( {2\theta {\omega _0}^{2\theta  - 1}\cos \frac{{\left( {2\theta  - 1} \right)\pi }}{2} - 4{\tau _0}{\omega _0}^{2\theta }\cos \theta \pi } \right) \\
&+ \left( {{c_{51}} + {c_{53}} - {c_{52}} - {c_{54}}} \right)\cos (5{\omega _0}{\tau _0})\left( {\theta {\omega _0}^{\theta  - 1}\sin \frac{{\left( {\theta  - 1} \right)\pi }}{2} - 5{\tau _0}{\omega _0}^\theta \sin \frac{{\theta \pi }}{2}} \right) \\
&- \left( {{c_{51}} + {c_{53}} - {c_{52}} - {c_{54}}} \right)\sin (5{\omega _0}{\tau _0})\left( {\theta {\omega _0}^{\theta  - 1}\cos \frac{{\left( {\theta  - 1} \right)\pi }}{2} - 5{\tau _0}{\omega _0}^\theta \cos \frac{{\theta \pi }}{2}} \right) \\
&+ 6\left( {{c_{61}} + {c_{63}} - {c_{62}} - {c_{64}} - {c_{65}}} \right){\tau _0}\sin (6{\omega _0}{\tau _0}).
\end{align*}

\section*{Appendix C}

\begin{align*}
{\Theta _1} =& 4\omega _0^*{\mathop{\rm Re}\nolimits} \left[ {{p_1}(\omega _0^*i)} \right]\sin (4\omega _0^*{\tau _4}) - 4{\omega _0}{\mathop{\rm Im}\nolimits} \left[ {{p_1}(\omega _0^*i)} \right]\cos (4\omega _0^*{\tau _4}) \\
&+ 4{\omega _0}{\mathop{\rm Re}\nolimits} \left[ {{p_3}(\omega _0^*i)} \right]\sin (4\omega _0^*{\tau _4}) + 4\omega _0^*{\mathop{\rm Im}\nolimits} \left[ {{p_3}(\omega _0^*i)} \right]\cos (4\omega _0^*{\tau _4}) + 8\omega _0^*{p_4}\sin (8\omega _0^*{\tau _4}), \\
{\Theta _2} =& 4\omega _0^*{\mathop{\rm Re}\nolimits} \left[ {{p_1}(\omega _0^*i)} \right]\cos (4\omega _0^*{\tau _4}) + 4\omega _0^*{\mathop{\rm Im}\nolimits} \left[ {{p_1}(\omega _0^*i)} \right]\sin (4\omega _0^*{\tau _4}) \\
&- 4\omega _0^*{\mathop{\rm Re}\nolimits} \left[ {{p_3}(\omega _0^*i)} \right]\cos (4\omega _0^*{\tau _4}) + 4\omega _0^*{\mathop{\rm Im}\nolimits} \left[ {{p_3}(\omega _0^*i)} \right]\sin (4\omega _0^*{\tau _4}) - 8\omega _0^*{p_4}\cos (8\omega _0^*{\tau _4}). \\
{\Upsilon _1} =& 6\theta {\omega_0^*}^{6\theta  - 1}\cos \left( {6{\omega_0^*}{\tau _1} - 4{\omega_0^*}{\tau _4} + \frac{{\left( {6\theta  - 1} \right)\pi }}{2}} \right) \\
&+ 5{c_{11}}\theta {\omega_0^*}^{5\theta  - 1}\cos \left( {5{\omega_0^*}{\tau _1} - 4{\omega_0^*}{\tau _4} + \frac{{\left( {5\theta  - 1} \right)\pi }}{2}} \right) \\
&+ 4{c_{21}}\theta {\omega_0^*}^{4\theta  - 1}\cos \left( {4{\omega_0^*}{\tau _1} - 4{\omega_0^*}{\tau _4} + \frac{{\left( {4\theta  - 1} \right)\pi }}{2}} \right) \\
&+ 3{c_{31}}\theta {\omega_0^*}^{3\theta  - 1}\cos \left( {3{\omega_0^*}{\tau _1} - 4{\omega_0^*}{\tau _4} + \frac{{\left( {3\theta  - 1} \right)\pi }}{2}} \right)
\end{align*}

\begin{align*}
&+ 2{c_{41}}\theta {\omega_0^*}^{2\theta  - 1}\cos \left( {2{\omega_0^*}{\tau _1} - 4{\omega_0^*}{\tau _4} + \frac{{\left( {2\theta  - 1} \right)\pi }}{2}} \right) \\
&+ {c_{51}}\theta {\omega_0^*}^{\theta  - 1}\cos \left( {{\omega_0^*}{\tau _1} - 4{\omega_0^*}{\tau _4} + \frac{{\left( {\theta  - 1} \right)\pi }}{2}} \right) \\
&+ 4{\tau _4}{\mathop{\rm Re}\nolimits} \left[ {{p_3}({\omega_0^*}i)} \right]\cos (4{\omega_0^*}{\tau _4}) - 4{\tau _4}{\mathop{\rm Im}\nolimits} \left[ {{p_3}({\omega_0^*}i)} \right]\sin (4{\omega_0^*}{\tau _4}) \\
&+ 2\theta \left( {{c_{43}} - {c_{44}}} \right){\omega_0^*}^{2\theta  - 1}\cos \left( {2{\omega_0^*}{\tau _1} + 4{\omega_0^*}{\tau _4} + \frac{{\left( {2\theta  - 1} \right)\pi }}{2}} \right) \\
&+ \theta \left( {{c_{53}} - {c_{54}}} \right){\omega_0^*}^{\theta  - 1}\cos \left( {{\omega_0^*}{\tau _1} + 4{\omega_0^*}{\tau _4} + \frac{{\left( {\theta  - 1} \right)\pi }}{2}} \right) \\
&+ 8{p_4}{\tau _4}\cos (8{\omega_0^*}{\tau _4}) - 4{\tau _4}{\mathop{\rm Re}\nolimits} \left[ {{p_1}({\omega_0^*}i)} \right]\cos (4{\omega_0^*}{\tau _4}) + 4{\tau _4}{\mathop{\rm Im}\nolimits} \left[ {{p_1}({\omega_0^*}i)} \right]\sin (4{\omega_0^*}{\tau _4}) \\
&- 4{c_{22}}\theta {\omega_0^*}^{4\theta  - 1}\cos \left( {4{\omega_0^*}{\tau _1} + \frac{{\left( {4\theta  - 1} \right)\pi }}{2}} \right) \\
&- 3{c_{32}}\theta {\omega_0^*}^{3\theta  - 1}\cos \left( {3{\omega_0^*}{\tau _1} + \frac{{\left( {3\theta  - 1} \right)\pi }}{2}} \right) \\
&- 2{c_{42}}\theta {\omega_0^*}^{2\theta  - 1}\cos \left( {2{\omega_0^*}{\tau _1} + \frac{{\left( {2\theta  - 1} \right)\pi }}{2}} \right) - {c_{52}}\theta {\omega_0^*}^{\theta  - 1}\cos \left( {{\omega_0^*}{\tau _1} + \frac{{\left( {\theta  - 1} \right)\pi }}{2}} \right), \\
{\Upsilon _2} =& 6\theta {\omega_0^*}^{6\theta  - 1}\sin \left( {6{\omega_0^*}{\tau _1} - 4{\omega_0^*}{\tau _4} + \frac{{\left( {6\theta  - 1} \right)\pi }}{2}} \right) \\
&+ 5{c_{11}}\theta {\omega_0^*}^{5\theta  - 1}\sin \left( {5{\omega_0^*}{\tau _1} - 4{\omega_0^*}{\tau _4} + \frac{{\left( {5\theta  - 1} \right)\pi }}{2}} \right) \\
&+ 4{c_{21}}\theta {\omega_0^*}^{4\theta  - 1}\sin \left( {4{\omega_0^*}{\tau _1} - 4{\omega_0^*}{\tau _4} + \frac{{\left( {4\theta  - 1} \right)\pi }}{2}} \right) \\
&+ 3{c_{31}}\theta {\omega_0^*}^{3\theta  - 1}\sin \left( {3{\omega_0^*}{\tau _1} - 4{\omega_0^*}{\tau _4} + \frac{{\left( {3\theta  - 1} \right)\pi }}{2}} \right) \\
&+ 2{c_{41}}\theta {\omega_0^*}^{2\theta  - 1}\sin \left( {2{\omega_0^*}{\tau _1} - 4{\omega_0^*}{\tau _4} + \frac{{\left( {2\theta  - 1} \right)\pi }}{2}} \right) \\
&+ {c_{51}}\theta {\omega_0^*}^{\theta  - 1}\sin \left( {{\omega_0^*}{\tau _1} - 4{\omega_0^*}{\tau _4} + \frac{{\left( {\theta  - 1} \right)\pi }}{2}} \right) \\
&+ 4{\tau _4}{\mathop{\rm Re}\nolimits} \left[ {{p_3}({\omega_0^*}i)} \right]\sin (4{\omega_0^*}{\tau _4}) + 4{\tau _4}{\mathop{\rm Im}\nolimits} \left[ {{p_3}({\omega_0^*}i)} \right]\cos (4{\omega_0^*}{\tau _4}) \\
&+ 2\theta \left( {{c_{43}} - {c_{44}}} \right){\omega_0^*}^{2\theta  - 1}\sin \left( {2{\omega_0^*}{\tau _1} + 4{\omega_0^*}{\tau _4} + \frac{{\left( {2\theta  - 1} \right)\pi }}{2}} \right) \\
&+ \theta \left( {{c_{53}} - {c_{54}}} \right){\omega_0^*}^{\theta  - 1}\sin \left( {{\omega_0^*}{\tau _1} + 4{\omega_0^*}{\tau _4} + \frac{{\left( {\theta  - 1} \right)\pi }}{2}} \right) \\
&+ 8{p_4}{\tau _4}\sin (8{\omega_0^*}{\tau _4}) - 4{\tau _4}{\mathop{\rm Re}\nolimits} \left[ {{p_1}({\omega_0^*}i)} \right]\sin (4{\omega_0^*}{\tau _4}) - 4{\tau _4}{\mathop{\rm Im}\nolimits} \left[ {{p_1}({\omega_0^*}i)} \right]\cos (4{\omega_0^*}{\tau _4}) \\
&- 4{c_{22}}\theta {\omega_0^*}^{4\theta  - 1}\sin \left( {4{\omega_0^*}{\tau _1} + \frac{{\left( {4\theta  - 1} \right)\pi }}{2}} \right) - 3{c_{32}}\theta {\omega_0^*}^{3\theta  - 1}\sin \left( {3{\omega_0^*}{\tau _1} + \frac{{\left( {3\theta  - 1} \right)\pi }}{2}} \right) \\
&- 2{c_{42}}\theta {\omega_0^*}^{2\theta  - 1}\sin \left( {2{\omega_0^*}{\tau _1} + \frac{{\left( {2\theta  - 1} \right)\pi }}{2}} \right) - {c_{52}}\theta {\omega_0^*}^{\theta  - 1}\sin \left( {{\omega_0^*}{\tau _1} + \frac{{\left( {\theta  - 1} \right)\pi }}{2}} \right).
\end{align*}

\end{document}